\documentclass[11pt]{article}
\usepackage[a4paper]{geometry}
\usepackage[utf8]{inputenc}
\usepackage[T1]{fontenc}
\usepackage{mathtools,amsthm,amssymb}
\usepackage[colorlinks=true]{hyperref}
\hypersetup{citecolor=blue}
\usepackage{enumitem}
\setlist[enumerate]{label=\textnormal{(\roman*)}}

\newtheorem{theorem}{Theorem}
\newtheorem{lemma}{Lemma}
\newtheorem*{proposition}{Proposition}
\theoremstyle{remark}
\newtheorem*{remark}{Remark}

\newcommand*\CC{\mathbb{C}}
\newcommand*\RR{\mathbb{R}}
\newcommand*\DD{\RR^2_+}

\newcommand*\px{\partial_x}
\newcommand*\py{\partial_y}
\newcommand*\pt{\partial_t}
\newcommand*\ps{\partial_s}
\newcommand*\PO{\partial_\omega}

\newcommand*\eu{e_1}
\newcommand*\ii{\textnormal{i}}
\newcommand*\ee{\textnormal{e}}
\newcommand*\TR{t}

\newcommand*\oz{\omega_0}
\newcommand*\po{\phi_\omega}
\newcommand*\Qo{Q_\omega}
\newcommand*\fo{f_\omega}
\newcommand*\FO{F_\omega}

\newcommand*\poz{\phi_{\oz}}
\newcommand*\Qoz{Q_{\oz}}
\newcommand*\Lo{\Lambda_\omega}
\newcommand*\Pz{P_0}
\newcommand*\Po{P_\omega}
\newcommand*\Yo{A_\omega}
\newcommand*\Yz{A_0}
\newcommand*\Ha{H_\alpha}

\newcommand*\Ka{K_{\alpha,\omega}}
\newcommand*\Ma{M_{\alpha,\omega}}
\newcommand*\Na{N_{\alpha}}
\newcommand*\La{L_{\alpha,\omega}}
\newcommand*\LP{L_+}
\newcommand*\LM{L_-}
\newcommand*\MP{M_+}
\newcommand*\MM{M_-}
\newcommand*\QP{Q_+}
\newcommand*\QM{Q_-}
\newcommand*\PP{P_+}
\newcommand*\PM{P_-}
\newcommand*\rr{r}
\newcommand*\af{w_0}

\newcommand*\cC{\mathcal{C}}
\newcommand*\cV{\mathcal{V}}
\newcommand*\cT{\mathcal{T}}
\newcommand*\cM{\mathcal{M}}
\newcommand*\cO{\mathcal{O}_2}
\newcommand*\cOO{\mathcal{O}(\omega)}
\newcommand*\cK{K}
\newcommand*\mS{\mathcal{S}}
\newcommand*\cS{S}

\newcommand*\cU{U}
\newcommand*\pp{Y_0}

\newcommand*\bI{\mathbf{I}}
\newcommand*\bi{\mathbf{i}}
\newcommand*\bJ{\mathbf{J}}
\newcommand*\bK{\mathbf{K}}
\newcommand*\bL{\mathbf{L}}
\newcommand*\bR{\mathbf{R}}

\newcommand*\bP{\mathbf{P}}
\newcommand*\bZ{\mathbf{Z}}

\newcommand*\Xt{X_\theta}
\newcommand*\Xu{X_\vartheta}

\DeclareMathOperator\sech{sech}
\DeclareMathOperator\sgn{sgn}

\begin{document}
\title{Asymptotic stability of small standing solitary waves 
of the one-dimensional cubic-quintic Schrödinger equation}

\author{Yvan Martel}
\date{\small Laboratoire de mathématiques de Versailles,
UVSQ, Université Paris-Saclay, CNRS,\\
45 avenue des États-Unis,
78035 Versailles Cedex, France}

\maketitle

\begin{abstract}
For the Schrödinger equation with a cubic-quintic,
focusing-focusing nonlinearity
in one space dimension, this article proves 
the local asymptotic completeness of the family of
small standing solitary waves under even perturbations in the energy space.
For this model, perturbative of the integrable cubic Schrödinger equation for small solutions,
the linearized equation around a small solitary wave has an
internal mode, whose contribution to the dynamics is handled by
the Fermi golden rule.
\end{abstract}

\section{Introduction}

\subsection{Main result}
We consider the one-dimensional Schrödinger equation with a double power,
focusing cubic and focusing quintic, nonlinearity
\begin{equation}\label{eq:NL}
\begin{cases}
\ii\pt\psi +\px^2\psi +|\psi|^2\psi +|\psi|^4\psi = 0\quad &(t,x)\in\RR\times\RR,\\
\psi(0)=\psi_0 \quad &x\in \RR.
\end{cases}
\end{equation}
The Cauchy problem \eqref{eq:NL} is locally well-posed in the space $H^1(\RR)$ (see \cite{CaBK}).
Moreover, for any solution $\psi$ in $H^1(\RR)$,
the mass, momentum and energy
\[\int|\psi|^2 ,\quad 
\Im\int\psi\py\bar\psi ,\quad
\int\Bigl(\frac 12|\px\psi|^2 -\frac 14|\psi|^4 -\frac 16|\psi|^6\Bigr)
\]
are conserved, as long as $\psi$ exists.
We recall the invariances
by Galilean transform, translation and phase:
if $\psi$ is a solution of~\eqref{eq:NL} then, for any
$\beta,\sigma,\gamma\in\RR$, the function
$
\zeta(t,x)= \ee^{\ii(\beta x-\beta^2 t+\gamma)}\psi(t,x-2\beta t-\sigma)
$
is also a solution of~\eqref{eq:NL}.
For any~$\omega>0$, there exists a unique even positive
solution~$\po\in H^1(\RR)$ of the equation
\[
\po''-\omega\po +\po^3 +\po^5 = 0,\quad x\in\RR,
\]
given (see~\cite{Oh95} and~\cite{PeKA}) by
$\po(x) =\sqrt{\omega}\Qo(\sqrt{\omega}x)$
where the function $\Qo$,
a solution of the equation $\Qo'' -\Qo +\Qo^3 +\omega\Qo ^5=0$, is defined by
\begin{equation}\label{eq:Qo}
\Qo(y) =\sqrt{\frac {4}{1+a_\omega\cosh 2y}}
\quad\mbox{with}\quad
a_\omega =\sqrt{1+\frac{16}3\omega}\,.
\end{equation}
Then, for any $\gamma\in \RR$, 
the function~$\psi(t,x)= \ee^{\ii\gamma}\ee^{\ii\omega t}\po(x)$ is
a standing wave solution of~\eqref{eq:NL}.
We recall the result of \emph{orbital stability} from~\cite{Oh95},
in the special case of even initial data,
and we refer to~\cite{CaLi,We86} for previous related works.

\begin{proposition}[{\cite[Theorem~1]{Oh95}}]
For all~$\omega_0>0$ and $\varepsilon>0$, there exists~$\delta>0$
such that for any even function~$\psi_0\in H^1(\RR)$
with $\|\psi_0 -\poz\|_{H^1(\RR)} <\delta$,
the solution $\psi$ of~\eqref{eq:NL} is globally defined and
satisfies
\[
\sup_{t\in\RR}\inf_{\gamma\in\RR}
\|\ee^{-\ii\gamma}\psi(t) - \poz\|_{H^1(\RR)}<\varepsilon.
\]
\end{proposition}

In the framework of the stability result, the main result of this article is the \emph{asymptotic stability} of the 
family of small standing waves of \eqref{eq:NL}, under even perturbations in the energy space.

\begin{theorem}\label{TH:as}
For all $\oz>0$ sufficiently small,
there exists $\delta>0$ such that for any even function
$\psi_0\in H^1(\RR)$ with $\|\psi_0-\phi_{\oz}\|_{H^1(\RR)}<\delta$,
there exist~$\omega_+>0$ and
a $\cC^1$ function $\gamma:[0,+\infty)\to \RR$ 
with $\lim_{+\infty} \gamma'=\omega_+$ such that
the solution $\psi$ of \eqref{eq:NL}
satisfies
\[
\lim_{t\to +\infty} \ee^{-\ii \gamma(t)} \psi(t) =\phi_{\omega_+}
\quad \mbox{uniformly on compact sets of~$\RR$.}
\]
\end{theorem}

\begin{remark}
The asymptotic stability result means that any even solution close in~$H^1(\RR)$ to a standing wave converges in large time
to a \emph{final} ground state $\phi_{\omega_+}$, locally in space
and up to a phase.
By the stability statement, $\omega_+$ is close to
$\oz$.
We point out that the symmetry assumption in Theorem~\ref{TH:as} is technical, in the sense that it simplifies the proof, but 
we expect no deep additional difficulty in the non symmetric case.
See the remark after Lemma \ref{LE:un}.
\end{remark}

\begin{remark}
For small standing waves and symmetric initial data,
Theorem \ref{TH:as} is identical to the main result in \cite{Ma22} 
concerning the equation
\begin{equation}\label{eq:FD}
\ii\pt\psi +\px^2\psi +|\psi|^2\psi -|\psi|^4\psi = 0
\end{equation}
with a \emph{focusing-defocusing} double power nonlinearity.
The proof of Theorem~\ref{TH:as} is partly inspired by~\cite{Ma22}, 
which extends to the nonlinear Schrödinger equation the
strategy initiated in~\cite{KMM1} for the nonlinear Klein-Gordon equation.
However, the existence of an \emph{internal mode} for \eqref{eq:NL}
drastically complicates the analysis
compared to~\eqref{eq:FD}. 
We refer to~\S\ref{sb:13} for the notion of internal mode, first discussed in \cite{PeKA} for 
both \eqref{eq:NL} and \eqref{eq:FD}.
In the present paper, the technique to deal with the internal 
mode is inspired by \cite{KoMa,KMM1}.
Other references related to Theorem \ref{TH:as} are given in~\S\ref{sb:12}.
\end{remark}
\begin{remark}
For a solution $\psi$ of \eqref{eq:NL} or \eqref{eq:FD},
by changing variables
\[
\psi(t,x) = \sqrt{\oz} \zeta(s,y), \quad
s = \oz t, ~ x=\sqrt{\oz} y,
\]
one obtains a solution $\zeta$ of the equation
$\ii\ps\zeta +\py^2\zeta +|\zeta|^2\zeta \pm\oz |\zeta|^4\zeta = 0$.
This means that the study of small solutions of \eqref{eq:NL} or
\eqref{eq:FD} is perturbative of 
the focusing cubic Schrödinger equation
\begin{equation}\label{eq:CC}
\ii\pt\psi +\px^2\psi +|\psi|^2\psi= 0.
\end{equation}
As pointed out in \cite{Ma22}, the family of $2$-solitons constructed by the inverse scattering transform in~\cite{Ol87,ZaSh} 
provides \emph{counter-examples} to the asymptotic stability of solitons for the integrable model~\eqref{eq:CC} for perturbations in the energy space,
even in the weak sense of Theorem~\ref{TH:as}.
However, \cite{CuPe} proves that the asymptotic stability of solitons of~\eqref{eq:CC} holds true in \emph{weighted spaces}.
In the present article, it is strongly used that the problem~\eqref{eq:NL} is perturbative of the integrable case, not because of actually using
any of the integrability properties of~\eqref{eq:CC}, but because of the remarquable property
of the linearised operator.
Indeed, after factorisation, the linearised operator for the integrable case
becomes simple and easy to perturb.
The proof of the asymptotic stability property for \eqref{eq:NL} and~\eqref{eq:FD} is based on the idea of computing the small discrepancy between
the integrable equation~\eqref{eq:CC} and close non integrable models.
Apart from convenient algebraic properties, the proof does
not rely on the fact that the perturbation is quintic. 
Indeed, for most perturbations, the resonance of the integrable case, which is considered as the major spectral difficulty, either disappears or bifurcates to a manageable internal mode.
We conjecture that for small generic perturbations~$g$, asymptotic stability 
of solitary waves holds for the semilinear model
\begin{equation}\label{eq:PE}
\ii\pt\psi +\px^2\psi +|\psi|^2\psi + g(|\psi|^2) \psi= 0.
\end{equation}
As a first evidence, a recent work~\cite{Ri23} extends the main result of \cite{Ma22} for \eqref{eq:FD}
to the general model~\eqref{eq:PE} for a wide range of \emph{negative} (in some sense) perturbations $g$
 and includes a proof of non existence of internal mode.
\end{remark}

\begin{remark}
Turning back to perturbations in the energy space of solitary waves of
\eqref{eq:NL}, we justify that convergence for the supremum norm \emph{on compact sets of $\RR$},
as stated in Theorem~\ref{TH:as}, is optimal.
For any $0<\omega<\omega_0$ and $\beta>0$, there exists an even solution $\psi$ of
\eqref{eq:NL} with the asymptotic behavior
\[
\lim_{t\to+\infty}\|\psi(t) - (q_0 + q_+ +q_-)(t)\|_{H^1(\RR)}= 0,
\]
where
$
q_0(t,x)= \ee^{\ii\omega_0 t }\phi_{\omega_0}(x)$ and 
$q_\pm(t,x)= \ee^{\ii(\pm\beta x-\beta^2 t +\omega t)}\phi_{\omega}(x\mp2\beta t).
$
Such a solution may be called a $3$-soliton or more accurately,
since the equation \eqref{eq:NL} is not completely integrable, an \emph{asymptotic} $3$-solitary wave. We refer to
\cite{MaMe} for the construction of such solutions
for general, non integrable, nonlinear Schrödinger equations with stable solitary waves.
Taking $0<\omega\ll \omega_0$, the solitary waves $q_+$ and $q_-$ are 
arbitrarily small in $H^1$ norm compared to $q_0$. Therefore,
the existence of the solution~$\psi$ shows that the solitary wave $q_0$ is \emph{not} asymptotically stable for the supremum norm on the whole~$\RR$,
for small perturbations in the energy space.
In the literature (see references in~\S\ref{sb:12}), stronger notions of asymptotic stability
are often considered, 
and explicit decay rates are obtained.
However, such results hold for small perturbations of the initial data 
in suitable \emph{weighted} spaces.
Small solitons like $q_\pm$ do belong to such weighted spaces but they 
have large norms in such spaces, and so they are not acceptable perturbations.
Working in weighted spaces thus provides
more precise asymptotic results and allows to deal with the integrable case
(by different techniques)
while working in the energy space allows the presence of small solitary waves and 
to highlight some specificities of the integrable case.
\end{remark}

\subsection{Related articles}\label{sb:12}

\emph{Classical references.}
The motivation for considering the one-dimensional cubic-quintic Schrödinger models
\eqref{eq:NL} and~\eqref{eq:FD} comes from
several pioneering articles
 published in the Nineties
on the asymptotic stability of 
solitary waves. 
We mention \cite{BuP1,BuP0,SoW1,SoW2} in the absence of internal mode and 
\cite{BuP2,BuP3,BuSu,PeKA,Si92,SoW3} in the presence of internal mode,
with the emergence of the fundamental notion of \emph{nonlinear Fermi golden rule} related to the damping of the internal mode component.
The spectral properties of the models~\eqref{eq:NL} and~\eqref{eq:FD} are studied in~\cite{PeKA}, while 
the survey \cite{KiMa} describes other relevant
models perturbative of~\eqref{eq:CC}.
Inspired by \cite{PeKA}, we have chosen to consider the equations~\eqref{eq:NL} and~\eqref{eq:FD} to provide explicit examples of Schrödinger models for which 
the asymptotic stability of solitary waves could be proved
in the \emph{one-dimensional space}, with
\emph{low nonlinearities}, without or \emph{with internal mode},
 which are well-known difficulties.

\medskip
\noindent\emph{Closely related articles.}
The proof of Theorem \ref{TH:as} relies on 
virial techniques developed for one-dimensional wave-type equations, such as the $\phi^4$ model in \cite{KMM1}, 
the nonlinear Klein-Gordon equation in \cite{KMM2} and general
scalar fields models \cite{KoMa,KMMV}.
Before being used for wave equations, 
localized virial arguments were introduced to study 
blowup and asymptotic stability of solitons for some nonlinear dispersive equations, like
the generalized Korteweg-de Vries equation \cite{Ma06,MaM0,MaM2} and
the mass critical nonlinear Schrödinger equation \cite{MeRa}.
In \cite{MaM0,MaM2,MeRa},  spectral properties related to the 
virial estimate were checked numerically. Then, in 
\cite{Ma06}, a transformed problem was introduced to avoid the use of numerics
for gKdV with power nonlinearities.
Later, extending this technique, a proof of asymptotic stability of solitons 
for general nonlinearities was given  in \cite{MaM3}. 

The specific strategy of using a transformed problem 
and two virial arguments was introduced in~\cite{KMM2}
and then extended to the nonlinear Schrödinger equation 
\eqref{eq:FD} in~\cite{Ma22}.
In the present article, we also extend to Schrödinger models an argument of \cite{KoMa,KMM1} to treat the presence of an internal mode.
As long as dynamical arguments are concerned, the present paper is thus mainly based on generalisations of~\cite{KoMa,KMM1,KMM2,Ma22}.
However, as shown in \cite{PeKA}, the spectral theory for the linearisation of \eqref{eq:NL} around a solitary wave is non trivial, and 
the internal mode is not explicit,
as it is the case for the $\phi^4$ equation, for example.
Thus, specific arguments from the perturbative spectral theory
are to be involved.
Here, we use the theory developed in \cite{Me02} for vectorial
spectral problems, extending arguments from~\cite{Si76} in the scalar case.
To use such perturbative arguments, it is essential to work on the 
transformed problem, as explained in \S\ref{sb:13}.
Another approach to the spectral theory is given in \cite{CoGu}, for near cubic pure power nonlinear Schrödinger equations,
but the higher flexibility of the method developed in \cite{Me02,Si76} allows us to 
compute the asymptotic expansion of the internal mode close the 
integrable case, which is needed to check explicitly the Fermi golden rule as well as the \emph{repulsive nature} of the operator
appearing after the second transformation;
see \S\ref{sb:13}, \S\ref{se:03} and \S\ref{se:06}.

\medskip

\noindent\emph{Other related works.}
The literature on asymptotic stability is abondant.
For wave-type equations, we refer to \cite{CuM4,GePu,LiLu,LLSS,LiLS,MaMu}, which contain some of the most advanced results in different directions.
Restricting now to Schrödinger-type models, we quote a few surveys~\cite{CuM3,CuM1,KMM3,Sc07}
and some of the most recent articles in various settings
\cite{ChPu,GuNT,GNT1,KrSc,Mi08}.
We point out the result in one dimension recently obtained in \cite{CoGe},
proving full asymptotic stability, that is
convergence to a final standing wave in the supremum norm on the whole $\RR$, with a decay rate,
under mild assumptions on the initial data, and assuming only the non existence of internal mode
and resonance.
Some other articles \cite{Ch21,CuM2, GuNT, Na16} concern nonlinear Schrödinger equations with a potential.

\subsection{Outline of the proof}\label{sb:13}

\emph{Modulation of the solitary wave \S\ref{se:04}.}
Let $\oz>0$ be sufficiently small and let $\psi(t,x)$ be a global solution
of \eqref{eq:NL} close to $\phi_{\oz}$ for all $t\geq 0$.
We define $u(s,y)=u_1 + \ii u_2$ by
\[
\psi(t,x)
=\exp (\ii\gamma(s) ) \sqrt{\omega(s)} \left( Q_{\omega(s)}(y) + u_1(s,y) + \ii u_2(s,y)\right)
\]
where $s$ and $y$ are the rescaled time and space variables, respectively
defined by
\[
dt = \frac{ds}{\omega(s)},\quad x = \frac{y}{\sqrt{\omega(s)}}.
\]
The time dependent $\cC^1$ functions $\gamma\in\RR$ and $\omega>0$ are adjusted for all $s\geq 0$ so that
the functions $u_1$ and $u_2$ are orthogonal to directions related to the phase invariance of the equation and to the continuum of solitary waves 
$\omega\mapsto\Qo$ defined by \eqref{eq:Qo}.

\medskip
\noindent\emph{Linearised system.}
The second order differential operators $\LP$ and $\LM$,
related to the linearization of \eqref{eq:NL} around $\Qo$, are defined
at the beginning of \S\ref{se:02}.
In the $(s,y)$ variables, the coupled system for $(u_1,u_2)$ is
\[
\begin{cases}
\dot u_1 =\LM u_2 +\mu_2 + p_2 - q_2\\
\dot u_2 = -\LP u_1 -\mu_1 - p_1 + q_1
\end{cases}
\]
where for $k=1,2$, $\mu_k$ are modulation terms coming from the time dependency of the functions $\omega$ and $\gamma$, $p_k$ are other modulation terms
of quadratic order in $u$, and $q_k$ are nonlinear terms,
at least quadratic in $u$.
Here, $\dot g$ stands for the derivative of the function~$g$ with respect to the
rescaled time variable $s$.
By hypothesis, the function $u(s)$ is small in~$H^1(\RR)$ and $\omega(s)$ is close to $\oz$, for all $s\geq 0$.
Studying the flow in the rescaled variables~$(s,y)$, our objective reduces to proving that the function $u(s)$ converges to $0$ uniformly on compact sets 
of $\RR$ and that $\omega(s)$ has a limit $\omega_+$ as $s\to +\infty$.

\medskip
\noindent\emph{The internal mode \S\ref{se:02}.}
The spectral problem
\[
\begin{cases}\LP V_1 =\lambda V_2\\\LM V_2 =\lambda V_1\end{cases}
\]
is relevant for the dynamics.
Indeed, if there exists a solution $(\lambda,V_1,V_2)$ then $(u_1,u_2)$ defined by
\begin{equation}\label{eq:ln}
u_1(s,y)=\sin(\lambda s) V_1(y)\quad\mbox{and} \quad
 u_2(s,y) =\cos(\lambda s)V_2(y)
\end{equation}
solves the linear evolution system
\[
\begin{cases}
\dot u_1 =\LM u_2\\
\dot u_2 = -\LP u_1
\end{cases}
\]
For example, the identity $\LM \Qo=0$ (which is just the equation of $\Qo$) provides the solution $(0,0,\Qo)$ to the spectral 
problem, but it corresponds to the phase invariance
and it is ruled out by the modulation of $\gamma$ and the orthogonality
relation imposed to $u_2$.
By definition, an \emph{internal mode of oscillations} \eqref{eq:ln} corresponds to 
a solution $(\lambda,V_1,V_2)$ which is \emph{not} related to an invariance.
As discussed in~\cite{PeKA}, 
there exists an internal mode for \eqref{eq:NL}
while there is no internal mode for \eqref{eq:FD}.
Working in the limit where $\omega$ is small,
the internal mode, denoted by $(\lambda,V_1,V_2)$, is such that 
$\lambda=1-\frac{64}{81}\omega^2+O(\omega^3)$.
It is also important to determine the precise
asymptotic expansion of the pair of functions $(V_1,V_2)$ in the limit $\omega\to 0$.
However, this presents a difficulty related to the fact that
$(V_1,V_2)$ converges to the resonance of the integrable case
\eqref{eq:CC}. Indeed, in the integrable case, $(1,1-Q_0  ^2,1)$ 
is formally solution of the spectral problem, which
obviously does not belong to $L^2(\RR)\times L^2(\RR)$.
Lemma~\ref{LE:VW} shows that $(V_1,V_2)$ is close to the resonance in compact
sets of $\RR$, while having exponential decay at $\infty$.
The articles~\cite{CGNT} and~\cite{CoGu} established the existence of
an internal mode
for the subcritical one-dimensional Schrödinger equation
\[
\ii\pt\psi +\px^2\psi +|\psi|^{p-1}\psi= 0
\]
respectively in the limits $p\to 5^-$ and $p\to 3^+$.
Facing the same difficulty of linearizing
around the resonance of the internal mode, the proof in \cite{CoGu} makes use of a Lyapunov-Schmidt reduction and a topological argument.
Here, we propose a different approach, inspired by the factorisation techniques used for evolution equations in \cite{KoMa,KMM2,KMMV,Ma22,RaRo}. 
We introduce a \emph{transformed problem}
\begin{equation}\label{eq:sW}
\begin{cases}\MP W_1 =\lambda W_2\\
\MM W_2 =\lambda W_1\end{cases}
\end{equation} 
where for the integrable case \eqref{eq:CC}, it holds $\MP=\MM = -\py^2+1$ and
for small solitary waves of \eqref{eq:NL}, $M_\pm$ are second order differential operators with small potentials (see \S\ref{se:02}).
We are thus reduced to studying a \emph{weakly coupled} eigenvalue problem,
entering the theory developed in~\cite{Me02}
(see also \cite{RSBK,Si76}).
The relation between the original eigenvalue problem and the transformed
problem \eqref{eq:sW} is based on the identity
\[
S^2 \LP\LM=\MP\MM S^2\quad \mbox{where}\quad S=\py - \frac{\Qo'}{\Qo},
\]
proved in \cite{CGNT,Ma22}, and 
on the introduction of $W_1$ such that $V_1 =(S^*)^2 W_1$.
Once a solution $(\lambda,W_1,W_2)$ of \eqref{eq:sW} is constructed, it is then easy to go back to $(\lambda,V_1,V_2)$.
Note that the introduction of such a transformed problem for linearised
Schrödinger problems in \cite{Ma22} is reminiscent of the mechanism of 
\emph{reduction of eigenvalues} (see \cite{CuM2,DeTr,KMM2}).
In short, the transformed problem eliminates the directions related to the invariances in a more convenient way than projecting onto the orthogonal vector space.

\medskip
\noindent\emph{The second factorisation \S\ref{se:03}.}
Focusing on the sole internal mode $(\lambda,V_1,V_2)$ 
is valid only if there is \emph{no other} internal mode.
To prove uniqueness of the internal mode,
it is also convenient to work on the transformed problem.
To study spectral problems such as \eqref{eq:sW}, it is natural to rely on a \emph{virial} argument.
However, since there exists a solution to \eqref{eq:sW}, it is essential to 
remove it before applying a virial argument.
Following the same strategy, we use
a \emph{second transformation} rather than a projection.
We establish the identity
\[
\cU \MP\MM = \cK \cU \quad \hbox{where} \quad \cU = \py - \frac{W_2'}{W_2}
\]
and where $\cK$ is a fourth order differential operator.
The operator $K$ has two remarkable properties.
It is a perturbation of $(-\py^2+1)^2$ for $\oz$ small
 and its potential is \emph{repulsive}, 
which makes it possible to prove the uniqueness result via a virial
argument on $K$.
The exact property to be used is a part of the main result of
\cite{Si76}, relating the absence of eigenvalue for a second order
differential operator to the sign of the 
 \emph{integral} of its supposedly small potential.
Here, this sign is checked by using the expansion of $(\lambda,W_1,W_2)$
around the (transformed) resonance $(1,1,1)$ of the integrable case.

\medskip
\noindent\emph{Decomposition using the internal mode \S\ref{se:04}.}
Recall that in the absence of internal mode, like for 
the focusing-defocusing model \eqref{eq:FD},
asymptotic stability of solitary waves of the 
nonlinear problem is in some sense a consequence of a \emph{linear asymptotic stability} property,
meaning that the asymptotic stability of the zero solution is true for the \emph{linear} system (modulo invariances).
The existence of the time periodic solution~\eqref{eq:ln} of the linear problem, called
\emph{internal mode of oscillations} in~\cite{PeKA},
rules out the linear asymptotic stability property and it
is thus a serious additional difficulty to prove the asymptotic stability for the nonlinear problem. 
Should this property be true, it has to be deduced from
a special structure of the nonlinearity.
As mentioned in the previous section, the articles \cite{BuP2,PeKA,Si92,SoW3}
pioneered the study of this question, introducing the notion of 
\emph{nonlinear Fermi golden rule}.
As in those papers, we will use a non vanishing property 
related to the internal mode and to the nonlinear terms of the evolution equation
to prove the damping of the internal mode component.
The first step is to extract this component by a usual decomposition
by projection, introducing $v=v_1+\ii v_2$,
\[
u_1 = v_1 + b_1 V_1,\quad u_2 = v_2 + b_2 V_2
\]
where $v_1$ and $v_2$ are orthogonal, respectively, to
$V_2$ and $V_1$.
Then, $(v_1,v_2)$ satisfies the linearised system
\[
\begin{cases}
\dot v_1 =\LM v_2
+\mu_2 + p_2^\perp - q_2^\perp - r_2^\perp\\
\dot v_2 = -\LP v_1
-\mu_1- p_1^\top + q_1^\top + r_1^\top
\end{cases}
\]
where the error terms are mainly projections of the error terms
of the system for $(u_1,u_2)$.
Moreover, the time-dependent function $b=b_1+\ii b_2$ satisfies 
\[
\begin{cases}
\dot b_1 =\lambda b_2+ B_2\\
\dot b_2 = -\lambda b_1 - B_1 
\end{cases}
\]
where $B_1$ and $B_2$ are error terms.
A key observation is that the systems for
$(v_1,v_2)$ and $(b_1,b_2)$ are coupled only at the quadratic level.

\medskip
\noindent\emph{The two-virial strategy \S\ref{se:05}, \S\ref{se:08}, \S\ref{se:09}, \S\ref{se:10}.} 
The first and second transformations used for the spectral problem
are also crucial to study the evolution problem.
The articles~\cite{KMM2,KMMV,Ma22} use only one factorization,
while an arbitrary number of factorisations was considered in \cite{CuM3,CuM4}.
The general strategy can be summarized as follows.
The internal mode component $(b_1,b_2)$ will be controlled in the next
step by a specific computation called the \emph{Fermi golden rule} and the primary objective of the two-virial argument
is to estimate the infinite dimensional component $(v_1,v_2)$. The difficulty is that a direct virial
argument cannot provide a complete estimate on $(v_1,v_2)$ since
there are non trivial solutions of the 
linearised problem due to the invariances.
As described above for the spectral problem, we do not remove those solutions by projection, but 
by factorisation.
As in \cite{Ma22}, elements coming from the invariances are taken care of by the \emph{first transformation}
\[
w_1=\Xt^2 \MM S^2 v_2,\quad w_2 = - \Xt^2 S^2 \LP v_1,
\]
where $\Xt$ is a smoothing operator, close to the identity.
Such a regularisation is necessary to have $w_1,w_2\in H^1$.
The pair of functions $(w_1,w_2)$ then satisfies a nonlinear
system which is perturbative (quadratic terms and error terms are
omitted here) of 
\[
\begin{cases}
\dot w_1 =\MM w_2 \\
\dot w_2 = -\MP w_1 
\end{cases}
\]
This linear system is more favorable than the original one, but it still
has a non trivial time-periodic solution coming from $(\lambda,W_1,W_2)$,
which prevents us from using a direct virial argument.
Thus, as for the spectral problem, we use the \emph{second transformation}
\[ 
z_1 =\Xu\cU w_2,\quad 
z_2 = -\Xu\cU\MP w_1.
\]
Here, $z_1\in H^2$ and $z_2\in L^2$. The pair of functions $(z_1,z_2)$ satisfies 
the transformed system
\[
\begin{cases}
\dot z_1 =z_2 \\
\dot z_2 = -\cK z_1 
\end{cases}
\]
at the linear order
(quadratic terms and error terms are omitted).
Since the potential of operator $K$ is repulsive,
as for the second transformed spectral problem, one can use a virial argument on this system 
to prove the linear asymptotic stability.
We mention a technical difficulty here, already solved in \cite{KMM2}.
The introduction of the transformed problems and of the necessary regularisation arguments breaks the structure of the nonlinear terms,
which is required to treat them by a virial argument.
Thus, one has to localize the virial argument on the transformed problem.
This provides estimates on $(z_1,z_2)$ only on compacts sets in space, with error terms outside this compact set.
The strategy designed in \cite{KMM2}
is to use a first localized virial argument to estimate the functions $(v_1,v_2)$ 
 at a large scale $A$, in terms of 
local norms and of the internal mode component. 
At the level of $(v_1,v_2)$, the structure of the nonlinear terms is preserved
and only the spectral argument is missing, which justifies the error term
in local norm; see
Lemma \ref{LE:v1}.
A localized virial argument on the second transformed problem is then used in Lemma \ref{LE:vz}, at a scale $B$, with $1\ll B\ll A$.
Exchanging information between the functions $(v_1,v_2)$ and $(z_1,z_2)$
requires estimates, the most delicate ones
being what we call \emph{coercivity estimates}, proved in \S\ref{se:09}, 
and reminiscent of coercivity properties proved in \cite{We86}.
The orthogonality relations for $(v_1,v_2)$ are 
required in this step.

\medskip
\noindent\emph{The Fermi golden rule \S\ref{se:06}, \S\ref{se:07}.}
The goal of the Fermi golden rule is to prove that the internal mode component $b$ of the solution is \emph{nonlinearly damped}, 
which rules out the periodic behavior illustrated by \eqref{eq:ln} for the linear system.
In previous approaches (see a few classical references
in \S\ref{sb:12}), a formal
ansatz of the solution $(v_1,v_2)$ is inserted into the system for
$(b_1,b_2)$, providing an approximate nonlinear system
for $(b_1,b_2)$ containing a \emph{damping cubic term}. The presence of this cubic term is one manifestation of the \emph{Fermi golden rule}.
However, this approach requires rather strong information on the infinite
dimensional part $(v_1,v_2)$ and we only expect 
to have the estimate
$\int_0^{+\infty} \|v\|_{\textnormal{loc}}^2 < +\infty$,
for some local norm $\|\cdot\|_{\textnormal{loc}}$.
In our approach, inspired by \cite{KoMa,KMM1},
we rather use the quadratic terms 
in the system for $(v_1,v_2)$ (such terms appear in $q_1^\perp$ and $q_2^\perp$,
see Lemmas~\ref{LE:tm} and~\ref{LE:sm}) 
and we introduce a simple functional to show the estimate
$\int_0^{+\infty} |b|^4 <+\infty$, 
provided that $\|v\|_{\textnormal{loc}}$ is already estimated. 
This proof of a weak form of damping
is the content of Lemma \ref{LE:eb}.
As in the classical approach, it is crucial that a certain constant does not vanish,
which is checked in \S\ref{se:06},
 using the asymptotic expansion of $(V_1,V_2)$
close to the resonance in the limit~$\omega$ small.
As in \cite{KoMa,KMM1}, a drawback of this approach is the relatively weak information obtained on the behavior of $b$.
In spite of the rather weak estimates obtained
on $v$ and $b$, we are able to 
prove that both $v(s)$ and $b(s)$ converge to zero and that
$\omega(s)$ has a limit as $s\to+\infty$, by using oscillatory properties
of $\dot \omega$.

\medskip
\noindent\emph{Double linearisation.}
The proof of Theorem~\ref{TH:as} is thus based on two linearisations.
Firstly, we study solutions in a vicinity of solitary waves and
linearize around adequately chosen standing waves in phase and frequency.
After a first transformation related to natural directions for \eqref{eq:NL},
the presence of an internal mode
leads us to introduce and study a second transformed problem.
Secondly, the description
of the spectral properties of the linearised operator, the verification of the Fermi golden rule and
the fact that the second transformed problem involves a repulsive potential all
rely on computations based on the linearization of the model \eqref{eq:NL} around the integrable case, 
by considering only small solitary waves.
However, the analysis can be extended to more general models, under natural   assumptions
such as the
existence of an internal mode, the Fermi golden rule and the repulsive nature of the second transformed problem.

The notation $\lesssim$ will be used to replace $\leq C$ for a constant $C>0$
independent of the parameters
$\oz$, $\varepsilon$, $\delta$, $\theta$, $A$ and $B$.
We denote~$\langle u , v\rangle =\Re\bigl(\int u\bar v\bigr)$ and~$\|u\|=\sqrt{\langle u , u\rangle}$.

\section{The internal mode}\label{se:02}
We define the operators
\begin{align*}
\LP& = -\py^2+1 -3\Qo^2-5\omega\Qo^4, 
&\MP &= -\py^2 + 1 +\frac\omega3\Qo^{4},\\
\LM &= -\py^2+1-\Qo^2-\omega\Qo^4 , 
&\MM &= -\py^2 + 1 -\omega\Qo^{4},
\end{align*}
and
\[
S =\py-\frac{\Qo'}{\Qo},\qquad
S^* =-\py-\frac{\Qo'}{\Qo}.
\]
We recall without proof an identity from~\cite[\S 3.4]{CGNT} and~\cite[Lemma 7]{Ma22}, which motivates the introduction of $\MP$ and $\MM$.

\begin{lemma}\label{LE:LM}
For any~$\omega>0$,
$S^2\LP\LM =\MP\MM S^2$ and
$\LM\LP(S^*)^2=(S^*)^2\MM\MP$.
\end{lemma}
\begin{remark}
The above identity was inspired by simpler conjugaison relations, such as
\[
S \LM = \MP S
\]
deduced from $\LM= S^* S$ and $\MP=SS^*$.
The interest of such identities lies on the properties of the transformed operators 
$\MP$ and $\MM$, which are more favorable than the ones of $\LP$ and $\LM$
from the spectral point of view. Indeed,
the potentials involved in $\MP$ and $\MM$ are small for $\omega$ small,
and the potential of $\MP$ is repulsive
(in the sense that $y(\Qo^4)'\leq 0$ on $\RR$).
The use of an identity similar to $S \LM = \MP S$ is crucial in~\cite{KMM2}
and the main result in~\cite{Ma22} is based on the analogue of Lemma~\ref{LE:LM} for~\eqref{eq:FD} and 
the properties of the corresponding operators 
$\MP$, $\MM$.
Here, the situation is less favorable than in \cite{Ma22} since the potential in $\MM$ is not repulsive and is larger,
in absolute value, than the one of $\MP$. Actually, we will prove in this section that 
the operator $\MP\MM$ has a non trivial eigenvalue.
Note that for the integrable case, one has $\MP=\MM=-\py^2+1$, which 
is a motivation for working close to the integrable case, that is for
$\omega>0$ small.
\end{remark}

This section is devoted to the proof of existence of $\lambda\neq 0$ 
and of a non trivial pair of 
smooth functions $(V_1,V_2)$
satisfying the eigenvalue problem
\begin{equation}\label{eq:VV}
\begin{cases}\LP V_1 =\lambda V_2\\
\LM V_2 =\lambda V_1\end{cases}
\end{equation}
for all small $\omega >0$.
The key observation is that if $\lambda \neq0$ and
$(W_1,W_2)$ satisfy
\begin{equation}\label{eq:WW}
\begin{cases}\MP W_1 =\lambda W_2\\
\MM W_2 =\lambda W_1\end{cases}
\end{equation}
then by~Lemma~\ref{LE:LM}, we have
\[
\LM\LP(S^*)^2 W_1 = (S^*)^2\MM\MP W_1=\lambda^2(S^*)^2 W_1.
\]
Thus, setting $V_1 =(S^*)^2 W_1$ and $V_2 =\lambda^{-1}\LP V_1$,
the pair $(V_1,V_2)$ solves~\eqref{eq:VV} with the same~$\lambda$.
With this in mind, we prove an existence result concerning
the eigenvalue problems~\eqref{eq:VV} and~\eqref{eq:WW}.

\begin{lemma}\label{LE:VW}
There exist $\omega_1>0$, a smooth function
$\alpha:(0,\omega_1)\to (0,+\infty)$
and smooth, even functions
$W_1,W_2:(0,\omega_1)\times\RR\to \RR$
that satisfy the properties \emph{(i)-(v)} on $(0,\omega_1)$.
\begin{enumerate}
\item\label{it:mu}\emph{Expansion of $\alpha$ at $0$:}
$
\alpha(\omega) =\frac89\omega +\omega^2\tilde\alpha(\omega)$
where $|\tilde\alpha^{(k)}|\lesssim 1$, for all $k\geq 0$.
\item\label{it:dV}\emph{Resolution of the eigenvalue problem.} 
Setting $\lambda=1-\alpha^2$, $(\lambda,W_1,W_2)$ solves \eqref{eq:WW}.
Setting $V_1 =(S^*)^2 W_1$ and $V_2 =\lambda^{-1}\LP V_1$,
 $(\lambda,V_1,V_2)$ solves \eqref{eq:VV}.
 \item\label{it:eV}\emph{Expansion of the eigenfunctions:} 
$V_1 = 1- Q_0^2 +\omega R_1 +\omega^2\tilde V_1$, 
$V_2 = 1 +\omega R_2 +\omega^2\tilde V_2$ and
$W_j = 1 +\omega S_j +\omega^2\tilde W_j$,
for $j=1,2$,
where the functions $R_j$, $S_j$, independent of $\omega$, and 
the functions $\tilde V_j$, $\tilde W_j$
satisfy on $\RR$, for all $k\geq 0$, 
\begin{align*}
&|R_j^{(k)}|+|S_j^{(k)}|\lesssim 1+|y|,\\
&|\py^k\tilde V_j|+|\py^k\tilde W_j|
+\frac{|\py^k\PO\tilde V_j|}{1+|y|}
+\frac{|\py^k\PO\tilde W_j|}{1+|y|}
\lesssim 1+y^2.
\end{align*}
\item\label{it:bW}\emph{Decay properties.} For $j=1,2$, for all $k\geq 0$, on $\RR$, it holds that
\[
|\py^kW_j|\lesssim\omega^k \ee^{-\alpha|y|}+ \omega\ee^{-|y|},\quad
|\py^kV_j|
+ \frac{|\py^k\PO V_j|}{1+|y|}+ \frac{|\py^k\PO W_j|}{1+|y|}
\lesssim\omega^k \ee^{-\alpha|y|}+ \ee^{-|y|}.
\]
For all $k\geq 0$, on $\RR$, it holds that
$|\py^k(W_1-W_2)(y)|\lesssim \omega \ee^{-\kappa|y|}$
where $\kappa = \sqrt{2-\alpha^2}$.
\item\label{it:aW}\emph{Asymptotic properties.} For $j=1,2$, on $\RR$,
 it holds that
\[
\big| W_j - \ee^{-\alpha|y|}\big|
+\big| V_1-(1-Q_0^2)\ee^{-\alpha|y|}\big|
+\big| V_2 - \ee^{-\alpha|y|}\big|
\lesssim\omega \ee^{-\alpha|y|}.
\]
In particular,
$|\langle W_1,W_2\rangle - 1/\alpha|
+|\langle V_1,V_2\rangle - 1/\alpha|\lesssim 1$.
\end{enumerate}
\end{lemma}

\begin{remark}
The functions $R_j$ and $S_j$ have explicit expressions;
see~\eqref{eq:TT1}, \eqref{eq:TT2}, \eqref{eq:SS},~\eqref{eq:R1} and~\eqref{eq:R2}.
\end{remark}

\begin{remark}
Recall that $(1,1-Q_0^2,1)$ is the \emph{resonance} of the integrable case, which corresponds
in the present setting to $\omega=0$.
Indeed, using $(Q_0^2)'' = 4Q_0^2-3Q_0^4$, 
we check that $(-\py^2 + 1- 3Q_0^2) (1-Q_0^2)=1$
and $(-\py^2 + 1- Q_0^2) 1=1-Q_0^2$.
The expansions $\lambda = 1 + O(\omega^2)$ and $V_1=1-Q_0^2+O(\omega)$, $V_2=1+O(\omega)$
 on compact sets of $\RR$, mean that $(\lambda,V_1,V_2)$ bifurcates from this resonance.
However, for $\omega>0$ small, the eigenfunction
$(V_1,V_2)$ belongs to $L^2$ as shown by the decay property~\ref{it:aW} of Lemma~\ref{LE:VW}.
\end{remark}

\begin{proof}
(i) The eigenvalue problem.
 We define an auxiliary problem of the transformed system~\eqref{eq:WW}.
For $\omega>0$ small, setting $\lambda = 1-\alpha^2$, $\kappa^2 = 1+\lambda=2-\alpha^2$ ($\alpha>0$, $\kappa>0$) and
\[
Z_1 =\tfrac 12(W_1+W_2),\quad Z_2 =\tfrac 12(W_1-W_2),
\]
we look for $(\alpha,Z_1,Z_2)$
satisfying the eigenvalue problem
\begin{equation}\label{eq:ZZ}\begin{cases} 
-\py^2 Z_1 +\alpha^2 Z_1
-\frac 13\omega\Qo^4 Z_1 +\frac 23\omega\Qo^4 Z_2= 0\\
-\py^2 Z_2 +\kappa^2 Z_2
+\frac 23\omega\Qo^4 Z_1 -\frac 13\omega\Qo^4Z_2= 0
\end{cases}\end{equation}
An important feature of this system is to be \emph{weakly coupled} for small $\omega$,
entering the category of systems that can be treated by perturbation 
following the theory developed in~\cite{Si76} for the scalar case
(see also~\cite[XIII.3, XIII.17]{RSBK}),
and in~\cite{Me02} for the vectorial case.
We closely follow~\cite{Me02}. Introducing the matrix notation
\[
Z=\begin{pmatrix} Z_1\\ Z_2\end{pmatrix},\quad
\Ha = 
\begin{pmatrix} -\py^2 +\alpha^2 & 0\\ 0 & -\py^2+\kappa^2\end{pmatrix},\quad
\Po = -\frac 13\Qo^4\begin{pmatrix} 1 & -2\\ -2 & 1\end{pmatrix},
\]
we rewrite the system~\eqref{eq:ZZ} as
\begin{equation}\label{eq:sZ}
(\Ha +\omega\Po)Z = 0.
\end{equation}
To reach the \emph{Birman-Schwinger formulation},
we define
\[
\Po^2 =\frac 19\Qo^8\begin{pmatrix} 5 & -4\\ -4 & 5\end{pmatrix},
\quad
|\Po| =\frac 13\Qo^4\begin{pmatrix} 2 & -1\\ -1 & 2\end{pmatrix},
\]
and
\[
|\Po|^\frac12 =\frac 1{\sqrt3}\Qo^2 
\begin{pmatrix} c & -\frac1{2c}\\ -\frac1{2c} & c\end{pmatrix},\quad
\Po^\frac12 =\begin{pmatrix} 0 & 1\\ 1 & 0\end{pmatrix}|\Po|^\frac12,
\]
where $c=(1+\frac{\sqrt{3}}2)^{1/2}$.
Note that the exact expressions of the matrices above will not be used,
but only basic estimates and the property 
\[
\Po^\frac12|\Po|^\frac12=|\Po|^\frac12\Po^\frac12=\Po.
\]
We define the operator $\Ka$ on $L^2(\RR)\times L^2(\RR)$ by
\[
\Ka=\Po^\frac12\Ha^{-1}|\Po|^\frac12 =
\Po^\frac12\begin{pmatrix}(-\py^2 +\alpha^2)^{-1} & 0\\
0 &(-\py^2 +\kappa^2)^{-1}\end{pmatrix}|\Po|^\frac12
\]
with the integral kernel
\[
\Ka(y,z)
=\frac1{2\alpha}\Po^\frac12(y)\begin{pmatrix} \ee^{-\alpha|y-z|} & 0
\\ 0 &\frac\alpha\kappa \ee^{-\kappa|y-z|}\end{pmatrix}|\Po(z)|^\frac12.
\]
Since we expect $\alpha$ to be close to $0$ and $\kappa$ to be close
to $\sqrt{2}$, we expand
\[
\Ka =\La +\Ma
\quad \mbox{where}\quad
\La(y,z)
=\frac 1{2\alpha}\Po^\frac12(y)\begin{pmatrix} 1 & 0\\ 0 & 0\end{pmatrix}|\Po(z)|^\frac12
\]
and
\[
\Ma(y,z) =\Po^\frac12(y)\Na(y,z)|\Po(z)|^\frac12, \quad
\Na(y,z) =\frac1{2\alpha}\begin{pmatrix} \ee^{-\alpha|y-z|}-1 & 0\\
0 &\frac\alpha\kappa \ee^{-\kappa|y-z|}\end{pmatrix}.
\]
By the decay properties of the function $\Qo$, the map
$(\alpha,\omega)\mapsto\Ma$, 
extended by
\[
M_{0,\omega}(y,z)=\Po^\frac12(y)N_0(y,z)|\Po(z)|^\frac12,\quad
N_0(y,z)=\frac1{2}\begin{pmatrix} -|y-z| & 0\\
0 &\frac{\sqrt{2}}2 \ee^{-\sqrt{2}|y-z|}\end{pmatrix},
\]
is well-defined and analytic in the Hilbert-Schmidt norm
in a neighborhood of $(0,0)$. (See for instance~\cite[Proof of Theorem~2.6]{Si76}.)

We observe that~\eqref{eq:sZ} is satisfied by $(\alpha,Z)$ if, and only if,
the function $\Psi=\Po^{1/2} Z$ solves
$\Psi = -\omega\Po^{1/2}\Ha^{-1}|\Po|^{1/2}\Psi= -\omega\Ka\Psi$.
Hence, the existence of $(\alpha,Z)$ solving~\eqref{eq:sZ}
is equivalent to the existence
of $\Psi\in L^2$, $\Psi\not\equiv0$, such that $\Psi+\omega\Ka\Psi = 0$.
(See also~\cite[Proposition~4.2]{Me02}.)
By the expansion of $\Ka$, this equation is equivalent to
$\Psi +\omega(1+\omega\Ma)^{-1}\La\Psi=0$.
(The existence and the analytic regularity of the operator $(1+\omega\Ma)^{-1}$
follows from the estimate $|||\omega\Ma|||<1$ for $\omega$ small
where $|||\cdot|||$ denotes the operator norm $L^2\to L^2$).
Hence, $-1$ is an eigenvalue of the operator $\omega\Ka$ 
if, and only if, $-1/\omega$ is an eigenvalue of the operator 
$(1+\omega\Ma)^{-1}\La$.
(See also~\cite[(iii) of Lemma 4.5]{Me02}.)
Therefore, our next goal is to find $\alpha>0$ small
such that $-1/\omega$ is an eigenvalue of
the operator $(1+\omega\Ma)^{-1}\La$. 
More generally, we consider the eigenvalue problem for $(\mu,\Psi)$
\begin{equation}\label{eq:mP}
(1+\omega\Ma)^{-1}\La\Psi =\mu\Psi
\end{equation}
which has a remarkable property:
by definition, $\La$ is a rank one operator
\[
(\La\varphi)(y)
=\frac {p_\omega(\varphi)}{2\alpha}\Po^\frac12(y)\eu\quad\mbox{where}\quad
p_\omega(\varphi)=\int\eu\cdot\bigl(|\Po|^\frac 12\varphi\bigr)
\]
for any $\varphi\in L^2(\RR)$.
Here, $\eu=(1,0)^\TR\in\RR^2$
and $a\cdot b$ denotes the scalar product in $\RR^2$. 
For a vector $v=(v_1,v_2)^\TR$, we denote $|v|_\infty=\sup_{k=1,2}|v_k|$
and for a $2\times2$ matrix $M=(M_{i,j})_{i,j=1,2}$,
 $|M|_\infty=\sup_{i,j=1,2}|M_{i,j}|$.
Thus, $(\mu,\Psi)$ solves~\eqref{eq:mP} if and only if
\begin{equation}\label{eq:Ps}
p_\omega(\Psi)(1+\omega\Ma)^{-1}\bigl(\Po^\frac12\eu\bigr)
=2\alpha\mu\Psi .
\end{equation}
Defining the even function
$\Psi=(1+\omega\Ma)^{-1}\bigl(\Po^{1/2}\eu\bigr)$
and $r:(\alpha,\omega)\mapsto r(\alpha,\omega)
= p_\omega\bigl(\Psi\bigr)$,
we see that $(\mu,\Psi)$ solves~\eqref{eq:Ps} if and only if 
$r(\alpha,\omega) = 2\alpha\mu$.
Therefore, $-1/\omega$ is an eigenvalue of
the operator $(1+\omega\Ma)^{-1}\La$ if and only if
$s(\alpha,\omega)=0$, where
\[
s(\alpha,\omega)=\alpha +\tfrac12\omega r(\alpha,\omega).
\]
The operators $\Ma$ and 
$(1+\omega\Ma)^{-1}$ are well-defined and analytic in a neighborhood of $(0,0)$ and the operator $\Po$ is analytic in $\omega$.
Thus, the function $r$ is analytic in a neighborhood of $(0,0)$.
Since 
$(\partial s/\partial\alpha)(0,0)=1$,
by the Implicit Function Theorem, there exists
an analytic function $\omega\mapsto\alpha(\omega)$ defined
in a neighborhood of $0$ such that $s(a,\omega)=0$
if and only if $a=\alpha(\omega)$.
By
\[
r(0,0)
=p_0\bigl( P_0^\frac12\eu\bigr)
=\int\eu\cdot(\Pz\eu)
=-\frac13\int Q_0^4 = -\frac{16}9
\]
we have the expansion
$\alpha(\omega) =\frac 89\omega +\omega^2\tilde\alpha(\omega)$
where $\tilde\alpha(\omega)=\int_0^1(1-\tau)\alpha''(\omega\tau) d\tau$
is bounded in a neighborhood of $0$, as well as all its derivatives.

(ii)-(iii) Construction and expansion of the eigenfunctions.
For $\omega=0$, we denote $Q= Q_0$, where
\[
 Q(y) =\frac{\sqrt 2}{\cosh y}
\mbox{ is a solution of }
 - Q''+ Q - Q^3=0.
 \]
 Moreover $Q$ satisfies
$(Q')^2-Q^2+\frac 12 Q^4 = 0$.
From the explicit expression of $\Qo$ in \eqref{eq:Qo},
one checks the expansions at $0$
\begin{equation}\label{eq:eQ}
a_\omega = 1 +\tfrac83\omega + O(\omega^2),\quad
\Qo = Q +\omega E +\omega^2 Q\tilde E
\end{equation}
where $E$ is defined by
$E = {\PO\Qo|_{\omega=0}} = -\tfrac 43 Q +\tfrac 13 Q^3$,
and where~$\tilde E$ and all its derivatives are bounded on $\RR$,
uniformly for $\omega$ small.
We check that $\LP E=Q^5$
which is also a consequence of differentiating the equation $\Qo''-\Qo+\Qo^3
+\omega \Qo^5=0$ with respect to $\omega$.

Now, $\alpha$ denotes $\alpha(\omega)$,
the function constructed above for small $\omega>0$.
We compute the first order expansion in $\omega$ 
of the eigenfunction $Z$ of~\eqref{eq:sZ}
corresponding to the eigenfunction $\Psi=\Po^{1/2} Z$ of~\eqref{eq:mP}
chosen as before with the normalisation $p_\omega(\Psi)=-2\alpha/\omega$,
that is $\Psi=(1+\omega\Ma)^{-1}\bigl(\Po^{1/2}\eu\bigr)$.
By the definition of $\Ma$,
\[
\Psi
=\Po^\frac12\eu-\omega\Ma(1+\omega\Ma)^{-1}\bigl(\Po^\frac12\eu\bigr)\\
=\Po^\frac12\left(\eu
-\omega\Na\Yo\right)
\]
where
$\Yo=|\Po|^\frac12(1+\omega\Ma)^{-1}\bigl(\Po^\frac12\eu\bigr)$.
Set also $\Yz=\Pz\eu$.
By the relation $\Psi=\Po^{1/2} Z$, we obtain
$Z=\eu-\omega\Na\Yo$ and 
we note that $|Z|_\infty\lesssim 1$ on $\RR$.
We also note the expansion
\[
Z=\eu-\omega N_0\Yz +\omega^2\tilde Z
=\eu +\omega\begin{pmatrix} T_1\\T_2\end{pmatrix} +\omega^2\tilde Z
,\quad
\tilde Z =\frac 1\omega (\Na\Yo-N_0\Yz)
\]
where
\begin{align}
T_1&=-\frac 16 \int |y-z| Q^4(z) dz=\frac 19 Q^2 +\frac 89\ln Q - \frac43 \ln 2
=\frac 19 Q^2 +\frac 89 \ln (Q/ \sqrt{8}) ,
\label{eq:TT1}\\
T_2&= -\frac{\sqrt{2}}6\int \ee^{-\sqrt{2}|y-z|} Q^4(z) d z.\label{eq:TT2}
\end{align}
The expression of $T_1$ is justified by checking that it satisfies the equation
$-T_1'' =\frac 13 Q^4$ and moreover that $T_1(0)=-\frac 16 \int |z| Q^4(z) dz=
-\frac 43 \int_0^\infty z \sech^4(z) d z=\frac 29 - \frac 89 \ln 2$.
The function $T_2$ satisfies
$- T_2''+2 T_2 = -\frac 23 Q^4$
on $\RR$. Observe that one formally gets those equations by inserting 
$Z_1=1+\omega T_1$ and $Z_2=\omega T_2$ into~\eqref{eq:ZZ}.

Now, we estimate $\tilde Z$ uniformly for small $\omega$.
From the elementary estimates
$|\ee^{-\alpha|y-z|}-1|\lesssim\alpha(1+|y|+|z|)$,
\[
\big|\ee^{-\alpha|y-z|}-1+\alpha|y-z|\big|\lesssim\alpha^2(1+|y|+|z|)^2,\quad
\bigl|\tfrac1\kappa \ee^{-\kappa|y-z|} 
-\tfrac1{\sqrt{2}} \ee^{-\sqrt{2}|y-z|}\bigr|\lesssim\alpha^2,
\]
it holds $|\Na(y,z)|_\infty\lesssim 1+|y|+|z|$ and
$|\Na(y,z)-N_0(y,z)|_\infty\lesssim\omega(1+|y|+|z|)^2$ on $\RR^2$.
From \eqref{eq:eQ} and $|\Ma(y,z)|_\infty\lesssim 1$, it holds
$|\Yo|_\infty\lesssim \ee^{-2|y|}$ and 
$|\Yo-\Yz|_\infty\lesssim\omega \ee^{-2|y|}$ on $\RR$.
Thus, it holds $|\tilde Z|_\infty\lesssim 1+y^2$ on $\RR$.
We derive similar estimates for the space derivatives of $\tilde Z$,
for all $k\geq 0$, on $\RR$,
$|\py^k \tilde Z|_\infty\lesssim 1+y^2$ (not optimal).
Moreover,
\[
\PO\tilde Z
=-\frac1{\omega^2}\bigl(\Na\Yo-N_0\Yz-\omega\PO(\Na\Yo)\bigr)
=-\frac 1\omega\int_0^\omega \omega_1 \PO^2(\Na\Yo)_{|\omega=\omega_1}d\omega_1
\]
From this identity, proceeding as before
and using $|\PO\alpha|\lesssim 1$,
we establish the estimates 
$|\py^k\PO\tilde Z|_\infty\lesssim 1+|y|^3$ on $\RR$, for all $k\geq 0$.
In what follows, $\cO$ denotes any smooth function $g$ of $\omega$ and $y$, possibly different from one line to another,
and such that for any $k\geq 0$, $|\py^k g|\lesssim 1+y^2$
and $|\py^k\PO g|\lesssim 1+|y|^3$, on $\RR$.
In particular, $\tilde Z_1 =\cO$ and $\tilde Z_2=\cO$.
Now, we define
a solution $(\lambda,W_1,W_2)$ of~\eqref{eq:WW} by setting
$W_1=Z_1+Z_2$, $W_2=Z_1-Z_2$ so that
\begin{equation}\label{eq:SS}
W_1 =1+\omega S_1+\omega^2\cO,\quad
W_2 =1+\omega S_2+\omega^2\cO,
\quad
S_1 = T_1+ T_2,\quad S_2 = T_1-T_2.
\end{equation}
Lastly, we define 
$V_1 =(S^*)^2 W_1$ and $V_2=\lambda^{-1}\LP V_1$
so that $(\lambda,V_1,V_2)$ is a solution of~\eqref{eq:VV}.
We observe that by construction, the functions $V_1$, $V_2$, $W_1$ and $W_2$ are
even. Then,
\[
V_1 =\frac{\Qo''}{\Qo} W_1+ 2\frac{\Qo'}{\Qo} W_1' + W_1''
= 1- Q^2 +\omega R_1 +\omega^2\cO,
\]
where
\begin{align}
R_1 & =- 2 Q E - Q^4 +(1- Q^2) S_1+2\frac{Q'}{Q} S_1'+ S_1''\nonumber \\
&=\frac {16}9 +\frac{7}3 Q^2 -\frac{5}3 Q^4 
+\frac89(1- Q^2)\ln (Q/\sqrt{8}) 
+(3- Q^2) T_2 + 2\frac{Q'}{Q} T_2'.\label{eq:R1}
\end{align}
Since $\lambda = 1-\alpha^2 = 1 + O(\omega^2)$ and
$( Q^2)'' = 4 Q^2 - 3 Q^4$,
we also have
\[
V_2
= - V_1'' + V_1 - 3\Qo^2 V_1 - 5\omega\Qo^4V_1 +\omega^2\cO
= 1 +\omega R_2+\omega^2\cO,
\]
where
\begin{align}
R_2 
&= -R_1'' +R_1- 3 Q^2 R_1 - 6 Q(1- Q^2) E-5 Q^4(1- Q^2)\nonumber\\
&= \frac{16}9 - \frac 13 Q^2 + \frac 49 Q^4
+\frac 89 \ln (Q/\sqrt{8})  - 3T_2 - 2 \frac {Q'}{Q} T_2'.\label{eq:R2}
\end{align}
From this expression of $R_2$, one also checks that
$R_1= - R_2'' + R_2 - Q^2 R_2 - 2 QE - Q^4$.
This equation and \eqref{eq:R2} correspond to the first order
linearization of the system \eqref{eq:VV}
around the resonance $(1,1-Q^2,1)$ corresponding to the case $\omega=0$.

(iv) Decay properties of the eigenfunctions.
By $|Z|_\infty\lesssim 1$,
we have
$|\Po Z|_\infty\lesssim\ee^{-4|y|}$ on~$\RR$.
Thus, using $Z = -\omega\Ha^{-1}(\Po Z)$ from~\eqref{eq:sZ}, and
\begin{equation}\label{eq:oH}
\omega\Ha^{-1}(y,z)= 
\frac\omega {2\alpha}\begin{pmatrix} \ee^{-\alpha|y-z|} & 0\\
0 &\frac\alpha{\kappa} \ee^{-\kappa|y-z|}\end{pmatrix},
\end{equation}
we obtain, on $\RR$,
\[
|Z_1| \lesssim\frac\omega\alpha\int \ee^{-\alpha|y-z|} \ee^{-3|z|} d z
\lesssim \ee^{-\alpha|y|},\quad
|Z_2| \lesssim\omega\int \ee^{-\kappa|y-z|}\ee^{-4|z|} d z
\lesssim\omega \ee^{-\kappa|y|}.
\]
More generally, 
differentiating \eqref{eq:oH} for $k=1$, and then using the system \eqref{eq:ZZ} for $k\geq 2$,
we check that for all $k\geq 0$, on $\RR$,
\[
|\py^kZ_1|\lesssim\omega^k \ee^{-\alpha|y|}+\omega \ee^{-4|y|},\quad
|\py^kZ_2|\lesssim\omega \ee^{-\kappa|y|}.
\]
Using this estimate and $\kappa\geq 1$, it holds, for $j=1,2$, for all $k\geq 0$,
on $\RR$,
\begin{equation}\label{eq:Wk}
|\py^kW_j|\lesssim\omega^k \ee^{-\alpha|y|}+ \omega \ee^{-|y|}.
\end{equation}
Using $V_1 =(S^*)^2 W_1$ and $V_2=\lambda^{-1}\LP V_1$, we derive the estimate
$|\py^kV_j|\lesssim\omega^k \ee^{-\alpha|y|}+ \ee^{-|y|}$
for all $k\geq 0$.
Now, we estimate $\PO Z$.
From the definition $Z=\eu-\omega\Na\Yo$, we check that
$|Z|_\infty\lesssim 1$ and differentiating with respect to
$\omega$, we check $|\PO Z|_\infty\lesssim 1+|y|$.
This being known, differentiating $Z = -\omega\Ha^{-1}(\Po Z)$ with respect to $\omega$, where $\omega\Ha^{-1}$
is given in~\eqref{eq:oH},
we also obtain, for all $k\geq 0$, on $\RR$,
\[
\big|\py^k\PO Z_1\big|+\big|\py^k\PO Z_2\big|
\lesssim(1+|y|) \bigl( \omega^k \ee^{-\alpha|y|}+ \ee^{-|y|}\bigr).
\]
This implies the estimates for $\PO V_j$ and $\PO W_j$
stated in the lemma.

(v) Finally, we describe the asymptotic behavior of the eigenfunctions.
From~\eqref{eq:ZZ},
\[
Z_1(y)=\frac{\omega}{6\alpha}\int \ee^{-\alpha|y-z|}\Qo^4(z)(Z_1-2Z_2)(z)d z.
\]
Using the inequalities
$|\ee^{-u}-1|\leq|u|\ee^{|u|}$, $||y-z|-|y||\leq|z|$, for all $u,y,z\in\RR$,
and the monotonicity of $u\mapsto u\ee^u$ on $[0,+\infty)$,
we note that
\begin{align*}
\Bigl|\int \ee^{-\alpha|y-z|}Q^4(z) dz -\ee^{-\alpha|y|}\int Q^4\Bigr|
&\leq\alpha \ee^{-\alpha|y|}\int||y-z|-|y||\ee^{\alpha||y-z|-|y||} Q^4(z) dz\\
&\leq\alpha \ee^{-\alpha|y|}\int|z|\ee^{\alpha|z|}Q^4(z) dz
\lesssim\omega \ee^{-\alpha|y|}.
\end{align*}
Using this and the estimates
\[
|Z_1-2Z_2-1| \lesssim\omega(1+y^2),\quad
|\Qo^4-Q^4| \lesssim \omega \ee^{-4|y|},\quad
\Big|\frac{\omega}{6\alpha}\int Q^4 - 1\Big| \lesssim \omega^2,
\]
we obtain
$ | Z_1 - \ee^{-\alpha|y|} | \lesssim\omega \ee^{-\alpha|y|}$ on $\RR$.
We have already proved $|Z_2|\lesssim\omega \ee^{-\kappa|y|}$.
Thus, on $\RR$,
\[
\big| W_1- \ee^{-\alpha|y|}\big| +\big| W_2- \ee^{-\alpha|y|}\big|
\lesssim\omega \ee^{-\alpha|y|}.
\]
Using the definitions $V_1 =(S^*)^2 W_1$, $V_2=\lambda^{-1}\LP V_1$,
the identity $( Q^2)'' = 4 Q^2 - 3 Q^4$
and the estimate \eqref{eq:Wk},
we check the corresponding estimates for $V_1$ and $V_2$
given in the lemma.
The last estimate
$|\langle W_1,W_2\rangle -1/\alpha|
+|\langle V_1,V_2\rangle -1/\alpha|\lesssim 1$
follows by integration.
\end{proof}

\section{Second factorisation}\label{se:03}

Since there exists an internal mode,
a second factorization is needed, both to understand
the spectral problem \eqref{eq:VV} and to study the linear evolution problem.
It involves the eigenfunction~$W_2$ of the transformed operator $\MP\MM$.
By~\ref{it:bW} and \ref{it:aW} of
 Lemma~\ref{LE:VW}, it holds $W_2\geq (1-C\omega) \ee^{-\alpha|y|}>0$ and~$ | {W_2'}/{W_2} |\lesssim\omega$, on $\RR$.
We set
\[
\cU =\py -\frac{W_2'}{W_2} .
\]

\begin{lemma}\label{LE:sf}
For $\omega>0$ small,
$\cU\MP\MM =\cK\cU$ where
\[
\cK =\py^4 - 2 \py^2 + K_2\py^2 + K_1\py +K_0+1,
\]
and the functions $K_2$, $K_1$, $K_0$
satisfy, for all $k\geq 0$, on $\RR$,
\begin{equation}\label{eq:Kj}
|\py^kK_2|+|\py^kK_1|+|\py^kK_0| \lesssim \omega \ee^{-(\kappa -\alpha)|y|}.
\end{equation}
\end{lemma}

\begin{proof}
For any smooth function~$h$, we set $g= h/W_2$ and $\rr=\cU h=W_2g'$.
We compute
\[
\MM h =\MM(W_2 g) 
=(\MM W_2) g - 2 W_2' g' - W_2 g'' 
=\lambda W_1 g - 2W_2' g' - W_2 g'',
\]
using $\MM W_2=\lambda W_1$, so that
\[
\MP\MM h =\lambda\MP(W_1 g) + (\py^2 - 1)(2W_2' g'+W_2 g'') 
 -\frac {2\omega}3 \Qo^4 W_2' g' -\frac \omega3\Qo^4 W_2 g''
\]
and then, using $\MP W_1=\lambda W_2$ and $W_2g=h$,
\begin{align*}
\MP\MM h 
& =\lambda ^2h - 2\lambda W_1' g' - \lambda W_1 g''
+2 W_2''' g' + 4 W_2''g'' + 2 W_2'g'''
+W_2'' g'' \\&\quad+ 2 W_2'g'''+ W_2 g''''
-2 W_2'g' - W_2 g'' -\frac {2\omega}3 \Qo^4 W_2' g' -\frac \omega3\Qo^4 W_2 g''.
\end{align*}
We replace $g'= {\rr}/{W_2}$ and we sort the different terms
\begin{align*}
\MP\MM h & = \lambda ^2h+
W_2\Bigl(\frac{\rr}{W_2}\Bigr)'''
 + 4 W_2'\Bigl(\frac{\rr}{W_2}\Bigr)''
+\Bigl(5W_2''-\lambda W_1-W_2-\frac \omega3\Qo^4 W_2\Bigr)\Bigl(\frac{\rr}{W_2}\Bigr)'\\
&\quad 
+2\Bigl(W_2'''-\lambda W_1'-W_2'-\frac \omega3\Qo^4 W_2'\Bigr)\frac{\rr}{W_2}. 
\end{align*}
Expanding the derivatives and using $(1/{W_2})' = - {W_2'}/W_2^2$, we obtain
\begin{align*}
&\MP\MM h =\lambda ^2h 
+ \rr''' -3\frac{W_2'}{W_2} \rr'' + 3 W_2\Bigl(\frac1{W_2}\Bigr)'' \rr'
+ W_2\Bigl(\frac1{W_2}\Bigr)''' \rr\\
&\qquad + 4\frac{W_2'}{W_2} \rr'' - 8\frac{(W_2')^2}{W_2^2} \rr'+4 W_2'\Bigl(\frac1{W_2}\Bigr)'' \rr
+\frac1{W_2}(5W_2''- \lambda W_1-W_2-\frac \omega3\Qo^4W_2) \rr'\\
&\qquad -\frac{W_2'}{W_2^2}(5W_2''- \lambda W_1-W_2-\frac \omega3\Qo^4W_2) \rr
+\frac2{W_2}(W_2'''- \lambda W_1'-W_2' -\frac \omega3\Qo^4 W_2') \rr .\end{align*}
Using
\[
\Bigl(\frac1{W_2}\Bigr)'' = -\frac{W_2''}{W_2^2}+2\frac{(W_2')^2}{W_2^3},\qquad
\Bigl(\frac1{W_2}\Bigr)''' =-\frac{W_2'''}{W_2^2}+6\frac{W_2''W_2'}{W_2^3}
-6\frac{(W_2')^3}{W_2^4}
\]
we find
\[
\MP\MM h =\lambda ^2h+\rr'''+\frac{W_2'}{W_2} \rr''+ J_1\rr' +J_0 \rr
\]
where 
\begin{align*}
J_1 & = -1 -\lambda
\frac {W_1}{W_2} + 2\frac{W_2''}{W_2} - 2\frac{(W_2')^2}{W_2^2} 
-\frac \omega3\Qo^4\\
J_0 &=\frac{W_2'''}{W_2}-3\frac{W_2''W_2'}{W_2^2}
+ 2\frac{(W_2')^3}{W_2^3} - 2\lambda \frac{W_1'}{W_2}-\frac{W_2'}{W_2}
+\lambda\frac{W_1 W_2'}{W_2^2} -\frac \omega3\Qo^4\frac{W_2'}{W_2}.
\end{align*}
Thus, recalling that $\cU h=\rr$,
\begin{align*}
\cU \MP\MM h 
& =\lambda^2 r + \Bigl(\py -\frac{W_2'}{W_2}\Bigr)\Bigl(\rr'''+\frac{W_2'}{W_2} \rr''+ J_1\rr' +J_0 \rr\Bigr) \\
& = \rr''''-2 \rr'' + K_2 \rr'' +K_1 \rr'+ K_0 \rr +\rr ,
\end{align*}
where we have defined
\[
K_2=2+\Bigl(\frac{W_2'}{W_2}\Bigr)'-\left(\frac{W_2'}{W_2}\right)^2 + J_1 ,\
K_1=J_1' -\frac{W_2'}{W_2}J_1 + J_0,\
K_0=J_0'-\frac{W_2'}{W_2}J_0 +\lambda^2 -1.
\]
Replacing $J_0$ and $J_1$ in the above definitions of $K_2$, $K_1$
and $K_0$, we find
\begin{align*}
K_2 &= 1 -\lambda \frac {W_1}{W_2} + 3\frac{W_2''}{W_2}
 - 4\frac{(W_2')^2}{W_2^2}-\frac \omega3\Qo^4,\\
K_1 &= -3\lambda \frac{W_1'}{W_2} + 3\lambda\frac{W_1W_2'}{W_2^2}+3\frac{W_2'''}{W_2}-11\frac{W_2'W_2''}{W_2^2}
+ 8\frac{(W_2')^3}{W_2^3}-\frac \omega3(\Qo^4)' ,
\end{align*}
and
\begin{align*}
K_0 & = -2\lambda\frac{W_1''}{W_2}+5\lambda\frac{W_1'W_2'}{W_2^2}
+2\frac{(W_2')^2}{W_2^2}-3\lambda \frac{W_1(W_2')^2}{W_2^3}
+\lambda\frac{W_1W_2''}{W_2^2}- \frac{W_2''}{W_2}\\
&\quad +\frac{W_2''''}{W_2} - 5\frac{W_2'}{W_2}\frac{W_2'''}{W_2}-3\frac{(W_2'')^2}{W_2^2}
+15\frac{W_2''}{W_2}\frac{(W_2')^2}{W_2^2} - 8\frac{(W_2')^4}{W_2^4}\\
&\quad -\frac \omega3(\Qo^4)'\frac{W_2'}{W_2}
-\frac \omega3\Qo^4\frac{W_2''}{W_2}
+\frac {2\omega}3\Qo^4\frac{(W_2')^2}{W_2^2}
+\lambda^2-1.
\end{align*}
Now, we prove the decay properties of $K_2$, $K_1$, $K_0$.
Writing
$ {W_1}/{W_2} -1 = (W_1-W_2)/W_2$
and using~\ref{it:bW} and \ref{it:aW} of Lemma~\ref{LE:VW}, it holds, for any $k\geq 1$, on $\RR$,
\begin{equation}\label{eq:eW}
\biggl|\frac {W_1}{W_2} -1\biggr|
+\bigg|\py^k\biggl(\frac {W_1}{W_2}\biggr)\bigg|
\lesssim\omega \ee^{-(\kappa-\alpha)|y|}.
\end{equation}
Moreover, by \ref{it:mu}-\ref{it:dV} of Lemma~\ref{LE:VW}, it holds
\[
W_2'' = W_2 -\lambda W_1 -\omega\Qo^4 W_2
= \alpha^2 W_2 - \af W_2 ,\quad \af = \lambda\frac{W_1-W_2}{W_2} +\omega\Qo^4 .
\]
By the decay properties of $\Qo$
and \eqref{eq:eW}, 
$|\py^k\af|\lesssim \omega \ee^{-(\kappa-\alpha)|y|}$, and
for any $k\geq 1$,
\[
\biggl|\frac{W_2''}{W_2} -\alpha^2\biggr|+\biggl|\py^k\biggl(\frac{W_2''}{W_2}\biggr)\biggr|\lesssim \omega \ee^{-(\kappa-\alpha)|y|}.
\]
Multiplying $W_2''= \alpha^2 W_2 - \af W_2 $ by $W_2'$ and integrating
on $[y,+\infty)$, we find, for $y>0$,
\begin{equation}\label{eq:W2}
(W_2')^2 =\alpha^2 W_2^2 + 2\int_y^{+\infty} \af W_2'W_2.
\end{equation}
By the decay properties of $\af$, $W_2$, $W_2'$
and \ref{it:aW} of Lemma~\ref{LE:VW}, we obtain, for $y>0$,
\[
\left|\frac{(W_2')^2}{W_2^2} -\alpha^2\right| 
\lesssim \omega^2 \ee^{-(\kappa-\alpha)y}.
\]
Using $W_2'<0$ for $y>0$ large, we obtain, for $y>0$,
$\left| {W_2'}/{W_2} +\alpha \right| 
\lesssim \omega \ee^{-(\kappa-\alpha)y}$.
For any $k\geq 1$, one has $\py^{k+2}W_2 = \alpha^2 \py^kW_2- \py^k(\af W_2)$,
and so, by induction on $k\geq 1$, for $y>0$,
\[
\biggl|\frac{\py^kW_2}{W_2}-(-\alpha)^{k}\biggr|
\lesssim \omega \ee^{-(\kappa-\alpha)y}.
\]
Proceeding similarly for $W_1$ using \eqref{eq:eW}, we obtain, for all $k\geq 0$,
for $y\geq 0$,
\begin{equation}\label{eq:qW}
\biggl|\frac{\py^kW_2}{W_2}-(-\alpha)^{k}\biggr|+
\biggl|\frac{\py^kW_1}{W_2}-(-\alpha)^{k}\biggr|
\lesssim \omega \ee^{-(\kappa-\alpha) y}.
\end{equation}
By the expression of $K_2$,
\eqref{eq:qW} and the relation $\lambda= 1-\alpha^2$, we obtain
for $y>0$,
\begin{align*}
|K_2 -( 1-\lambda + 3 \alpha^2 - 4 \alpha^2 )|
&= |K_2|\lesssim \omega \ee^{-(\kappa-\alpha)y},\\
|K_1 - (3 \lambda \alpha - 3 \lambda \alpha - 3\alpha^3
+11 \alpha^3 - 8 \alpha^3)|
&=|K_1|\lesssim \omega \ee^{-(\kappa-\alpha)y}.
\end{align*}
Using $
-2 \lambda \alpha^2 + 5 \lambda \alpha^2 +2 \alpha^2
-3 \lambda \alpha^2 + \lambda \alpha^2 - \alpha^2
+ \alpha^4 - 5 \alpha^4 -3 \alpha^4 + 15 \alpha^4 - 8 \alpha^4 + \lambda^2 -1 =0$,
we also find
$|K_0|\lesssim \omega \ee^{-(\kappa-\alpha)|y|}$ for $y> 0$.
The estimates on the derivatives of $K_2$, $K_1$ and
$K_0$ are obtained similarly.
\end{proof}

We establish a virial identity for the fourth order operator $K$.
\begin{lemma}\label{LE:v2}
The operator $\cK$ being defined in Lemma~\ref{LE:sf}, it holds 
for any  $h\in \mS(\RR)$ 
\begin{align*}
\int (2 yh' + h )\cK h
=4\int( h'')^2 +4\int(h')^2 + \int Y_1 ( h')^2 +\int Y_0 h^2
\end{align*}
where the functions
$
Y_1 = - 2K_2 - yK_2' + 2yK_1$ and 
$Y_0 =\frac 12\left( K_2'' - K_1' - 2 y K_0'\right)$
satisfy, for all $k\geq 0$,
$
|\py^k Y_1|+|\py^k Y_0| \lesssim \omega e^{-|y|}
$
on $\RR$.
\end{lemma}
\begin{proof}
By integration by parts, we compute
\begin{align*}
&\int (2 yh' + h)h'''' = 4\int( h'')^2 ,\quad
\int (2 yh' + h)( - 2h'') = 4\int(h')^2,\\
&\int (2 yh' + h)K_2 h'' = - \int (2 K_2+ y K_2' )( h')^2 +\frac 12\int K_2'' h ^2,\displaybreak[0]\\
&\int (2 yh' + h)K_1 h' = 2\int y K_1( h ')^2 -\frac 12\int K_1' h ^2,\quad
\int (2 yh' + h)(K_0+1) h = -\int y K_0' h ^2,
\end{align*}
and the identity follows.
The estimates on $Y_1$ and $Y_0$ follow
from \eqref{eq:Kj}.
\end{proof}

For a nonzero function $Y$ satisfying $\int (1+y^2) |Y(y)|dy<+\infty$, 
it is known by~\cite{Si76} that for $\epsilon>0$ small, the positivity of the quadratic form 
$\int(h')^2 +\epsilon \int Y h^2$
is equivalent to the condition $\int Y> 0$.
In this direction, we give here a weak form of the main result in~\cite{Si76} and~\cite[Theorem~XIII.110]{RSBK}.

\begin{lemma}\label{LE:BS}
Let $c\in (0,1)$. If $Y:\RR\to\RR$ is a function such that
\[
|Y(x)|\leq \ee^{-|x|}\mbox{ on $\RR$ and }
\int Y>0,
\]
then, for any $h\in H^1$,
\[
\int \ee^{-c|x|} h^2 
\leq \frac 4{c\int Y}\int Y h^2 
+\frac {64}{c^2\left(\int Y\right)^2}\int(h')^2 .
\]
\end{lemma}
\begin{proof}
For $x,y\in\RR$,
$
h^2(x) =h^2(y) - 2\int_x^y h'h.
$
We multiply by $Y(y)$ and integrate in $y$
\[
\left(\int Y\right) h^2(x) =\int Y h^2 
- 2\int_x^\infty Y(y)\left(\int_x^y h' h\right) d y 
 + 2\int_{-\infty}^x Y(y)\left(\int_y^x h'h\right) d y.
\]
We multiply by $ \ee^{-c|x|}$ and integrate in $x$,
using $\int \ee^{-c|x|} d x=2/c$,
\begin{align*}
 \left(\int Y\right)\int \ee^{-c|x|} h^2 & = (2/c)\int Y h^2
- 2\int \ee^{-c|x|}\int_x^\infty Y(y)\left(\int_x^y h'h\right) d y d x\\
&\quad + 2\int \ee^{-c|x|}\int_{-\infty}^x Y(y)\left(\int_y^x h' h\right)d y d x.
\end{align*}
By the Fubini Theorem,
\[
\int \ee^{-c|x|}\int_x^\infty Y(y)\left(\int_x^y h'h\right) d y d x
 =\int\left(\int_{-\infty}^z \ee^{-c|x|} d x\right)\left(\int_z^{\infty} Y \right) h'(z) h(z) d z.
\]
Observe that 
$
\int_{-\infty}^z \ee^{-c|x|} d x\leq 2/c$ if $z>0$ and 
$\int_{-\infty}^z \ee^{-c|x|} d x
\leq \ee^{-c|z|}/c$ if $z<0$.
Similarly,
by the assumption on $Y$, 
$\left|\int_z^{\infty} Y\right| \leq \ee^{- z}$ if $z>0$ and 
$\left|\int_z^{\infty} Y\right|\leq 2 $ if $z<0$.
Thus, for all $z\in\RR$,
\[
\left|\left(\int_{-\infty}^z \ee^{-c|x|} d x\right)
\left(\int_z^{\infty} Y\right)\right|
\leq (2/c) \ee^{-c|z|}.
\]
We obtain by the Cauchy-Schwarz inequality,
\begin{align*}
\left|\int \ee^{-c|x|}\int_x^\infty Y(y)\left(\int_x^y h'\right)d y d x\right|
\leq (2/c)\left(\int \ee^{-c|x|}(h')^2 \right)^{\frac 12}
\left(\int \ee^{-c|x|} h^2 \right)^{\frac 12}.
\end{align*}
Using a similar estimate for $\int \ee^{-c|x|}\int_{-\infty}^x Y(y)\bigl(\int_y^x h'\bigr) d y d x$, we deduce
\begin{align*}
 \left(\int Y\right)\int \ee^{-c|x|} h^2 
 &\leq (2/c)\int Y h^2 
+(8/c)\left(\int \ee^{-c|x|}(h')^2 \right)^{\frac 12}
\left(\int \ee^{-c|x|}h^2\right)^{\frac 12}
\\ &\leq (2/c)\int Y h^2 
+\frac {32 }{c^2\int Y}\int\ee^{-c|x|}(h')^2 
+\frac 12 \left(\int Y\right)\int \ee^{-c|x|}h^2 ,
\end{align*}
which implies the desired estimate.
\end{proof}

Now, we show the repulsive nature of the second transformed problem by checking that the
integral of the small potential $Y_0$ in Lemma \ref{LE:v2} is indeed positive
for $\omega>0$ small.

\begin{lemma}\label{LE:IK}
For $\omega>0$ small,
$
\int Y_0 =\frac{32}9 \omega + O(\omega^2).
$
\end{lemma}
\begin{proof}
By \eqref{eq:Kj}, we have $\int K_2''= \int K_1' = 0$,
and so
$\int Y_0 = - \int y K_0' = \int K_0$.
From the expression of $K_0$ in the proof of Lemma \ref{LE:sf}, we decompose
\[
K_0 = -2 \biggl(\frac{W_1''}{W_2}-\alpha^2\biggr) 
+ \biggl(\frac{W_2''''}{W_2}-\alpha^4\biggr)
+ \tilde K_0
\]
where $ {W_1''}/{W_2}-\alpha^2$ and $ {W_2''''}/{W_2}-\alpha^4$
are integrable thanks to \eqref{eq:qW} and
\begin{align*}
\tilde K_0
& = 
2\alpha^2 \frac{W_1''}{W_2} 
+5\lambda\frac{W_1'W_2'}{W_2^2}
+2\frac{(W_2')^2}{W_2^2}-3\lambda \frac{W_1(W_2')^2}{W_2^3}
+\lambda\biggl(\frac{W_1}{W_2}-1\biggr)\frac{W_2''}{W_2}
\\
&\quad
-\alpha^2 \frac{W_2''}{W_2}- 5\frac{W_2'}{W_2}\frac{W_2'''}{W_2}-3\frac{(W_2'')^2}{W_2^2}
+15\frac{W_2''}{W_2}\frac{(W_2')^2}{W_2^2} - 8\frac{(W_2')^4}{W_2^4}\\
&\quad -\frac \omega3(\Qo^4)'\frac{W_2'}{W_2}
-\frac \omega3\Qo^4\frac{W_2''}{W_2}
+\frac {2\omega}3\Qo^4\frac{(W_2')^2}{W_2^2}
-4\alpha^2 + 2 \alpha^4.
\end{align*}
We now prove the decay property
$|\tilde K_0| \lesssim \omega^2 \ee^{-(\kappa-\alpha)|y|}$ on $\RR$, which implies
$\int |\tilde K_0| \lesssim \omega^2$, by examining the 
asymptotic properties of each term in the expression of $\tilde K_0$
using the estimate \eqref{eq:qW}.
For example, for the first two terms, using \eqref{eq:qW}, we have
for $y>0$,
\[
\left| 2\alpha^2 \frac{W_1''}{W_2} - 2\alpha^4\right|
\lesssim \omega^2 \left| \frac{W_1''}{W_2} - \alpha^2\right|
\lesssim \omega^3 \ee^{-(\kappa-\alpha)y},
\]
and
\[
\left|5\lambda\frac{W_1'W_2'}{W_2^2} - 5 \lambda\alpha^2\right|
\lesssim 
 \left|\frac{W_1'}{W_2} + \alpha\right|\cdot \left|\frac{W_2'}{W_2} \right|
+\omega \left|\frac{W_2'}{W_2} + \alpha\right|
\lesssim 
\omega^2 \ee^{-(\kappa-\alpha)y}.
\]
The other terms are treated similarly using \eqref{eq:qW} and the estimate
$|\alpha|\lesssim \omega$.
The identity 
$2\alpha^4+5 \lambda \alpha^2 + 2\alpha^2 - 3 \lambda \alpha^2
+\lambda \alpha^2-\alpha^2 -5 \alpha^4 -3 \alpha^4+15 \alpha^4 -8 \alpha^4
 -4\alpha^2 + 2 \alpha^4=0$, deduced from $\lambda=1-\alpha^2$ then implies that
 the limit of $\tilde K_0$ at $+\infty$ is zero and the desired decay property for $\tilde K_0$
follows. 
Then,
\[
\int \biggl(\frac{W_1''}{W_2}-\alpha^2\biggr)
=\int \biggl(\frac{W_1'}{W_2}\biggr)' 
+\int \biggl( \frac{W_1'W_2'}{W_2^2}-\alpha^2\biggr).
\]
Using \eqref{eq:qW} for $k=1$, one has
\[
 \int \biggl(\frac{W_1'}{W_2}\biggr)'
= \bigg[ \frac{W_1'}{W_2} \bigg]_{-\infty}^{+\infty}
= -2\alpha.
\]
Moreover,
\[
\frac{W_1'W_2'}{W_2^2}-\alpha^2
=\frac{W_1'}{W_2}\biggl(\frac{ W_2'}{W_2}+\alpha\biggr)
-\alpha \biggl(\frac{W_1'}{W_2}+\alpha\biggr)
\]
and by \eqref{eq:qW},
\[
\bigg|\int\biggl(\frac{W_1'W_2'}{W_2^2}-\alpha^2\biggr) \bigg|=
2\bigg| \int_{y>0} \biggl(\frac{W_1'W_2'}{W_2^2}-\alpha^2\biggr) \bigg| \lesssim \omega^2 .
\]
Therefore, 
\[
-2 \int \biggl(\frac{W_1''}{W_2}-\alpha^2\biggr) 
= 4 \alpha + O(\omega^2) .
\]
Lastly, using \eqref{eq:qW},
\begin{align*}
\int_{y>0}
\biggl(\frac{W_2''''}{W_2}-\alpha^4\biggr)
& = 
\biggl[\frac{W_2'''}{W_2}\biggr]_0^{\infty}
+\int_{y>0} \frac{W_2'}{W_2}\biggl(\frac{W_2'''}{W_2}+\alpha^3 \biggr)
-\alpha^3 \int_{y>0} \biggl( \frac{W_2'}{W_2}+\alpha \biggr)
= O(\omega^2).
\end{align*}
Using \ref{it:mu} of Lemma \ref{LE:VW}, one obtains
$\int Y_0=4 \alpha + O(\omega^2) = \frac{32}9\omega + O(\omega^2)$.
\end{proof}

\begin{lemma}\label{LE:nZ}
For $\omega>0$ small,
if $(\tilde \mu,Z)\in \RR\times H^4(\RR)$ solves
$K Z = \tilde\mu Z$
then $Z\equiv0$.
\end{lemma}
\begin{proof}
Let $(\tilde\mu, Z)\in \RR\times H^4(\RR)$ solve
$KZ = \tilde\mu Z$ on $\RR$.
Using $\int (2yZ' + Z) Z =0$ and Lemma \ref{LE:v2}, we obtain
\[
4 \int (Z'')^2 + 4 \int (Z')^2 = - \int Y_1 (Z')^2
-\int Y_0 Z^2 .
\]
The computations in Lemma \ref{LE:v2} are formal but are easily justified 
for $Z\in H^4(\RR)$, using cut-off functions.
By Lemma~\ref{LE:v2}, we have $\|Y_1\|_{L^\infty} \lesssim \omega$ and
$|Y_0(y)|\leq C \omega \ee^{-|y|}$, for some $C>0$.
Using Lemma~\ref{LE:BS} (with $c=1$, $Y=Y_0/C\omega$ and $h=Z$) and Lemma \ref{LE:IK} we deduce
\[
 - \int Y_0 Z^2 \lesssim \omega \int (Z')^2.
\]
For $\omega>0$ small enough, we obtain
$\int (Z'')^2 + \int (Z')^2 =0$, which proves
$Z=0$ on $\RR$.
\end{proof}

\begin{lemma}\label{LE:un}
For $\omega>0$ small, the only solutions
$(\tilde \lambda,\tilde V_1,\tilde V_2)\in [0,+\infty)\times H^2(\RR)\times H^2(\RR)$ 
 of the eigenvalue problem \eqref{eq:VV} are
$(\mu,0,0)$ for any $\mu\in [0,+\infty)$, $(0,a\Qo',b\Qo)$ for any $a,b\in \RR$
and $(\lambda,cV_1,cV_2)$ for any $c\in \RR$, where 
$(\lambda,V_1,V_2)$ is constructed in Lemma~\ref{LE:VW}.
\end{lemma}
\begin{proof}
By Lemma \ref{LE:LM},
the relation $\LP\LM \tilde V_2=\tilde \lambda^2 \tilde V_2$ implies 
that the function $\tilde W_2 = S^2 \tilde V_2$
satisfies $\MP\MM\tilde W_2 = \tilde \lambda^2 \tilde W_2$
and then by Lemma \ref{LE:sf}, $\tilde Z_2 = U \tilde W_2$ satisfies
$K \tilde Z_2 = \tilde \mu \tilde Z_2$.
By Lemma \ref{LE:nZ}, $\tilde Z_2=0$. Thus, there exists $c\in \RR$
such that $\tilde W_2 = cW_2$.
We also deduce that 
$\tilde V_2 = c V_2 + (b+dx) \Qo$ for some $b,d\in \RR$.
Then, $\LP\LM \tilde V_2 = c \lambda^2 V_2$.
If $c\neq 0$ then $b=d=0$, $\tilde \lambda=\lambda$, $\tilde V_2=cV_2$
and $\tilde V_1=cV_1$.
If $c=0$ then $b=d=0$ or $\tilde \lambda=0$.
In the latter case, we obtain $d=0$ and $\tilde V_1 = a\Qo'$
for some $a\in \RR$.
\end{proof}
\begin{remark}
The uniqueness result given in Lemma~\ref{LE:un} holds without symmetry assumption.
To prove the uniqueness only among even functions, Lemmas \ref{LE:BS} and \ref{LE:IK} 
are not required. Indeed, the auxiliary pair of functions $(Z_1,Z_2)$ is then odd and
the positivity of the operator in Lemma \ref{LE:v2} for odd functions
 follows directly from the smallness of the potentials $Y_1$ and $Y_0$;
 see for example \cite[Claim 4.1]{KMM1}.
The same remark will apply to the evolution problem in \S\ref{se:10}.
Since the pair of functions $(z_1,z_2)$ defined after two transformations is
odd, the use of Lemmas \ref{LE:BS} and \ref{LE:IK} in the proof of Lemma \ref{LE:vz}
is not necessary.

Moreover, as pointed out by a referee, the fact that there cannot be other eigenvalue of \eqref{eq:VV} 
is also a consequence of the explicitly known spectrum of the 
linearized problem in the integrable case combined with perturbation arguments.
However, proving 
that the small potential $Y_0$ is repulsive in the sense of Lemmas \ref{LE:BS} and \ref{LE:IK}
has the advantage of showing that the extension of the proof of Theorem \ref{TH:as} to the case without symmetry
using the method of the present paper should not
require additional spectral arguments (see also \cite{Ma22}).
\end{remark}

\section{Rescaled decomposition}\label{se:04}
Define the operators
\[\Lambda =\frac 12 + \frac 12 y\py,
\quad\Lambda^* = -\frac 12 y\py,\quad \Lo =\Lambda +\omega\PO.
\]
Let $\DD=\RR\times(0,+\infty)$.
For $\varphi\in H^1(\RR)$ and $\Pi =(\gamma,\omega)\in\DD$,
define $\zeta[\varphi,\Pi]:\RR\to\CC$ by
\[
\zeta[\varphi,\Pi](y)=
\frac1{\sqrt{\omega}} 
\exp\left(-\ii\gamma\right)\varphi\Bigl(\frac{y}{\sqrt{\omega}}\Bigr).
\]
We start with a standard decomposition result for time-independent functions
close to a solitary wave.
\begin{lemma}\label{LE:mi}
For any~$\oz>0$ and any~$\varepsilon>0$, there exists~$\delta>0$
such that for all even function $\varphi\in H^1(\RR)$ with
$\|\varphi -\poz\|_{H^1(\RR)} <\delta$,
there exists a unique $\Pi=(\gamma,\omega)\in\DD$
such that $|\gamma|+|\omega-\oz|<\varepsilon$
and $u= \zeta[\varphi,\Pi]-\Qo$ satisfies $\|u\|_{H^1(\RR)}<\varepsilon$
and
\begin{equation}\label{eq:or}
\langle u,\ii\Lo\Qo\rangle=\langle u,\Qo\rangle=0.
\end{equation}
\end{lemma}
\begin{proof}
For $\varphi\in H^1(\RR)$ and $\Pi=(\gamma,\omega)\in\DD$, we set $u=u[\varphi,\Pi]
=\zeta[\varphi,\Pi]-\Qo$ and
\[
\Upsilon[\varphi,\Pi] =\begin{pmatrix}
\langle u,\ii\Lo\Qo\rangle\\
\omega \langle u,\Qo\rangle
\end{pmatrix}
=\sqrt{\omega}\begin{pmatrix}
\langle \varphi-\ee^{\ii \gamma} \po,\ii\ee^{\ii \gamma}\PO\po\rangle \\
\langle \varphi-\ee^{\ii \gamma}\po,\ee^{\ii \gamma}\po\rangle
\end{pmatrix}.
\]
Set $\Pi_0=(0,\omega_0)\in\DD$.
Note that $\zeta[\poz,\Pi_0]=\Qoz$, $u[\poz,\Pi_0]=0$,
$\Upsilon[\poz,\Pi_0]=0$.
We check 
$ \partial_\Pi\Upsilon[\poz,\Pi_0]=
-c_0 \textnormal{I}_2$ where $c_0=\frac 12\sqrt{\oz}\PO(\|\po\|^2)_{|\omega=\oz}>0$.
The partial derivative $\partial_\Pi\Upsilon$
being invertible at the point $(\poz,\Pi_0)$,
it follows from the implicit function Theorem that there exists
a neighborhood $\cV$ of $\poz$ in $H^1(\RR)$ and a smooth map
$\Pi_1 :\varphi\in\cV\mapsto\Pi_1[\varphi]\in\DD$
such that for all $\varphi\in\cV$,
$\Upsilon[\varphi,\Pi]=0$ if and only if $\Pi =\Pi_1[\varphi]$
and $|\Pi_1[\varphi]-\Pi_0|_\infty\leq C(\oz) \|\varphi-\poz\|_{H^1}$.
\end{proof}
\begin{lemma}\label{LE:Ty}
For any $p\geq 1$ integer and $a=a_1+\ii a_2$ with $|a|<1$, it holds
\begin{align*}
|1+a|^{2p} (1+a)
&= 1+ (2p+1) a_1+\ii a_2 + (2p+1)pa_1^2+pa_2^2
+ \ii 2 pa_1a_2\\
&\quad +\tfrac13(4p^2-1)p a_1^3 + (2p-1)p a_1a_2^2
+\ii ((2p-1)pa_1^2a_2+pa_2^3) + O(|a|^4).
\end{align*}
In particular, setting $\fo(\psi) =|\psi|^2 \psi +\omega|\psi|^4 \psi$
and
\[
q_1 =\Re\left\{\fo(\Qo+u)-\fo(\Qo)-\fo'(\Qo)u_1\right\},\
q_2 =\Im\left\{\fo(\Qo+u)-\ii ({\fo(\Qo)}/{\Qo})u_2\right\}
\]
it holds for $u=u_1+\ii u_2$ with $|u|<1$,
\begin{align*}
q_1 & = \Qo(3+10\omega\Qo^2)u_1^2 + \Qo(1+2\omega\Qo^2)u_2^2\\
&\quad +(1+10\omega \Qo^2)u_1^3 + (1+ 6\omega\Qo^2) u_1u_2^2+O(|u|^4)\\
q_2& = 2\Qo(1+2\omega\Qo^2)u_1u_2
+ (1+6\omega \Qo^2)u_1^2u_2+(1+2\omega\Qo^2)u_2^3 +O(|u|^4).
\end{align*}
\end{lemma}
\begin{proof}
Expanding
\begin{align*}
&|1+a|^{2p} = ((1+a_1)^2+a_2^2)^p
=(1+2a_1+a_1^2+a_2^2)^p\\
&\quad = 1+2pa_1+(2p-1)pa_1^2+pa_2^2
+\bigl( \tfrac 43 p -\tfrac23\bigr)p(p-1)a_1^3
+2p(p-1)a_1a_2^2+O(|a|^4)
\end{align*}
and multiplying by $(1+a_1+\ii a_2)$, we get the first relation.
Then, we apply it to $p=1$ and $p=2$
with $a=u/\Qo$ to find the expansion of $\fo(\Qo+u)$
up to order $3$ in $u$.
\end{proof}

We introduce the functions
\[
\nu(y)=\sech\Bigl(\frac{y}{10}\Bigr),\quad
\rho(y)=\sech\Bigl(\frac{\oz}{10} y\Bigr).
\]
We now prove a global decomposition result based on the stability
property.

\begin{lemma}\label{LE:tm}
For any~$\oz>0$ small and any~$\varepsilon>0$, there exists~$\delta>0$
such that for any even function $\psi_0\in H^1(\RR)$ with
$\|\psi_0 -\poz\|_{H^1(\RR)} <\delta$,
there exists a unique $\cC^1$ function
$\Pi : [0,+\infty)\mapsto(\gamma,\omega)\in\DD$
such that if $\psi$ is the solution of \eqref{eq:NL}, 
\[
u(s )=\zeta[\psi(\tau(s)),\Pi(s)] -Q_{\omega(s)}
\quad \mbox{where}\quad
\tau(s) =\int_0^{s}\frac{ds'}{\omega(s')},
\]
then the following properties hold, for all $s\in[0,+\infty)$.
\begin{enumerate}
\item\label{it:pa} \emph{Stability:}
$|\omega-\oz|+\|u\|_{H^1}\leq\varepsilon$.
\item \emph{Orthogonality relations:}
$u$ satisfies~\eqref{eq:or}.
\item \emph{Equation:} $u=u_1+\ii u_2$ satisfies
\begin{equation}\label{eq:uu}
\begin{cases}
\dot u_1 =\LM u_2 +\mu_2 + p_2 - q_2\\
\dot u_2 = -\LP u_1 -\mu_1 - p_1 + q_1
\end{cases}
\end{equation}
where
\begin{align*}
&m_\gamma= \dot\gamma-1,\quad
m_\omega=\frac{\dot\omega}{\omega},\quad
\mu_1 = m_\gamma\Qo,\quad \mu_2 = - m_\omega\Lo\Qo,\\
&p_1 = m_\gamma u_1+ m_\omega\Lambda u_2, \quad
p_2 = m_\gamma u_2- m_\omega\Lambda u_1.
\end{align*}
\item\label{it:Xi} \emph{Equation of the parameters:}
$|m_\gamma|+|m_\omega|\lesssim\|\nu u\|^2$.
\end{enumerate}
\end{lemma}
\begin{remark}
For a function $g$ depending on $s$, we will denote $\dot g=\ps g$.
\end{remark}
\begin{proof}
By Lemma~\ref{LE:mi} applied to $\psi_0$ there exists a unique
$\Pi^{\rm in}=(\gamma^{\rm in},\omega^{\rm in})$ close to $(0,\oz)$ such that~\eqref{eq:or} is 
satisfied for $s=0$ with $u^{\rm in}=\zeta[\psi^{\rm in},\Pi^{\rm in}]-\Qoz$.
Then, we assume that
there exists a $\cC^1$ function $\Pi=(\gamma,\omega)$
 on $[0,\bar s]$ for some small
$\bar s>0$, with $\Pi(0)=\Pi^{\rm in}$ and such that \eqref{eq:or} hold on $[0,\bar s]$, and 
we derive the equations of $\gamma$, $\omega$ and $u$ on $[0,\bar s]$.
By the definition of $u$,
\[
\psi(t,x) = \ee^{\ii\gamma}\varphi(t,x)
\ \mbox{where}\
\varphi(\tau(s),x) =\sqrt{\omega} P(s,\sqrt{\omega}x)\ \mbox{and}\
P(s,y)=Q_{\omega(s)}(y)+u(s,y).
\]
From the equation of $\psi$, we compute
the equations of $\varphi$, $P$ and $u$. One obtains
\[
\ii\pt\varphi+\px^2\varphi+|\varphi|^2\varphi+|\varphi|^4\varphi
-\frac{d\gamma}{dt}\varphi=0\]
and using $\dot \tau=1/\omega$,
\[
\ii\dot P +\py^2 P -P +\fo(P) 
 +\ii\frac{\dot\omega}{\omega}\Lambda P
-\left(\dot\gamma-1\right)P=0.
\]
Using 
$\Qo''-\Qo=-\fo(\Qo)$, $\dot Q_\omega= \dot\omega \PO\Qo$ and
the definition of $\Lo\Qo$, we obtain
\[
\ii\dot u+\py^2 u - u +\fo(\Qo+u)-\fo(\Qo) 
 +\ii\frac{\dot\omega}{\omega}\Lo\Qo
-\left(\dot\gamma-1\right)\Qo 
+\ii\frac{\dot\omega}{\omega}\Lambda u
-\left(\dot\gamma-1\right)u=0.
\]
The system \eqref{eq:uu} for $(u_1,u_2)$ follows from 
the definitions of $\LP$ and $\LM$ and the notation of the lemma.
We now use the first orthogonality relation
$\langle u,\ii\Lo\Qo\rangle=\langle u_2,\Lo\Qo\rangle=0$.
By~\eqref{eq:uu}, $\LP(\Lo\Qo)=- \Qo$ 
(obtained by direct computation or by differentiating the equation of $\po$
with respect to $\omega$)
and the orthogonality relation $\langle u_1,\Qo\rangle=0$, we get
\begin{align*}
0&=\frac{d}{ds}\langle u_2,\Lo\Qo\rangle
=\langle\dot u_2,\Lo\Qo\rangle
+ m_\omega\langle u_2,\omega\PO(\Lo\Qo)\rangle\\
&=\langle -\LP u_1 -\mu_1 - p_1 + q_1,\Lo\Qo\rangle
+ m_\omega\langle u_2,\omega\PO(\Lo\Qo)\rangle\\
&= -m_\gamma (c_\omega + \langle u,\Lo\Qo\rangle )
+ m_\omega \langle u,\ii(\Lo-\tfrac 12)(\Lo\Qo)\rangle
+\langle q_1,\Lo\Qo\rangle
\end{align*}
where $c_\omega = \langle \Qo,\Lo\Qo\rangle
=\frac 12\sqrt{\omega}\PO(\|\po\|^2)\gtrsim 1$ for $\omega>0$ small.
Similarly, the second orthogonality relation
$\langle u,\Qo\rangle=\langle u_1,\Qo\rangle=0$ and the relation
$\LM\Qo=0$ yield
\begin{align*}
0&=\frac{d}{ds}\langle u_1,\Qo\rangle
=\langle\dot u_1,\Qo\rangle + m_\omega\langle u_1,\omega\PO\Qo\rangle\\
&=\langle\LM u_2 +\mu_2 + p_2 - q_2,\Qo\rangle
+ m_\omega\langle u_1,\omega\PO\Qo\rangle\\
&=- m_\omega (c_\omega -\langle u,(\Lo-\tfrac12)\Qo\rangle )
+m_\gamma \langle u,\ii \Qo\rangle-\langle q_2,\Qo\rangle.
\end{align*}
These two identities, together with $\dot \tau=1/\omega$, are written
under the form
\begin{equation}\label{eq:sp}
\begin{pmatrix}
 1+ j_{1,1} & j_{1,2} &0\\
 j_{2,1} & 1+ j_{2,2} &0\\
0&0&1
\end{pmatrix}
\begin{pmatrix} m_\gamma\\ m_\omega\\ \dot\tau\end{pmatrix} =\begin{pmatrix}
k_1\\
k_2\\
k_3
\end{pmatrix}
\end{equation}
where $k_3=1/\omega$ and
\begin{align*}
&j_{1,1}=(1/c_\omega)\langle u,\Lo\Qo\rangle, \ 
j_{1,2}=-(1/c_\omega) \langle u,\ii(\Lo-\tfrac 12)(\Lo\Qo)\rangle,\
k_1=(1/c_\omega)\langle q_1,\Lo\Qo\rangle,\\
&j_{2,1}=-(1/c_\omega) \langle u,\ii\Qo)\rangle,\
j_{2,2}=-(1/c_\omega)\langle u,(\Lo-\tfrac 12)\Qo\rangle,\ 
k_2=-(1/c_\omega)\langle q_2,\Lo\Qo\rangle.
\end{align*}
It is easily checked using $\psi\in \cC^1([0,+\infty),H^{-1}(\RR))$
and the definition of $u$ 
that the functions $j_{l,n}$ and $k_l$ are locally Lipschitz in 
$(\gamma,\omega,\tau)$ for $l=1,2$ and $n=1,2$.
Moreover, by the Cauchy-Schwarz inequality and $c_\omega \gtrsim 1$
for $\omega>0$ small, we have for $l=1,2$ and $n=1,2$
\begin{equation}\label{eq:jk}
|j_{l,n}(u)|\lesssim \|\nu u\|,\quad
|k_l(u)|\lesssim \|\nu u\|^2.
\end{equation}
We construct a local solution $(\gamma,\omega,\tau)$ of \eqref{eq:sp}
by applying the Cauchy-Lipschitz theorem
with the initial data $(\gamma^{\rm in}, \omega^{\rm in},0)$.
Moreover, we obtain the bound $|m_\gamma|+|m_\omega|\lesssim\|\nu u\|^2$.
The orbital stability theorem giving a uniform estimate on 
$\|u\|_{H^1}$ and $|\omega-\oz|$, we are able to extend the local solution
of \eqref{eq:sp} to a global solution.
\end{proof}

We now refine the decomposition of Lemma~\ref{LE:tm}
using the internal mode.
Recall from~\ref{it:aW} of Lemma~\ref{LE:VW} that
$\langle V_1,V_2\rangle=1/\alpha + O(1)>0$.
We introduce the notation
\[
g^\top=g-\frac{\langle g,V_1\rangle}{\langle V_1,V_2\rangle}V_2,\quad
h^\perp=h-\frac{\langle h,V_2\rangle}{\langle V_1,V_2\rangle}V_1.
\]
\begin{lemma}\label{LE:sm}
Under the assumptions of Lemma~\ref{LE:tm},
possibly taking a smaller $\delta$, there exists a unique 
$\cC^1$ function $b=b_1+\ii b_2:[0,+\infty)\to\CC$
such that $v=v_1+\ii v_2$ defined by
\[
u_1 = v_1 + b_1 V_1,\quad u_2 = v_2 + b_2 V_2
\]
satisfies, for all $s\in [0,+\infty)$ the properties \emph{(i)--(v)}.
\begin{enumerate}
\item\label{it:sv} \emph{Stability:}
$\|v\|_{H^1}+|b|\leq\varepsilon$.
\item\label{it:ov} \emph{Orthogonality relations:}
$\langle v,\ii\Lo\Qo\rangle=\langle v,\Qo\rangle
=\langle v,\ii V_1\rangle=\langle v, V_2\rangle =0$.
\item \emph{Equation of the parameters:}
\begin{equation}\label{eq:Xv}
|m_\gamma|+|m_\omega|\lesssim\|\nu v\|^2 +|b|^2.
\end{equation}
\item\label{it:ev} \emph{Equation of $v$:}
setting $r_1= -m_\omega b_2\omega\PO V_2$ and 
$r_2= m_\omega b_1\omega\PO V_1$,
\begin{equation}\label{eq:vv}
\begin{cases}
\dot v_1 =\LM v_2
+\mu_2 + p_2^\perp - q_2^\perp - r_2^\perp\\
\dot v_2 = -\LP v_1
-\mu_1- p_1^\top + q_1^\top + r_1^\top
\end{cases}
\end{equation}
\item \emph{Equation of $b$:}
setting $B_l =\langle p_l - q_l - r_l,V_l\rangle/\langle V_1,V_2\rangle$ for $l=1,2$,
\begin{equation}\label{eq:bb}
\begin{cases}
\dot b_1 =\lambda b_2+ B_2\\
\dot b_2 = -\lambda b_1 - B_1 
\end{cases}
\end{equation}
and \begin{equation}\label{eq:eB}
|B_1|+|B_2|\lesssim \oz (|b|^2 +\|\rho^4 v\|^2).
\end{equation}
\end{enumerate}
\end{lemma}
\begin{proof}
We define $(b_1,b_2)$ by
\[
b_1=\frac{\langle u_1,V_2\rangle}{\langle V_1,V_2\rangle},\quad
b_2=\frac{\langle u_2,V_1\rangle}{\langle V_1,V_2\rangle}.
\]
Note that $v_1=u_1^\perp$ and $v_2=u_2^\top$.
By Lemma~\ref{LE:VW},
$\langle V_1,\Qo\rangle=\langle V_1,S^2 \Qo\rangle=0$
and $\langle V_2,\Lo\Qo\rangle=\lambda^{-1}\langle \LP V_1,\Lo\Qo\rangle
=-\lambda^{-1}\langle V_1,\Qo\rangle=0$.
Thus the orthogonality relations for $v$ are deduced from the ones for
$u$ and the definition of $b$.
By the Cauchy-Schwarz inequality,
$|b|\lesssim\alpha\|u\|\|V\|\lesssim\sqrt{\omega}\|u\|$.
It is thus clear that $\|v\|\lesssim\|u\|$.
Besides, \eqref{eq:Xv} follows directly from \ref{it:Xi} of Lemma \ref{LE:tm}.
Now, using~\eqref{eq:uu}, one obtains
\[
\begin{cases}
\dot v_1 +\dot b_1 V_1 =\LM v_2 +\lambda b_2 V_1 
+\mu_2 - p_2 - q_2 -r_2\\
\dot v_2 +\dot b_2 V_2 = -\LP v_1 -\lambda b_1 V_2 
-\mu_1 + p_1 + q_1 + r_1
\end{cases}
\]
where $r_1,r_2$ are defined in \ref{it:ev} of the lemma.
Projecting the first line of the above system on $V_2$
and the second one on $V_1$, we get~\eqref{eq:bb}.
Since $\langle V_1,\Qo\rangle=\langle V_2,\Lo\Qo\rangle=0$,
we have $\mu_2^\perp=\mu_2$ and $\mu_1^\top=\mu_1$
and~\eqref{eq:vv} follows.
We now justify the estimate \eqref{eq:eB}.
First,
\[
\int p_1V_1 = m_\gamma\int u_1 V_1 + m_\omega\int u_2\Lambda^* V_1,\quad
\int p_2V_2 = m_\gamma\int u_2 V_2 - m_\omega\int u_1\Lambda^* V_2,
\]
and so, using 
$|V|+|yV'|\lesssim \rho^8$ 
(from \ref{it:mu} and \ref{it:aW} of Lemma \ref{LE:VW}, for $\omega$ small),
by the Cauchy-Schwarz inequality, we get
\begin{align*}
\Bigl|\int p_1 V_1\Bigr|+\Bigl|\int p_2 V_2\Bigr|
\lesssim(|m_\gamma|+|m_\omega|)\int\rho^8|u|
&\lesssim (|b|^2 +\|\nu v\|^2)(|b|/\oz+\|\rho^4v\|/\sqrt{\oz})\\
&\lesssim(1/\oz)(|b|^2 +\|\nu v\|^2)(|b| +\|\rho^4v\| ).
\end{align*}
Since $\langle V_1,V_2\rangle\gtrsim1/\oz$,
we obtain
\[
\frac 1 {\langle V_1,V_2\rangle}\Bigl(\Bigl|\int p_1 V_1\Bigr|+\Bigl|\int p_2 V_2\Bigr|\Bigr)
\lesssim(|b|^2 +\|\nu v\|^2)(|b| +\|\rho^4v\| ).
\]
Using \ref{it:eV} of Lemma \ref{LE:VW}, $|V|\lesssim \rho^8$ and
$|r_1|+|r_2|\lesssim \omega_0|m_\omega| |b| (1+|y|)$,
\[
\frac 1 {|\langle V_1,V_2\rangle|}
\Bigl(\Big|\int r_1 V_1\Big|+\Big|\int r_2 V_2\Big|\Bigl)
\lesssim\oz^2(|m_\gamma|+|m_\omega|)|b|\int (1+|y|)\rho^8
\lesssim(|b|^2 +\|\nu v\|^2)|b|.
\]
Replacing
$u_1 = v_1+b_1 V_1$ and $u_2=v_2+b_2 V_2$ in the expansions 
of $q_1,q_2$ in Lemma \ref{LE:Ty}, we obtain at the second order in $b$,
\begin{align*}
&\left|q_1 -\left(\Qo(3+10\omega \Qo^2) V_1^2b_1^2 + \Qo (1+2\omega \Qo^2)V_2^2b_2^2\right)\right|
\lesssim \nu^2|b||v|+\nu^2 |v|^2+|b|^3+|v|^3,\\
&\left|q_2 - 2 \Qo (1+2\omega \Qo^2) V_1V_2b_1b_2\right| 
\lesssim \nu^2|b||v|+\nu^2 |v|^2+|b|^3+|v|^3. 
\end{align*}
Thus, setting
\[
\tilde d_1(\omega)=\frac {\int \Qo(3+10\omega \Qo^2) V_1^3}{\langle V_1,V_2\rangle} ,\quad
\tilde d_2(\omega)=\frac {\int \Qo (1+2\omega \Qo^2)V_1V_2^2}{\langle V_1,V_2\rangle},
\]
\[
\tilde d_3(\omega)=\frac {\int 2 \Qo (1+2\omega \Qo^2) V_1V_2^2}{\langle V_1,V_2\rangle} ,
\]
we get using $|V|\lesssim \rho^8$,
\begin{align*}
&\Bigl|\frac{\int q_1 V_1}{\langle V_1,V_2\rangle}
- \tilde d_1(\omega) b_1^2 -\tilde d_2(\omega) b_2^2\Bigr|+\Bigl|\frac{\int q_2 V_2}{\langle V_1,V_2\rangle} -\tilde d_3(\omega) b_1b_2\Bigr|\\
&\quad \lesssim\oz\left(|b|\|\nu v\| +\|\nu v\|^2 +|b|^3/\oz +\|v\|_{L^\infty} 
\|\rho^4 v\|^2\right)\lesssim\oz\left(|b|\|\nu v\| +\|\rho^4 v\|^2\right)
+|b|^3.
\end{align*}
Therefore,
\begin{equation}\label{eq:B9}
 \big|B_1 - \bigl(\tilde d_1(\omega) b_1^2 + \tilde d_2(\omega) b_2^2\bigr)\big|
+\big|B_2 - \tilde d_3(\omega) b_1 b_2\big|
\lesssim \oz\left(|b|\|\nu v\| +\|\rho^4 v\|^2\right)
+|b|^3.
\end{equation}
In particular, one obtains
$|B|\lesssim\oz(|b|^2 +\|\rho^4 v\|^2)$, which is~\eqref{eq:eB}.
\end{proof}

We give an elementary pointwise estimate on the projections
$g\mapsto g^\perp$ and $g\mapsto g^\top$.

\begin{lemma}\label{LE:tp}
For all $k\geq 0$,
$|(g^\perp)^{(k)}|+|(g^\top)^{(k)}|\lesssim |g^{(k)}| + \sqrt{\oz} \rho^8\|\rho^4 g\|$.
In particular, $\|\rho g^\perp\|+\|\rho g^\top\|\lesssim \|\rho g\|$.
\end{lemma}
\begin{proof}
By the Cauchy-Schwarz inequality,
 $|V|\lesssim \rho^8$ and using
\ref{it:aW} of Lemma~\ref{LE:VW}, we have
\[
\left|\frac{\langle g,V_1\rangle}{\langle V_1,V_2\rangle}V_2\right|
\lesssim \oz \|\rho^4 g\| \|\rho^4 \| \rho^8 
\lesssim \sqrt{\oz} \|\rho^4 g\| \rho^8,
\]
which is sufficient to treat the case $k=0$.
The estimates for $k\geq 1$ are similar.
\end{proof}

\begin{lemma}\label{LE:31}
Let $\cM = |b|^4 + \|\rho v\|^2$.
For all $s\geq 0$, 
\[
|\dot{\cM}| \lesssim |b|^4 + \|\rho \py v\|^2 + \|\rho v\|^2.
\]
\end{lemma}
\begin{proof}
Using \eqref{eq:vv} and \eqref{eq:bb}, we compute
\begin{align*}
\dot{\cM} & = 2 |b|^2 (b_1\dot b_1 + b_2 \dot b_2)
+2 \int \rho^2 (v_1\dot v_1+v_2\dot v_2)\\
& = 2 |b|^2 (b_1B_2-b_2B_1)
+2 \int \rho^2 (v_1 \LM v_2 - v_2 \LP v_1)\\
&\quad +\int \rho^2 v_1 (\mu_2+p_2^\perp-q_2^\perp-r_2^\perp)
+\int \rho^2 v_2 (-\mu_1-p_1^\top+q_1^\top+r_1^\top).
\end{align*}
Using \eqref{eq:eB}, we have $|b|^2 |b_1B_2-b_2B_1|\lesssim |b|^3(|b|^2+ \|\rho^4 v\|^2)$.
Using the expression of $\LP$, $\LM$ and integrating by parts,
\[
\Bigl|\int \rho^2 (v_1 \LM v_2 - v_2 \LP v_1)\Bigr|
\lesssim \|\rho \py v\|^2 + \|\rho v\|^2.
\]
Then, using the definition of $\mu_1$, $\mu_2$,
the Cauchy-Schwarz inequality and \eqref{eq:Xv},
one gets
\[
\Bigl|\int\rho^2 v_1 \mu_2\Bigr|+\Bigl|\int\rho^2 v_2 \mu_1 \Bigr|
\lesssim (|m_\gamma|+|m_\omega|)\|\nu v\|
\lesssim ( |b|^2 + \|\nu v\|^2 ) \|\nu v\| \lesssim \cM.
\]
Lastly, using the Cauchy-Schwarz inequality and Lemma \ref{LE:tp}, we have
\begin{align*}
\Bigl|\int \rho^2 v_1 (p_2^\perp-q_2^\perp-r_2^\perp)\Bigr|
&+\Bigl|\int \rho^2 v_2 (p_1^\top-q_1^\top-r_1^\top)\Bigr|\\
&\quad\lesssim \|\rho v\| (\|\rho p_1\|+\|\rho p_2\|+ \|\rho q_1\|+\|\rho q_2\|
+\|\rho r_1\|+\|\rho r_2\|).
\end{align*}
Using $|y|\rho \lesssim 1/\oz$, $\|u\|_{H^1} \lesssim \varepsilon\lesssim \oz$ from \ref{it:pa} of Lemma \ref{LE:tm}, for $\varepsilon$ sufficiently small,
and then \eqref{eq:Xv}, 
one has 
\[
\|\rho p_1\|+\|\rho p_2\|\lesssim \oz^{-1}(|m_\gamma|+|m_\omega|)\|u\|_{H^1}
\lesssim |b|^2 + \|\nu v\|^2.
\]
One has
$|q_1|+|q_2| \lesssim \nu |u|^2 + |u|^3 \lesssim (\nu + \varepsilon) (|b|^2 + \varepsilon|v|)$ by the definitions of $q_1$ and $q_2$ in Lemma \ref{LE:Ty}, and thus
\[
\| \rho q_1\| + \|\rho q_2\|
\lesssim (1+\varepsilon/\sqrt{\oz}) |b|^2 + \|\rho v\|
\lesssim |b|^2 + \|\rho v\|.
\]
Then, by the definition of $r_1$ and $r_2$ in Lemma \ref{LE:sm}
and $|\omega \PO V_1|+|\omega\PO V_2|\lesssim 1$ from \ref{it:bW} of Lemma \ref{LE:VW},
one has $|r_1|+|r_2|\lesssim |m_\omega| |b|$ and so by \eqref{eq:Xv},
for $\varepsilon$ small,
\[
\|\rho r_1\| + \|\rho r_2 \|\lesssim |m_\omega| |b|/\sqrt{\oz}
\lesssim (|b|^2 + \|\nu v\|^2 ) \varepsilon/\sqrt{\oz}
\lesssim |b|^2 + \|\nu v\|^2.
\]
We obtain $|\dot \cM|
\lesssim |b|^5 + \|\rho \py v\|^2 + \|\rho v\|^2 + \|\rho v\| |b|^2
\lesssim |b|^4 + \|\rho \py v\|^2 + \|\rho v\|^2$
by gathering all the above estimates. 
\end{proof}

\begin{lemma}\label{LE:FF}
Let
\[
F = -\frac 2{\langle V_1,V_2\rangle} \Qo(\Lo\Qo)(1+2\omega\Qo),\quad
F_1=FV_2,\quad F_2=FV_1.
\]
There exist smooth even functions $A_1,A_2:(0,\omega_1)\times\RR\to \RR$
 satisfying the nonhomogeneous system
\[
\begin{cases}
\LP A_1 - \lambda A_2 = -F_1\\
\LM A_2 - \lambda A_1 = F_2
\end{cases}
\]
and for all $k\geq 0$, $j=1,2$, on $\RR$,
\begin{equation}\label{eq:AA}
|\py^k A_j|+ \frac{|\py^k\PO A_j|}{1+|y|} \lesssim \oz\ee^{-\alpha|y|}.
\end{equation}
\end{lemma}
\begin{proof}
We define an auxiliary problem, setting
\[
X_1 =\tfrac 12(A_1+A_2),\quad X_2 =\tfrac 12(A_1-A_2),
\]
we look for a solution of 
\[\begin{cases} 
-\py^2 X_1 +\alpha^2 X_1
+\Qo^2(2-\frac 13\omega\Qo^2) X_1 +\Qo^2(1+\frac 23\omega\Qo^2)X_2= -\tfrac12 (F_1-F_2)\\
-\py^2 X_2 +\kappa^2 X_2
+\Qo^2(1+\frac 23\omega\Qo^2) X_1 +\Qo^2(2-\frac 13\omega\Qo^2)X_2= -\tfrac12 (F_1+F_2)
\end{cases}
\]
Using the notation $H_\alpha$ from the proof of Lemma \ref{LE:VW},
we rewrite the system as
\[
\begin{pmatrix} X_1\\ X_2\end{pmatrix}
+\cT \begin{pmatrix} X_1\\ X_2\end{pmatrix}
=-\tfrac12 \Ha^{-1} \begin{pmatrix} F_1-F_2\\ F_1+F_2\end{pmatrix}, \quad 
\cT=\Ha^{-1}\Qo^2\begin{pmatrix} 2-\frac 13\omega\Qo^2 & 1+\frac 23\omega\Qo^2\\[3pt]
1+\frac 23\omega\Qo^2 & 2-\frac 13\omega\Qo^2\end{pmatrix}.
\]
The space $(L^2(\RR))^2$ is equipped 
with the standard scalar product $(g_1,h_1)\cdot (g_2,h_2)=\int (g_1g_2+h_1h_2)$.
The existence of a solution $(X_1,X_2)^\TR$ 
then follows from the Fredholm alternative.
Indeed, $\cT$ is a compact operator, 
and the uniqueness result of Lemma \ref{LE:un}, together with the orthogonality relation
$\int (-F_1 V_1) + (F_2 V_2) = 0$, ensure the existence of
a solution $(X_1,X_2)^\TR$.
Then $A_1=X_1+X_2$ and $A_2=X_1-X_2$ solve the 
original system and the decay properties of $A_1,A_2$ are proved as the ones of $V$ and $W$ in Lemma \ref{LE:VW}.
\end{proof}

The next result shows that $m_\omega$ has 
oscillatory properties which will allow us to prove that 
$\omega$ has a limit. We refer to \cite[Proposition 4.1]{BuSu} for a similar
computation.

\begin{lemma}\label{LE:Om}
There exist $\cC^1$ functions 
$c_1,c_2,c_3:(0,\omega_1)\to \RR$ such that
\[
\Omega =
b_1 \int v_1 A_2 + b_2 \int v_2 A_1+
c_1(\omega) (b_1^2-b_2^2)
+ c_2(\omega) b_1^3 + c_3(\omega) b_1b_2^2
\]
satisfies
\[
\big|m_{\omega} + \dot \Omega \big|
\lesssim C(\oz)\left(\|\rho^4 v\|^2+ |b|^4\right).
\]
\end{lemma}
\begin{proof}
In the system \eqref{eq:sp}, 
we invert the subsystem for $(m_\gamma,m_\omega)^\TR$ and we focus on the expression of $m_\omega$, expanding and using the estimates \eqref{eq:jk}.
We get
\[
m_\omega = k_2 - j_{2,1} k_1 - j_{2,2} k_2 +O(\|\nu u\|^4).
\]
Using the definitions of $k_1$, $k_2$, $j_{2,1}$ and $j_{2,2}$
in the proof of Lemma \ref{LE:tm}, this yields
\begin{align*}
m_\omega &= -(1/c_\omega) \langle q_2,\Lo\Qo\rangle
+(1/c_\omega^2) \langle u,\ii\Qo\rangle\langle q_1,\Lo\Qo\rangle\\
&\quad -(1/c_\omega^2) \langle u,(\Lo-\tfrac 12)\Qo\rangle\langle q_2,\Lo\Qo\rangle
+O(\|\nu u\|^4).
\end{align*}
Using the expansions of
$q_1$ and $q_2$ in Lemma \ref{LE:Ty} and then
substituting $u_1 = v_1+b_1 V_1$ and $u_2=v_2+b_2 V_2$, we obtain
\[
m_\omega
=b_1 \int v_2 F_1+b_2\int v_1 F_2
+\tilde c_1(\omega) b_1 b_2 +\tilde c_2(\omega) b_1^2 b_2 +\tilde c_3(\omega) b_2^3+O(\|\nu v\|^2+|b|^4)
\]
where $F_1$ and $F_2$ are defined in Lemma \ref{LE:FF} and where
$\tilde c_1$, $\tilde c_2$ and $\tilde c_3$ are explicit smooth functions of $\omega$. Their expressions are given for information,
but they will not be used
\begin{align*}
\tilde c_1&=-\frac2{c_\omega} \int (\Lo\Qo)\Qo(1+2\omega\Qo^2)V_1V_2,\displaybreak[0]\\
\tilde c_2 & = -\frac1{c_\omega} \int (\Lo\Qo) (1+6\omega\Qo^2)V_1^2V_2
+\frac{\langle V_2, \Qo\rangle}{c_\omega^2} \int (\Lo\Qo) \Qo(3+10\omega \Qo^2)V_1^2 \\
&\quad -\frac{\langle V_1 ,(\Lo-\tfrac12)\Qo\rangle}{c_\omega^2} \int (\Lo\Qo)\Qo(2+4\omega \Qo^2)V_1V_2,
\displaybreak[0]\\
\tilde c_3&= -\frac1{c_\omega} \int (\Lo\Qo)(1+2\omega\Qo^2)V_2^3
+\frac{\langle V_2, \Qo\rangle }{c_\omega^2} \int (\Lo\Qo)\Qo (1+2\omega \Qo^2)V_2^2.
\end{align*}
We proceed similarly for $m_\gamma$, at the second order for $b$ only,
\[
m_\gamma = \tilde c_4(\omega) b_1^2 + \tilde c_5(\omega) b_2^2
+O(\|\nu v\|^2 + |b|^3+ |b|\|\nu v\|)
\]
where
\[
\tilde c_4 = \frac 1{c_\omega} \int \Qo(\Lo\Qo)(3+10\omega\Qo^2V_1^2),\quad
\tilde c_5 = \frac 1{c_\omega} \int \Qo(\Lo\Qo)(1+2\omega\Qo^2V_2^2).
\]
On the one hand, setting $\Omega_1 = b_1\int v_1A_2+b_2 \int v_2 A_1$,
we compute
\begin{align*}
\dot\Omega_1 
&= \dot b_1\int v_1A_2+ b_1\int \dot v_1 A_2+ b_1 \int v_1\dot A_2
+\dot b_2 \int v_2 A_1 + b_2 \int \dot v_2 A_1+ b_2 \int v_2 \dot A_1 \\
&=- b_2 \int v_1 (\LP A_1 - \lambda A_2) + b_1 \int v_2 (\LM A_2 - \lambda A_1)
+b_1 \int \mu_2 A_2-b_2\int \mu_1A_1+ \Omega_3\\
& = b_1 \int v_2 F_1+b_2\int v_1 F_2 
- \tilde c_1 b_1^2b_2 \int A_2 \Lo\Qo -( \tilde c_4 b_1^2b_2+\tilde c_5 b_2^3) \int A_1\Qo
+ \Omega_3 +\Omega_5
\end{align*}
where
\begin{align*}
\Omega_3 & = \int v_1 A_2B_2 
+ b_1 \int (p_2^\perp - q_2^\perp - r_2^\perp)A_2
+ b_1 \dot \omega \int v_1 \PO A_2\\
&\quad -\int v_2 A_1B_1 + b_2\int (p_1^\top + q_1^\top + r_1^\top)A_1
+ b_2 \dot\omega \int v_2 \PO A_1
\end{align*}
and
\[
\Omega_5 = - b_1(m_\omega -\tilde c_1b_1b_2) \int A_2 \Lo\Qo
-b_2 (m_\gamma - \tilde c_4 b_1^2 -\tilde c_5 b_2^2) \int A_1 \Qo
\]
are error terms. Indeed, using \eqref{eq:Xv}, \eqref{eq:eB}, \eqref{eq:AA}
and Lemma \ref{LE:tp},
we check that
\[
|\Omega_3|+|\Omega_5|\lesssim C(\oz)(\|\rho^4 v\|^2+|b|^4 ).
\]
Thus, for some constants $\tilde c_6(\omega)$ and $\tilde c_7(\omega)$,
\[
m_\omega-\dot \Omega_1
=\tilde c_1(\omega) b_1 b_2 +\tilde c_6(\omega) b_1^2 b_2 +\tilde c_7(\omega) b_2^3
+O(\|\nu v\|^2+|b|^4).
\]
On the other hand, by~\eqref{eq:bb}
and \eqref{eq:B9}, one has
\begin{align*}
&\frac{d}{ds}(b_1^2-b_2^2) 
 =4\lambda b_1b_2 +2 (\tilde d_3-\tilde d_1) b_1^2 b_2 - 2 \tilde d_2 b_2^3
+O(\|\rho^4 v\|^2+|b|^4), \\
&\frac{d}{ds}(b_1^3) 
=3 \lambda b_1^2 b_2+O(\|\rho^4 v\|^2+|b|^4),\quad
\frac{d}{ds}(b_1b_2^2) =\lambda b_2^3 - 2\lambda b_1^2 b_2+O(\|\rho^4 v\|^2+|b|^4).
\end{align*}
In particular, there exist smooth functions of $\omega$, 
denoted by $c_1$, $c_2$ and $c_3$ such that the function
$\Omega_2 = c_1 (b_1^2-b_2^2)+ c_2 b_1^3 + c_3 b_1b_2^2$
satisfies 
\[
\dot \Omega_2 = \tilde c_1(\omega) b_1 b_2 +\tilde c_6(\omega) b_1^2 b_2 +\tilde c_7(\omega) b_2^3 +\Omega_4
\]
where
\begin{align*}
\Omega_4
&=2c_1 (b_1B_2-b_2B_1)+3c_2 b_1^2B_2+ c_3 B_2b_2^2+2c_3b_1b_2 B_1)\\
&\quad +\dot \omega (\PO c_1) (b_1^2-b_2^2)
+\dot \omega (\PO c_2) b_1^3 +\dot \omega (\PO c_3) b_1b_2^2.
\end{align*}
Using \eqref{eq:Xv} and \eqref{eq:eB}, we see that $\Omega_4$ satisfies
$|\Omega_4|\lesssim C(\oz)(|b|^4 +|b|^2\|\rho^4 v\|^2)$.
Thus, $\Omega= \Omega_1+\Omega_2$ 
satisfies the desired properties.
\end{proof}

\section{Estimate at large scale}\label{se:05}

We introduce some notation for \emph{localized} virial identities.
Fix a smooth even function $\chi:\RR\to\RR$ such that
\begin{equation}\label{eq:ch}
\mbox{$\chi=1$ on~$[0,1]$,~$\chi=0$ on~$[2,+\infty)$,
$\chi'\leq 0$ on~$[0,+\infty)$.}
\end{equation}
Let~$1\ll B\ll A$ be large constants to be defined later. We define
\begin{align*}
\chi_A(y)&=\chi\left(\frac yA\right),\quad 
\eta_A(y)=\sech\left(\frac{2y} A\right),\\
\zeta_A(y)&=\exp\left(-\frac{|y|}A(1-\chi(y))\right),\quad
\Phi_A(y)=\int_0^y\zeta_A^2.
\end{align*}
We remark that
$0<\Phi'_A=\zeta_A^2\leq 1$, $|\Phi_A|\leq|y|$,
and $|\Phi_A|\leq A$ on $\RR$.
We define the function $\Psi_{A,B}$
and the operators $\Theta_A$, $\Xi_{A,B}$ by
\[
\Psi_{A,B}=\chi_A^2\Phi_B,\quad 
\Theta_A = 2\Phi_A\py +\Phi_A' ,\quad
\Xi_{A,B} = 2\Psi_{A,B}\py +\Psi_{A,B}' .
\]
For future use, we recall two 
inequalities from \cite[Lemma 4]{KMM2} and \cite[Claim 1]{KMM2}.

\begin{lemma}
For all $A>0$ and all $g\in H^1$,
\begin{align}
\|\zeta_A g\| &\lesssim \sqrt{A}\|\nu g\| +A\|\py g\|,\label{eq:KK}\\
\|\zeta_A g^2\| &\lesssim A\|g\|_{L^\infty}\|\py(\zeta_A g)\|.
\label{eq:gK}
\end{align}
\end{lemma}
\begin{proof}
Let $y,z\in \RR$.
Using $g(y)=g(z) +\int_z^y\py g$
and the Cauchy-Schwarz inequality, we have
$g^2(y)\leq 2 g^2(z) + 2(|y|+|z|)\int(\py g)^2$.
Multiplying this inequality by $\zeta_A^2(y)\nu^2(z)$ and integrating on $\RR^2$,
we find~\eqref{eq:KK}.

Set $h=\zeta_A g$. By integration by parts and the Cauchy-Schwarz inequality
\begin{align*}
\frac 2A\int_0^{\infty} \ee^{\frac{2y}A} h^4 dy&
=-h^4(0)- 4\int_0^{\infty} \ee^{\frac{2y}A} h^3\py h dy
\leq 4\|\ee^{\frac{y}A}h\|_{L^\infty}\|\py h\|
\Bigl(\int_0^{\infty} \ee^{\frac{2y}A} h^4 dy\Bigr)^{\frac12}.
\end{align*}
Hence
$\|\zeta_A^{-1} h^2\|\lesssim A\|\zeta_A^{-1} h\|_{L^\infty}\|\py h\|$,
which is~\eqref{eq:gK}.
\end{proof}

The next lemma provides an estimate of~$v$ at spatial scale~$A$ in terms of 
$|b|^2$ and of $v$ at a local scale.
The proof of this estimate being based on a virial identity, its holds in time average.

\begin{lemma}\label{LE:v1}
For all $s>0$,
\[
\int_0^s\Bigl(\|\eta_A\py v\|^2 +\frac1{A^2}\|\eta_A v\|^2\Bigr) 
\lesssim \varepsilon +\int_0^s\bigl(\|\rho^4 v\|^2 +|b|^4\bigr) .
\]
\end{lemma}
\begin{remark}
In this proof, the parameter $\oz>0$ is to be taken sufficiently small, then $A$ sufficiently large depending on $\oz$, and lastly
$\varepsilon>0$ sufficiently small, depending both on $A$ and $\oz$. See also the remark
after Lemma \ref{LE:vz}.
\end{remark}

\begin{proof}
The proof is similar to~\cite[Proof of Proposition 2]{Ma22},
~\cite[Proposition 1]{KMM2}
and~\cite[Proof of Proposition 1]{KoMa}.
Define
\[
\bI=\omega\int(\Theta_A v_2) v_1,\quad \tilde v=\zeta_A v.
\]
By~\ref{it:pa} of Lemma \ref{LE:tm}, we have
the estimate $\tfrac 12 \oz \leq\omega\leq 2\oz$
which we will often use tacitly.
Firstly, we prove that
there exists a constant $C>0$ such that
for all $s\geq 0$,
\begin{equation}\label{eq:dI}
\dot{\bI} 
\geq\oz\left(\tfrac14\|\py\tilde v\|^2 - C\|\rho^4 v\|^2
- C| b|^4\right).
\end{equation}
Proof of~\eqref{eq:dI}.
By~\eqref{eq:vv} and~$\int(\Theta_A g) h = -\int(\Theta_A h) g$,
we have
$\frac{d}{ds} \int(\Theta_A v_2) v_1 = \sum_{j=1}^5\bi_j$
where
\begin{align*} 
\bi_1 &= -\int(\Theta_A v_1)\py^2 v_1-\int(\Theta_A v_2)\py^2 v_2, \quad
\bi_2=\int (\Theta_A v_1)\mu_1 +\int(\Theta_A v_2)\mu_2 ,\displaybreak[0]\\
\bi_3 &=\int (\Theta_A v_1)p_1^\top+\int (\Theta_A v_2) p_2^\perp ,\quad
\bi_4= -\int(\Theta_A v_1)r_1^\top-\int (\Theta_A v_2)r_2^\perp, \displaybreak[0]\\
\bi_5&=-\int (\Theta_A v_1)\bigl(\fo'(\Qo)v_1+q_1^\top\bigr)
-\int (\Theta_A v_2)\bigl((\fo(\Qo)/\Qo)v_2 + q_2^\perp\bigr).
\end{align*}
Integrating by parts, one has
\[
-\int(\Theta_A v_1)\py^2 v_1
= 2\int \Phi_A' (\py v_1)^2 - \frac 12 \int \Phi_A''' v_1^2
= 2\int(\py\tilde v_1)^2 +\int(\ln\zeta_A)''\tilde v_1^2.
\]
Since
$(\ln\zeta_A)''
=\frac1{A}\left(|y|\chi''(y)+ 2\chi'(y)\sgn(y)\right)$
and since the function $\chi$ is supported on $[-2,2]$, it holds
$ |(\ln\zeta_A)'' |\lesssim \nu^2/A$.
Thus, for a constant $C>0$,
\[
\bi_1
\geq 2\|\py\tilde v\|^2 -\tfrac {C}{A}\|\nu \tilde v\|^2
\geq 2\|\py\tilde v\|^2 -C\|\nu v\|^2.
\]
We turn to $\bi_2$.
Using $|\Phi_A|\lesssim|y|$, $0<\Phi_A'\leq 1$,
the definition of $\mu_k$ in Lemma~\ref{LE:tm}
combined with~\eqref{eq:Xv} and the decay properties of $\Qo$, we find, for $k=1,2$,
\[
|\Theta_A\mu_k|\lesssim(|m_\gamma|+|m_\omega|)\nu^5
\lesssim(\|\nu v\|^2+|b|^2)\nu^5.
\]
Thus, by the Cauchy-Schwarz inequality
\[
|\bi_2|\lesssim
\sum_{k=1,2}\Big|\int(\Theta_A v_k)\mu_k\Big|
=\sum_{k=1,2}\Big|\int v_k(\Theta_A\mu_k)\Big|
\lesssim\bigl(\|\nu v\|^2+|b|^2\bigr)\|\nu v\|\lesssim 
\|\nu v\|^2 +| b|^4.
\]
Then, we expand the two terms in $\bi_3$
\begin{align*}
\int(\Theta_A v_1)p_1^\top
&=\int(\Theta_A v_1)\tilde p_1 +\int(\Theta_A v_1)\check p_1
-\frac{\langle p_1,V_1\rangle}{\langle V_1,V_2\rangle}\int(\Theta_A v_1)V_2\\
\int(\Theta_A v_2)p_2^\perp
&=\int(\Theta_A v_2)\tilde p_2 +\int(\Theta_A v_2)\check p_2
-\frac{\langle p_2,V_2\rangle}{\langle V_1,V_2\rangle}\int(\Theta_A v_2)V_1
\end{align*}
where
\begin{align*}
\tilde p_1 &= m_\gamma v_1
+m_\omega\Lambda v_2,\quad
\check p_1 = m_\gamma b_1 V_1
+m_\omega b_2\Lambda V_2\\
\tilde p_2 &= m_\gamma v_2
-m_\omega\Lambda v_1,\quad
\check p_2 = m_\gamma b_2 V_2
-m_\omega b_1\Lambda V_1 .
\end{align*}
Using $\int (\Theta_A g) g=0$, it holds
\[
\int(\Theta_A v_1)\tilde p_1
= m_\omega\int(\Theta_A v_1)\Lambda v_2,\quad
\int(\Theta_A v_2)\tilde p_2
= -m_\omega\int(\Theta_A v_2)\Lambda v_1.
\]
Thus, using a cancellation,
\begin{align*}
\int(\Theta_A v_1)\tilde p_1+\int(\Theta_A v_2)\tilde p_2
&= m_\omega\int\left(-\Phi_A+\tfrac 12y\Phi_A'\right)(v_1\py v_2-v_2\py v_1)\\
&= -m_\omega\int(\Theta_A v_2) v_1
+\tfrac12{m_\omega}\int y\zeta_A^2(v_1\py v_2-v_2\py v_1).
\end{align*}
Using $|y|\zeta_A^{1/2}\lesssim A$, the Cauchy-Schwarz inequality and then
$\|\py v\|\lesssim\varepsilon$
 we find
\[
\int|y|\zeta_A^2\bigl| v_2\py v_1+ v_1\py v_2\bigr|
\lesssim A\int\zeta_A^\frac 32|v||\py v|
\lesssim A\|\py v\|\|\zeta_A^\frac 12\tilde v\|
\lesssim A\varepsilon\|\zeta_{2A}\tilde v\|.
\]
Then, we use~\eqref{eq:KK} and
$|\tilde v|\lesssim|v|$,
\[\|\zeta_{2A}\tilde v\|
\lesssim A \left(\|\nu\tilde v\|+\|\py\tilde v\|\right)
\lesssim A \left(\|\nu v\|+\|\py\tilde v\|\right).
\]
Using also~\eqref{eq:Xv}, we obtain
\begin{align*}
\Bigl| {m_\omega}\int y\zeta_A^2(v_1\py v_2-v_2\py v_1)\Bigr|
\lesssim A^2\varepsilon\left(\|\nu v\|^2+|b|^2\right)
\left(\|\nu v\|+\|\py\tilde v\|\right).
\end{align*}
Thus, for $\varepsilon$ small enough, we
have
\[
\Bigl|\int(\Theta_A v_1)\tilde p_1+\int(\Theta_A v_2)\tilde p_2
+m_\omega\int(\Theta_A v_2) v_1\Bigr|
\leq\frac 12\|\py\tilde v\|^2+\|\nu v\|^2 +|b|^4.
\]
For simplicity, the constants have been fixed to simple values, by taking $\varepsilon$
small enough depending on $A$, but only the constant $1/2$ 
in front of the term $\|\py\tilde v\|^2$ really matters.
To control the terms $\int(\Theta_A v_k)\check p_k$,
we observe first that by \ref{it:bW} of Lemma~\ref{LE:VW}
and then using $\alpha > 4\omega/5 $ by \ref{it:mu} of Lemma~\ref{LE:VW} (for $\omega$ sufficiently small), for $k=1,2$,
\begin{equation}\label{eq:eV}
(y^2+1)|V_k''|+(|y|+1)|V_k'| +|V_k| 
\lesssim(y^2\omega^2+ 1) \ee^{-\alpha|y|}
+(y^2+1)\ee^{-|y|}
\lesssim \rho^8.
\end{equation}
Therefore, using $|\Phi_A|\lesssim|y|$, $|\Phi_A'|\lesssim 1$,
and then~\eqref{eq:Xv}, we obtain
\[
|\Theta_A\check p_1|+|\Theta_A\check p_2|
\lesssim|b|\left(|m_\gamma|+|m_\omega|\right) \rho^8
\lesssim|b|\left(\|\nu v\|^2+|b|^2\right)\rho^8.
\]
Since
$\int(\Theta_A v_k)\check p_k=-\int v_k(\Theta_A\check p_k)$,
we obtain by the Cauchy-Schwarz inequality,
and for $\varepsilon$ small enough, for $k=1,2$,
\[
\Bigl|\int(\Theta_A v_k)\check p_k\Bigr|
\lesssim\frac{1}{\sqrt{\oz}}|b| 
\left(\|\nu v\|^2+|b|^2\right)\|\rho^4 v\|
\lesssim\frac 1\oz\|\rho^4 v\|^4 +|b|^4
\lesssim\|\rho^4 v\|^2 +|b|^4.
\]
Moreover, using \eqref{eq:eV} and $|\langle V_1,V_2\rangle|\gtrsim 1/\alpha$ (from \ref{it:aW} of
Lemma \ref{LE:VW}), one has for $k=1,2$,
\[
\Bigl|\frac{\langle p_k,V_k\rangle}{\langle V_1,V_2\rangle}\Bigr|
\lesssim\sqrt{\oz}(|m_\gamma|+|m_\omega|)\|\rho^4 u\|
\lesssim\left(\|\nu v\|^2+|b|^2\right)(\|\rho^4 v\|+|b|) 
\]
and
\[
\Bigl|\int v_1(\Theta_A V_2)\Bigr|
+\Bigl|\int v_2(\Theta_A V_1)\Bigr| 
 \lesssim\frac1{\sqrt{\oz}}\|\rho^4 v\|.
\]
Thus,
\begin{align*}
&\Bigl|\frac{\langle p_1,V_1\rangle}{\langle V_1,V_2\rangle}
\int v_1(\Theta_A V_2)\Bigr|
+\Bigl|\frac{\langle p_2,V_2\rangle}{\langle V_1,V_2\rangle}
\int v_2(\Theta_A V_1)\Bigr|\\
&\quad\lesssim (1/{\sqrt{\oz}})
\left(\|\nu v\|^2+|b|^2\right)(\|\rho^4 v\|+|b|)\|\rho^4 v\|
\lesssim (1/\oz^2)\|\rho^4 v\|^4+|b|^4
\lesssim\|\rho^4 v\|^2+|b|^4.
\end{align*}
Summarizing, for $\varepsilon>0$ small enough (depending on $A$ and $\oz$), we have proved
\[
|\bi_3 + m_\omega\int(\Theta_A v_2) v_1|
\leq\tfrac 12\|\py\tilde v\|^2
+C\|\rho^4 v\|^2 +C|b|^4.
\]
For $\bi_4$, we recall from the definition of $r_1^\top$
\begin{align*}
-\int (\Theta_A v_1) r_1^\top
&=\int v_1 \Theta_A r_1^\top 
= \int v_1 \Theta_A r_1 - \frac{\langle r_1,V_1\rangle}{\langle V_1,V_2\rangle}\int v_1\Theta_A V_1\\
&= -m_\omega b_2 \Bigl(\int v_1 \Theta_A \omega \PO V_2
 - \frac{\langle \omega \PO V_2,V_1\rangle}{\langle V_1,V_2\rangle}\int v_1\Theta_A V_1\Bigr).
\end{align*}
For $k=1,2$, from \eqref{eq:eV}, we already know that $|\Theta_A V_k|+|V_k|\lesssim \rho^8$.
Using \ref{it:bW} of Lemma~\ref{LE:VW}, we also check that $|\Theta_A \omega\PO V_k|+|\omega \PO V_k|\lesssim \rho^8$.
Thus, we obtain
\[
\Bigl|\int (\Theta_A v_1) r_1^\top \Bigr|\lesssim 
\frac1{\sqrt{\oz}}|m_\omega| |b| \|\rho^4 v\|
\lesssim \frac1{\sqrt{\oz}}\left(\|\nu v\|^2+|b|^2\right)|b| \|\rho^4 v\|
\lesssim\|\rho^4 v\|^2+|b|^4.
\]
The same estimate is checked on $\int (\Theta_A v_2) r_2^\top$, and we obtain
\[
|\bi_4|\lesssim \|\rho^4 v\|^2+|b|^4.
\]
For the terms in $\bi_5$, we decompose
\begin{align*}
& \int(\Theta_A v_1 )(\fo'(\Qo)v_1+q_1^\top)
 = \int(\Theta_A v_1)\tilde q_1+\int(\Theta_A v_1)\check q_1
+\frac{\langle q_1,V_1\rangle}{\langle V_1,V_2\rangle}\int v_1(\Theta_A V_2)\\
&\int(\Theta_A v_2 )(({\fo(\Qo)}/{\Qo})v_2 + q_2^\perp)
 = \int(\Theta_A v_2)\tilde q_2+\int(\Theta_A v_2)\check q_2
+\frac{\langle q_2,V_2\rangle}{\langle V_1,V_2\rangle}\int v_2(\Theta_A V_1)
\end{align*}
where
\begin{align*}
\tilde q_1 &
=\Re\left\{\fo(\Qo+v)-\fo(\Qo)\right\},\quad
\tilde q_2
=\Im\left\{\fo(\Qo+v)-\fo(\Qo)\right\},\\
\check q_1 & 
=\Re\left\{\fo(\Qo+u)-\fo(\Qo+v)-\fo'(\Qo)(u_1-v_1)\right\},\\
\check q_2 &
=\Im\left\{\fo(\Qo+u)-\fo(\Qo+v)-\ii(\fo(\Qo)/{\Qo})(u_2-v_2)\right\}.
\end{align*}
Note that for $p>1$,
$\py ({|u|^{p+1}} ) =(p+1)\Re ((\py\bar u)|u|^{p-1} u ).
$
Setting $\FO(\psi) = \frac14|\psi|^4+\frac\omega6|\psi|^6$,
\begin{align*}
\Re\left\{(\py\bar v)(\fo(\Qo+v) -\fo(\Qo))\right\}
&=
\py\left\{\Re\left( \FO(\Qo + v) - \FO(\Qo) -\fo(\Qo) v\right)\right\}\\
&\quad -\Re\left\{\Qo'(\fo(\Qo+v) - \fo(\Qo)- \fo'(\Qo) v)\right\}.
\end{align*}
Therefore, by the definition of $\Theta_A v$ and 
integration by parts, we have
\[
\int(\Theta_A v_1)\tilde q_1
+\int(\Theta_A v_2)\tilde q_2
=\Re\int(\Theta_A\bar v)(\fo(\Qo+v)-\fo(\Qo) )
=\bi_{5,1}+\bi_{5,2}+\bi_{5,3}
\]
where
\begin{align*}
\bi_{5,1} & = -2\Re\int\Phi_A'\left(\FO(\Qo + v) - \FO(\Qo) -\fo(\Qo) v\right)\displaybreak[0]\\
\bi_{5,2} & = -2\Re\int\Phi_A \Qo'(\fo(\Qo+v) - \fo(\Qo)- \fo'(\Qo) v)\displaybreak[0] \\
\bi_{5,3} & = \Re\int\Phi_A' \bar v\left(\fo(\Qo+v)-\fo(\Qo)\right).
\end{align*}
By $|\Phi_A|\leq |y|$, $0<\Phi_A'=\zeta_A^2\leq 1$ and $|v|\lesssim 1$, we have
\begin{align*}
|\bi_{5,1}|+|\bi_{5,2}|+|\bi_{5,3}|
\lesssim\int\left((1+|y|)\Qo^2|v|^2 +\zeta_A^2|v|^4\right)
\lesssim\|\nu v\|^2 +\|\zeta_A v^2\|^2.
\end{align*}
Using the estimate~\eqref{eq:gK}, we have
$\|\zeta_A v^2\| 
\lesssim A\|v\|_{L^\infty} \|\py\tilde v\| 
\lesssim A\varepsilon\|\py\tilde v\|$.
Thus, we obtain for $\varepsilon>0$ small enough (depending on $A$).
\[
\Bigl|\int(\Theta_A v_1)\tilde q_1\Bigr|
+\Bigl|\int(\Theta_A v_2)\tilde q_2\Bigr|
\leq\frac 12\|\py\tilde v\|^2 + C\|\rho^4 v\|^2.
\]
Then, we estimate the terms in $\bi_5$ containing 
$\check q_1$ and $\check q_2$.
For integer $p>0$, we set
\begin{align*}
k_{p,\omega}(u) & 
=|\Qo+u|^{2p}(\Qo+u)-\Qo^{2p+1}-(2p+1)\Qo^{2p}u_1-\ii\Qo^{2p} u_2\\
& =\Qo^{2p+1}\left(|1+{\tilde u}|^{2p}(1+{\tilde u}) - 1 -(2p+1) {\tilde u}_1 -\ii {\tilde u}_2\right),
\end{align*}
where ${\tilde u}={\tilde u}_1+\ii {\tilde u}_2=u/{\Qo}$.
In particular, with similar notation for $v$,
it holds
\begin{align*}
k_{p,\omega}(u)-k_{p,\omega}(v) 
& =\Qo^{2p+1}\big\{|1+{\tilde u}|^{2p}-|1+{\tilde v}|^{2p} - 2p({\tilde u}_1-{\tilde v}_1)
\\&\quad+(|1+{\tilde u}|^{2p}-|1+{\tilde v}|^{2p}) {\tilde u} 
+(|1+{\tilde v}|^{2p}-1)({\tilde u}-{\tilde v})
\big\}
\end{align*}
From this identity, we obtain readily the estimate
\begin{align*}
\left| k_{p,\omega}(u)-k_{p,\omega}(v)\right| 
&\lesssim\Qo^{2p+1}\left(|{\tilde u}|+|{\tilde v}|+|{\tilde u}|^{2p}+|{\tilde v}|^{2p}\right)|{\tilde u}-{\tilde v}|\\
&\lesssim\left(\Qo^{2p-1}(|u|+|v|)+|u|^{2p}+|v|^{2p}\right)|u-v|.
\end{align*}
Applying this estimate with $p=1$ and $p=2$, 
and using $|u|\lesssim |b|+|v|\lesssim\varepsilon$, we obtain
\begin{align*}
|\check q_1|+|\check q_2| &\lesssim
\big|\fo(\Qo+u)-\fo(\Qo+v) -\fo'(\Qo)(u_1-v_1)
-\ii(\fo(\Qo)/\Qo)(u_2-v_2)\big|\\
&\lesssim|u-v|\bigl(\Qo(|u|+|v|) +|u|^{2}+|v|^{2}\bigr)
\lesssim|b||V||\bigl(\nu^{10}+\varepsilon\bigr) (|v|+|b|).
\end{align*}
Thus, using $|V|\lesssim \rho^8$ (see \eqref{eq:eV}), it holds
$|\check q_1|+|\check q_2| 
\lesssim \bigl(\nu^{10} +\varepsilon\rho^8\bigr)|b|(|b|+|v|)$.
Recall that $|\Phi_A(y)|\lesssim|y|$ and $|\Phi_A'(y)|\lesssim 1$.
In particular, 
$|y|\nu\lesssim 1$ 
and $|\Phi_A(y)|\rho\lesssim|y|\rho\lesssim 1/\oz$.
Thus, for a constant $c\in(0,1)$ to be chosen later,
\begin{align*}
& |(\Theta_A v_1)\check q_1|+|(\Theta_A v_2)\check q_2|
\lesssim \bigl(\nu^9 +(\varepsilon/\oz)\rho^7\bigr)|b|(|b|+|v|)(|\py v|+|v|)\\
&\quad\lesssim 
\nu^9 c^{-1}|b|^4+(c+|b|)\nu^9(|\py v|^2+|v|^2)
+(\varepsilon/\oz)\rho|b|^4+ (\varepsilon/\oz)\rho^9 (|\py v|^2+|v|^2),
\end{align*}
where for the last term above, we have used the estimate
\[
\rho^6|b||v|(|\py v|+|v|)\lesssim |b|^4+(\rho^6|v|(|\py v|+|v|))^{4/3}
\lesssim |b|^4+\rho^8 (|\py v|+|v|)^2.
\]
Hence,
\[
\Bigl|\int(\Theta_A v_1)\check q_1\Bigr|
+\Bigl|\int(\Theta_A v_2)\check q_2\Bigr|
\lesssim
 (c^{-1}+ \varepsilon/\oz^2 )|b|^4
+(c+ \varepsilon/{\oz}) (\|\rho^4\py v\|^2+\|\rho^4 v\|^2
\]
and so, taking $\varepsilon$ small enough, depending on $\oz$, and $c>0$
small enough,
\[
\Bigl|\int(\Theta_A v_1)\check q_1\Bigr|
+\Bigl|\int(\Theta_A v_2)\check q_2\Bigr|
\leq\frac 1{2^9}\|\rho^4\py v\|^2+ C\|\rho^4 v\|^2+ C|b|^4.
\]
We deal with the term $\|\rho^4\py v\|$.
Using $\tilde v =\zeta_A v$ and integrating by parts
\[
\int\frac{\rho^8}{\zeta_A^2}|\py\tilde v|^2
=\int\rho^8|\py v|^2
+\int\rho^8\left(2\frac{(\zeta_A')^2}{\zeta_A^2}-\frac{\zeta_A''}{\zeta_A}
-8\frac{\rho'}{\rho}\frac{\zeta_A'}{\zeta_A}\right)v^2
\]
which implies
\[
\|\rho^4\py v\|^2
\leq\int\frac{\rho^8}{\zeta_A^2}|\py\tilde v|^2
+ C\int\rho^8 v^2.
\]
Using $\rho^4\leq2^4e^{-2\oz|y|/5} \leq 2^4\zeta_A$ (for $2\oz A>5$), we have
$\|\rho^4\py v\|^2\leq2^8\|\py\tilde v\|^2 + C\|\rho^4 v\|^2$
and thus
\[
\Bigl|\int(\Theta_A v_1)\check q_1\Bigr|
+\Bigl|\int(\Theta_A v_2)\check q_2\Bigr|
\leq\frac 12\|\py\tilde v\|^2+C\|\rho^4 v\|^2+ C|b|^4.
\]
Next, by $|\Phi_A(y)|\lesssim|y|$, $|\Phi_A'(y)|\lesssim 1$
and~\eqref{eq:eV}, we have
\[
\Bigl|\int v_1(\Theta_A V_2)\Bigr| +\Bigl|\int v_2(\Theta_A V_1)\Bigr|
\lesssim\int|v|(|y||V'|+|V|)
\lesssim\int\rho^8|v|
\lesssim\frac 1{\sqrt{\oz}}\|\rho^4 v\|.
\]
Moreover, using $\langle V_1,V_2\rangle\gtrsim\oz^{-1}$,
and $|q_1|+|q_2|\lesssim|\Qo||u|^2 +|u|^3
\lesssim(\nu^{10}+\varepsilon)|u|^2$,
\[
\Bigl|\frac{\langle q_1,V_1\rangle}{\langle V_1,V_2\rangle}\Bigl|+
\Bigl|\frac{\langle q_2,V_2\rangle}{\langle V_1,V_2\rangle}\Bigl| 
\lesssim \oz \int \rho^8 |q|
\lesssim\oz(\|\nu^5 u\|^2 +\varepsilon\|\rho^4 u\|^2)
\lesssim\oz(\|\rho^4 v\|^2+|b|^2),
\]
for $\varepsilon$ sufficiently small.
Thus,
\[
\Bigl|\frac{\langle q_1,V_1\rangle}{\langle V_1,V_2\rangle}
\int v_1(\Theta_A V_2)\Bigr|
+\Bigl|\frac{\langle q_2,V_2\rangle}{\langle V_1,V_2\rangle}
\int v_2(\Theta_A V_1)\Bigr|
\lesssim\|\rho^4 v\|^2 +|b|^4 .
\]
In conclusion for this term, we have shown
\[
|\bi_5|
\leq \|\py\tilde v\|^2 + C\|\rho^4 v\|^2 +C|b|^4 .
\]
Combining all the above estimates, we have proved that
\[
\frac d{ds}\int(\Theta_A v_2) v_1 + m_\omega\int(\Theta_A v_2) v_1\geq
\frac 12\|\py\tilde v\|^2 -C\|\rho^4 v\|^2 -C|b|^4 .
\]
Therefore,~\eqref{eq:dI} is proved.

Secondly,
for any $s\geq0$, using $|\Phi_A|\lesssim A$ and~\ref{it:sv}
of Lemma \ref{LE:sm}, we estimate
\[
|\bI(s)|\leq\oz A\|v\|_{H^1(\RR)}^2
\lesssim\oz A\varepsilon^2 ,
\]
for $\varepsilon$ small depending on $A$.
Therefore, integrating~\eqref{eq:dI} on $[0,s]$
and dividing by $\oz$, we obtain
\[
\int_0^s\|\py\tilde v\|^2 
\lesssim A\varepsilon^2 +\int_0^s\left(\|\rho^4 v\|^2 +
|b|^4\right) .
\]
Using $\eta_A |v| = (\eta_A/\zeta_A) |\tilde v|
\lesssim \zeta_A |\tilde v|$ and then~\eqref{eq:KK},
\[
\frac1{A^2}\|\eta_{A}v\|^2
\lesssim 
\frac1{A^2}\|\zeta_A \tilde v\|^2
\lesssim\|\py\tilde v\|^2 +\frac{1}{A}\|\nu\tilde v\|^2 
\lesssim\|\py\tilde v\|^2 +\frac{1}{A}\|\rho^4 v\|^2 .
\]
Last, expanding $|\py\tilde v|^2=|\py(\zeta_A v)|^2$
and using $|\zeta_A'|\lesssim A^{-1}\zeta_A$,
\[
\int\zeta_A^2|\py\tilde v|^2 
=\int\zeta_A^4|\py v|^2 +2 \int \zeta_A^3\zeta_A'v\py v
+\int \zeta_A^2(\zeta_A')^2 |v|^2
\geq \frac 12 \int \zeta_A^4|\py v|^2 - \frac C{A^2} \int \zeta_A^4 | v|^2
\]
and so using $ \eta_A\lesssim \zeta_A^2$ and then
$\zeta_A^2\lesssim \eta_A$, we obtain
\[
\|\eta_A \py v\|^2
\lesssim \|\zeta_A^2 \py v\|^2
\lesssim\|\py\tilde v\|^2 +\frac 1{A^2}\|\eta_Av\|^2,
\]
which completes the proof of the lemma.
\end{proof}

\section{The Fermi golden rule}\label{se:06}

To formulate the Fermi golden rule,
we need a non trivial bounded solution $(g_1,g_2)$ of
\begin{equation}\label{eq:ff}
\begin{cases}
\LP g_1 = 2\lambda g_2\\
\LM g_2 = 2\lambda g_1
\end{cases}
\end{equation}
where $\lambda$ is defined in Lemma \ref{LE:VW}.
The key observation to obtain such a solution is similar to the beginning of Section~\ref{se:02}.
If a function $h_1$ satisfies the equation
$\MM\MP h_1 = 4\lambda^2 h_1$,
then by~Lemma~\ref{LE:LM},
$g_1 =(S^*)^2 h_1$ satisfies
$\LM\LP g_1 = 4\lambda^2 g_1$
and setting $g_2 =\frac1{2\lambda}\LP g_1$, the pair $(g_1,g_2)$ satisfies~\eqref{eq:ff}.

\begin{lemma}\label{le:gf}
Let $\tau =\sqrt{2\lambda-1}$.
For $\omega>0$ small, there exist smooth even functions $h_1,h_2$
of $\omega$ and $y\in\RR$, satisfying
\begin{equation}\label{eq:gv}\begin{cases}
\MP h_1 = 2\lambda h_2\\
\MM h_2 = 2\lambda h_1
\end{cases}\end{equation}
and for any $k\geq 0$, on $\RR$,
\[
\big|\py^k(h_1+\cos\tau y)|+|\py^k(h_2+\cos\tau y)\big|\lesssim\omega, \quad
\big|\py^k \PO h_1|+|\py^k\PO h_2\big|\lesssim 1+|y|.
\]
Setting
$ g_1 =(S^*)^2 h_1$ and $g_2 = \tfrac1{2\lambda}\LP g_1$,
the pair $(g_1,g_2)$ satisfies~\eqref{eq:ff} and, for any $k\geq 0$, on $\RR$,
\begin{align*}
&\big|\py^k\bigl(g_1 -\bigl( 2( {Q'}/{Q})\sin\tau y + Q^2\cos\tau y\bigr)\bigr)\big|+
\big|\py^k\bigl(g_2 - 2({Q'}/{Q})\sin\tau y\bigr)\big|\lesssim\omega,\\
&\big|\py^k \PO g_1|+|\py^k\PO g_2\big|\lesssim 1+|y|.
\end{align*}
Moreover,
\begin{equation}\label{eq:og}
\langle g_1,\Qo \rangle =\langle g_2, \Lo\Qo\rangle
 =\langle g_1, V_2 \rangle =\langle g_2 , V_1\rangle = 0.
\end{equation}
\end{lemma}
We note that Lemma \ref{le:gf} is related 
to \cite[Lemma 6.3]{KrSc}.
\begin{proof}
Setting
\[
l_1 = \tfrac 12(h_1+h_2),\quad l_2 = \tfrac 12(h_1 - h_2),
\]
from~\eqref{eq:gv}, we look for $(l_1,l_2)$ satisfying
\[\begin{cases}
-l_1'' -\left( 2\lambda-1\right) l_1 +\frac 13\omega\Qo^4(- l_1 + 2 l_2) =0\\
-l_2'' +\left( 2\lambda+1\right) l_2 +\frac 13\omega\Qo^4(2l_1-l_2)= 0
\end{cases}\]
Setting $l_1 = \check l_1 - \cos (\tau y )$, $ l_2 = \check l_2$,
where $\tau =\sqrt{2\lambda-1}$,
we look for $(\check l_1,\check l_2)$ such that
\[\begin{cases}
-\check l_1'' -\tau^2\check l_1 =\frac 13\omega\Qo^4(\check l_1-2\check l_2)
-\frac 13\omega\Qo^4\cos\left(\tau y\right)\\
-\check l_2'' +(2+\tau^2)\check l_2 =\frac 13\omega\Qo^4(-2\check l_1+\check l_2)
+\frac 23\omega\Qo^4\cos\left(\tau y\right)
\end{cases}\]
We define a bounded linear map $\check \Upsilon:(\cC_b(\RR))^2\to (\cC_b(\RR))^2$,
where $\cC_b(\RR)$ is the space of bounded continuous functions on $\RR$
equipped with the supremum norm $\|\cdot\|_\infty$, by setting
\[
\check \Upsilon\begin{pmatrix}\check l_1\\ \check l_2\end{pmatrix}
= \frac\omega3 \begin{pmatrix}-\frac1\tau\int_0^y\sin(\tau(y-y'))\Qo^4(y')(\check l_1-2\check l_2)(y') d y' \\[3pt]
 \frac1{2\sqrt{2+\tau^2}}\int \ee^{-\sqrt{2+\tau^2}|y-y'|}
\Qo^4(y')(-2\check l_1+\check l_2)(y') d y' 
\end{pmatrix}
\]
We also define
\begin{align*}
\check f_1 & = \frac\omega3\frac 1\tau\int_0^y\sin(\tau(y-y'))\Qo^4(y')\cos(\tau y') d y',\\ 
\check f_2 & = \frac\omega3\frac 1{\sqrt{2+\tau^2}}
 \int\ee^{-\sqrt{2+\tau^2}|y-y'|}\Qo^4(y')\cos(\tau y') d y'.
\end{align*}
so that the integral formulation of the system
(for even functions satisfying $\check l_1(0)=0$ by convention) 
is written
\begin{equation}\label{eq:if}
 \begin{pmatrix}\check l_1\\ \check l_2\end{pmatrix}
=\check \Upsilon\begin{pmatrix}\check l_1\\ \check l_2\end{pmatrix}
+\begin{pmatrix}\check f_1\\ \check f_2\end{pmatrix}.
\end{equation}
Since $\max(\|\check f_1\|_\infty,\|\check f_2\|_\infty)\leq\check C\omega$
on $\RR$ for a constant $\check C>0$ 
and $|||\check \Upsilon||| \lesssim \omega$, for $\omega>0$ sufficiently small,
it is elementary to prove by a fixed point argument
(or a Neumann series) that there exists
a solution $(\check l_1,\check l_2)$
of \eqref{eq:if} satisfying 
$\max(\|\check l_1\|_\infty,\|\check l_2\|_\infty)\leq 2\check C\omega$.
Moreover, the functions are smooth in $y$
and similar estimates for the derivatives of $(\check l_1,\check l_2)$ are easily checked.
The regularity with respect to $\omega$ is also checked
in the fixed point argument, and it is easy to see that
$\big|\py^k \PO l_1|+|\py^k\PO l_2\big|\lesssim 1+|y|$, for all $k\geq 0$.
In what follows, $\cOO$ denotes any smooth function $g$ of $\omega$ and $y$,
possibly different from one line to another,
and such that $|\py^k g|\lesssim \omega$
and $|\py^k\PO g|\lesssim 1+|y|$, on $\RR$, for all $k\geq 0$.
In particular, $\check l_1=\cOO$ and $\check l_2=\cOO$.
Setting $h_1=l_1+l_2=-\cos (\tau y) + \check l_1 + \check l_2$,
$h_2 = - \cos (\tau y) + \check l_1 - \check l_2$, we check that
$(h_1,h_2)$ satisfies \eqref{eq:gv} and the estimates of the lemma
and $h_1 = -\cos (\tau y)+\cOO$ and $h_2=-\cos (\tau y) + \cOO$.

Using~Lemma~\ref{LE:LM}, we see that the pair $(g_1,g_2)$ 
defined in the statement of the lemma 
satisfies \eqref{eq:ff}.
Moreover,
\begin{align*}
g_1 &=\frac{\Qo''}{\Qo} h_1+ 2\frac{\Qo'}{\Qo} h_1' + h_1''
= (1-Q^2) h_1+ 2\frac{Q'}{Q} h_1' + h_1'' + \cOO\\
&= Q^2 \cos \tau y + 2\frac{Q'}{Q} \sin \tau y + \cOO.
\end{align*}
Using $Q''=Q-Q^3$, $(Q')^2=Q^2-\frac 12 Q^4$,
so that $(Q^2)''=4Q^2-3Q^4$ and $(Q'/Q)' = -\frac 12 Q^2$,
we obtain
\[
g_2 = \frac 12 (-g_1''+g_1 - 3 Q^2 g_1) + \cOO
=2\frac{Q'}{Q} \sin \tau y + \cOO.
\]
Lastly, we prove~\eqref{eq:og}.
Using $\LP\Lo\Qo = -\Qo$ and then $\LM\Qo= 0$, we have
\[
4\lambda^2\langle \Lo\Qo, g_2\rangle= {2\lambda}\langle \Lo\Qo, \LP g_1 \rangle
=- {2\lambda}\langle \Qo, g_1\rangle
=- \langle \Qo,\LM g_2 \rangle= 0.
\]
The last two relations in \eqref{eq:og} are obtained 
using the equations of $(V_1,V_2)$ and $(g_1,g_2)$.
\end{proof}

We denote
\begin{align*}
G &= V_1^2(3\Qo + 10\omega\Qo^3), 
&G_1 &= G - H,\\
H& = V_2^2(\Qo + 2\omega\Qo^3) ,
& G_2 &= 2 V_1V_2(\Qo+2\omega\Qo^3)
\end{align*}
and
\begin{equation}\label{eq:GG}
G_1^\top = G_1 -\frac{\langle G_1,V_1\rangle}{\langle V_1,V_2\rangle} V_2,\quad
G_2^\perp = G_2 -\frac{\langle G_2, V_2\rangle}{\langle V_1,V_2\rangle} V_1.
\end{equation}
We define the quantity
\begin{equation}\label{eq:Ga}
\Gamma(\omega) =\int(G_1^\top g_1 + G_2^\perp g_2).
\end{equation}

\begin{lemma}\label{LE:Fg}
For
$\omega>0$ small,
\[
\Gamma(\omega) = \Gamma_0 \omega + O(\omega^2)\quad\mbox{where}\quad
\Gamma_0=\frac{32\pi\sqrt{2}}{3\cosh(\frac\pi2)}.
\]
\end{lemma}
\begin{proof}
First, by~\eqref{eq:og}, we have
$\Gamma =\int(G_1 g_1 + G_2 g_2)$.
Then, $g_2 =\frac1{2\lambda}\LP g_1$ implies
\[
\Gamma =\int g_1\bigl( G_1 +\tfrac 1{2\lambda}\LP G_2\bigr).
\]
Using \ref{it:eV} of Lemma~\ref{LE:VW} and \eqref{eq:eQ}, we have
the expansion
\[
G_1 = 3 Q(1- Q^2)^2 - Q +\omega \Delta_1 +\omega^2Q\tilde G_1,
\]
at the first order in $\omega$,
where
\[
\Delta_1 = 6Q(1- Q^2) R_1 + (1- Q^2)^2(3E + 10 Q^3)-2QR_2-(E+2 Q^3)
\]
and where the error term 
$\tilde G_1$ and all its derivatives are bounded 
by $C(1+y^2)$ on $\RR$.
Similarly,
\[
\tfrac 12 G_2 = Q(1- Q^2) +\omega \Delta_2 +\omega^2Q\tilde G_2
\]
where
\[
\Delta_2 = Q R_1 + Q(1- Q^2) R_2 + (1- Q^2)(E+2Q^3)
\]
and where the function $\tilde G_2$ and all its derivatives are bounded 
by $C(1+y^2)$ on $\RR$.
Using $\lambda = 1 + O(\omega^2)$ (\ref{it:dV} of Lemma~\ref{LE:VW})
and \eqref{eq:eQ}, we have
\[
G_1 +\tfrac 1{2\lambda}\LP G_2 = G_1 +\tfrac 12 (-G_2''+G_2 -3Q^2G_2)
- \tfrac 12\omega Q G_2 (6E+5Q^3) + O(\omega^2)Q.
\]
By 
$Q'' -Q +Q^3 =0$ and 
$(Q')^2 -Q^2 +\frac 12Q^4 = 0$,
we compute
\[
3 Q(1- Q^2)^2 - Q - \left(Q(1- Q^2)\right)'' + Q(1- Q^2) -3 Q^3(1- Q^2)=2 Q.
\]
Thus, we obtain
\[
G_1 +\tfrac 1{2\lambda}\LP G_2 
= 2 Q +\omega\Delta_3 + O(\omega^2)Q
\]
where
\[
\Delta_3 
= \Delta_1 - \Delta_2'' + \Delta_2-3 Q^2\Delta_2
- Q^2(1- Q^2)(6E+5 Q^3).
\]
By $g_1 = \frac 1{2\lambda} \LM g_2$ and $\LM\Qo=0$, we have
$\int g_1\Qo=0$, and so by~\eqref{eq:eQ},
\[
\int g_1 Q =\int g_1( Q-\Qo) 
= -\omega\int g_1 E + O(\omega^2).
\]
Therefore,
\[
\Gamma =\int g_1\left(2 Q +\omega\Delta_3\right) + O(\omega^2)
=\omega\int g_1(- 2E+\Delta_3) + O(\omega^2).
\]
Note that
\[
-2E+\Delta_3 = \Delta_4 - \Delta_2'' + \Delta_2-3 Q^2\Delta_2 
\]
where $\Delta_4 = -2E+ \Delta_1 
- Q^2(1- Q^2)(6E+5 Q^3)$. 
Since $-g_1''+g_1 - 3 Q^2 g_1 
=2 g_2+ O(\omega)$, we obtain
\[
\Gamma 
=\omega\int g_1(\Delta_4 - \Delta_2'' + \Delta_2-3 Q^2\Delta_2 )+O(\omega^2)
= \omega \int g_1 \Delta_4
+2\omega\int g_2\Delta_2+ O(\omega^2).
\]
Using  Lemma \ref{le:gf} and
$|\tau-1|\lesssim \alpha^2 \lesssim \omega^2$,
we have
\[
g_1 = Q^2\cos y + 2\frac{Q'}{Q}\sin y + (1+|y|) O(\omega),\quad
g_2 = 2\frac{Q'}{Q}\sin y + (1+|y|) O(\omega),
\]
and so
$
\Gamma
= \Gamma_0 \omega +O(\omega^2)
$
where the universal constant $\Gamma_0$ is defined by
\[
\Gamma_0
= \int Q^2\Delta_4\cos y + 2 \int \frac{Q'}Q(\Delta_4+2\Delta_2)\sin y.
\] 
To compute $\Gamma_0$ explicitly, 
we express it
as a linear combination of elementary integrals.
After lengthy computations, not reproduced here, we find
\begin{align*}
\Gamma_0 & = - \tfrac{80}9 p_1 +\tfrac{372}{9} p_3
+\tfrac{2446}{25}p_5-\tfrac{9613}{63}p_7+\tfrac{1312}{27}p_9
 - \tfrac{128}{9} q_1 +\tfrac{128}3 q_3 - \tfrac{2624}{45} q_5
+\tfrac{64}3 q_7\\
&\quad -32 r_1 -124r_3 + 388 r_5 - 168 r_7 
 +16 s_1 + 108 s_3 + 156 s_5-168 s_7
\end{align*}
where for $k\geq 1$, we have defined
\begin{align*}
p_k &=\int Q^k\cos y, \quad
q_k =\int Q^k\ln (Q/\sqrt{8}) \cos y, \\
r_k &=\int T_2 Q^k\cos y,\quad
s_k =\int T_2 Q^{k-1} Q'\sin y.
\end{align*} 
Then, by integration by parts, one easily checks the relations,
for $k\geq 1$,
\begin{align*}
p_{k+2} &= \frac{2(k^2+1)}{k(k+1)} p_k,\quad
q_{k+2} = \frac{2(k^2+1)}{k(k+1)}q_k
+ \frac{2(k^2-2k-1)}{k^2(k+1)^2} p_k,\displaybreak[0]\\
r_{k+2} &= \frac 2{k(k+1)}
\left( (k^2-3) r_k -2ks_k -\frac 23 p_{k+4} \right),\displaybreak[0]\\
s_{k+2} &= \frac 2{k(k+1)(k+2)}
\left( 6 r_k + k(k^2+1) s_k + \frac {2(3k+8)}{3(k+4)} p_{k+4}
\right),
\end{align*}
which allow us to deduce inductively the values of
$p_k$, $q_k$, $r_k$ and $s_k$ for any odd integer $k\geq 3$
in terms of $p_1$, $q_1$, $r_1$ and $s_1$.
Then, inserting these values into the above expression of
$\Gamma_0$, we obtain a linear combination of
$p_1$, $q_1$, $r_1$ and $s_1$ with rational coefficients.
Actually, all the occurrences of $q_1$, $r_1$ and $s_1$ disappear
in this linear combination,
which is surprising but helpful, and we find the simple formula
\[
\Gamma_0= \frac{32}3 p_1=
\frac{32\pi\sqrt{2}}{3\cosh(\frac\pi2)}
\]
where we have used
$p_1= \pi\sqrt{2} /\cosh(\frac\pi2)$
computed by the residue Theorem.
\end{proof}

\section{Estimate of the internal mode component}\label{se:07}

In this section, we estimate the internal mode component $ b$
in terms of a local norm of~$v$.
The proof is inspired by~\cite[Proof of Proposition 2]{KoMa}
for scalar field models.
However, the Fermi golden rule established in Lemma~\ref{LE:Fg}
is one of the key ingredient here.

\begin{lemma}\label{LE:eb}
For any $s>0$,
\[
 \int_0^s| b|^4 
\lesssim \varepsilon +\frac 1{A\oz^2}\int_0^s\|\rho^4 v\|^2 .
\]
\end{lemma}

\begin{remark}
Exponent $4$ for $|b|$ versus exponent $2$ for the local norm of $v$
is a key feature of the control of the internal mode component.
Formally, it illustrates the fact that the internal mode component $b$ has a
slower decay in time.

The constraints on the parameters $\oz$, $A$ and $\varepsilon$ follow the same
rules as in the proof of Lemma \ref{LE:v1}. See the remark after Lemma \ref{LE:v1}.
\end{remark}

\begin{proof}
We introduce
\[
d_1 = b_1^2 - b_2^2,\quad d_2 = 2 b_1 b_2.
\]
(Equivalently, $d=d_1+\ii d_2 = b^2$.)
Using~\eqref{eq:bb}, we have
\begin{equation}\label{eq:dd}
\begin{cases}
\dot d_1 = 2\lambda d_2 + D_2\\
\dot d_2 = - 2\lambda d_1 + D_1
\end{cases}
\end{equation}
where
$D_2= 2 b_1 B_2 + 2 b_2 B_1$ and $D_1= 2 b_2 B_2-2b_1 B_1$.
Moreover,
\begin{equation}\label{eq:nb}
\frac{d}{ds}(| b|^2) = 2 b_1 B_2 - 2 b_2 B_1.
\end{equation}
Recall $g_1$, $g_2$ defined in Lemma~\ref{le:gf}
and the notation in~\eqref{eq:GG},~\eqref{eq:Ga}.
We also set
\[
\Gamma_1 =\frac 12\int(G^\top + H^\top)g_1,\quad
\Gamma_2 =\frac 14\int(G_1^\top g_1 - G_2^\perp g_2).
\]
We define the function $\bJ$ by
\[
\bJ = d_1\int v_2 g_1\chi_A - d_2\int v_1 g_2\chi_A
+\Gamma_1\frac{d_2}{2\lambda}| b|^2 +\Gamma_2\frac{d_1 d_2}{2\lambda}
\]
where $\chi_A$ is defined in~\eqref{eq:ch}.
Firstly, we note that
\begin{equation}\label{eq:eJ}
|\bJ| \lesssim A^\frac12|b|^2\|v\| + |b|^4
\lesssim A^\frac12 \varepsilon^3 .
\end{equation}
Secondly, we start computing $\dot\bJ$
\begin{align*}
\dot\bJ & =\dot d_1\int v_2 g_1\chi_A - d_2\int\dot v_1 g_2\chi_A
+ d_1\int\dot v_2 g_1\chi_A -\dot d_2\int v_1 g_2\chi_A\\
&\quad 
+\Gamma_1\frac{\dot d_2}{2\lambda}| b|^2 +\Gamma_1\frac{d_2}{2\lambda}\frac d{ds}(| b|^2)
+\Gamma_2\frac{\dot d_1 d_2 + d_1\dot d_2}{2\lambda} + \bJ_6
\end{align*}
where $\bJ_6$ is an error term defined by
\[
\bJ_6 = d_1\int v_2\dot g_1\chi_A - d_2\int v_1\dot g_2\chi_A
+\dot\Gamma_1\frac{d_2}{2\lambda}| b|^2 
+\dot\Gamma_2\frac{d_1 d_2}{2\lambda}-\dot \lambda\Gamma_1\frac{d_2}{2\lambda^2}| b|^2 
- \dot \lambda\Gamma_2\frac{d_1 d_2}{2\lambda^2}
\]
and to be estimated later.
We insert~\eqref{eq:vv},~\eqref{eq:dd} and~\eqref{eq:nb}
in the expression for $\dot \bJ$.
First,
\begin{align*}
\dot d_1\int v_2 g_1\chi_A - d_2\int\dot v_1 g_2\chi_A 
 &= 
2\lambda d_2\int v_2 g_1\chi_A - d_2\int(\LM v_2) g_2\chi_A
+ D_2\int v_2 g_1\chi_A\\
&\quad
- d_2\int(\mu_2 +p_2^\perp-q_2^\perp-r_2^\perp) g_2\chi_A.
\end{align*}
Using~\eqref{eq:ff},
\begin{align*}
\int(\LM v_2) g_2\chi_A
& =\int v_2(\LM g_2)\chi_A -\int v_2 g_2\chi_A'' - 2\int v_2 g_2'\chi_A'\\
& = 2\lambda\int v_2 g_1\chi_A + \int v_2 g_2\chi_A'' + 2\int (\py v_2) g_2\chi_A'.
\end{align*}
Thus,
\begin{align*}
\dot d_1\int v_2 g_1\chi_A - d_2\int\dot v_1 g_2\chi_A
&= -d_2\int v_2 g_2\chi_A'' - 2 d_2\int (\py v_2) g_2\chi_A'
+ D_2\int v_2 g_1\chi_A\\
&\quad 
- d_2\int(\mu_2+p_2^\perp-q_2^\perp-r_2^\perp) g_2\chi_A.
\end{align*}
Similarly,
\begin{align*}
 d_1\int\dot v_2 g_1\chi_A -\dot d_2\int v_1 g_2\chi_A
 & = -d_1\int v_1 g_1\chi_A'' - 2 d_1\int (\py v_1) g_1\chi_A'
- D_1\int v_1 g_2\chi_A\\
&\quad - d_1\int(\mu_1 + p_1^\top-q_1^\top-r_1^\top) g_1\chi_A.
\end{align*}
Next,
\[
\Gamma_1\frac{\dot d_2}{2\lambda}| b|^2 
+\Gamma_1\frac{d_2}{2\lambda}\frac d{ds}(| b|^2) 
= -\Gamma_1 d_1| b|^2 + \frac{\Gamma_1}{\lambda} \bigl(b_1(b_1+b_2)^2B_2
+b_2(b_1-b_2)^2B_1 \bigr)
\]
and
\[
\Gamma_2\frac{\dot d_1 d_2 + d_1\dot d_2}{2\lambda}
= -\Gamma_2\left( d_1^2 -d_2^2\right)
+\frac {\Gamma_2}{ \lambda}(b_1|b|^2B_1+b_2(3b_1^2-b_2^2)B_2).
\]
Therefore
\[
\dot{\bJ} =\bJ_1+\bJ_2+\bJ_3+\bJ_4+\bJ_5+\bJ_6,
\]
where the main term $\bJ_1$, containing all the 
terms of order $4$ in $b$, is defined by
\[
\bJ_1 = d_2\int q_2^\perp g_2\chi_A + d_1\int q_1^\top g_1\chi_A
-\Gamma_1 d_1| b|^2 -\Gamma_2\left( d_1^2 -d_2^2\right)
\]
and 
\begin{align*}
\bJ_2 & = d_2\int v_2 g_2\chi_A'' + 2 d_2\int(\py v_2) g_2\chi_A'
 + d_1\int v_1 g_1\chi_A'' + 2 d_1\int(\py v_1) g_1\chi_A'\\
\bJ_3 & = -d_2\int(\mu_2+p_2^\perp-r_2^\perp) g_2\chi_A
-d_1\int(\mu_1+p_1^\top- r_1^\top) g_1\chi_A\displaybreak[0]\\
\bJ_4 & = D_2\int v_2 g_1\chi_A - D_1\int v_1 g_2\chi_A\\
\bJ_5 & =
\frac{\Gamma_1}{\lambda} \bigl(b_1(b_1+b_2)^2B_2
+b_2(b_1-b_2)^2B_1 \bigr)
+\frac {\Gamma_2}{ \lambda}(b_1|b|^2B_1+b_2(3b_1^2-b_2^2)B_2).
\end{align*}
We decompose further
$\bJ_1 =\bJ_{1,1}+\bJ_{1,2}+\bJ_{1,3}$
where
\begin{align*}
\bJ_{1,1} & =d_1\Bigl(b_1^2\int G^\top g_1
+ b_2^2\int H^\top g_1-\Gamma_1| b|^2\Bigr)
 + d_2 b_1 b_2\int G_2^\perp g_2
-\Gamma_2\left( d_1^2 -d_2^2\right),\\
\bJ_{1,2} & =d_2\int(q_2^\perp\chi_A -b_1 b_2 G_2^\perp)g_2,\quad 
\bJ_{1,3} = d_1\int(q_1^\top\chi_A - b_1^2 G^\top - b_2^2 H^\top)g_1.
\end{align*}
We observe that
\[
b_1^2\int G^\top g_1+ b_2^2\int H^\top g_1-\Gamma_1| b|^2
=\tfrac 12 d_1\int G_1^\top g_1
\]
and thus
\[
\bJ_{1,1} =\frac 12 d_1^2\int G_1^\top g_1+\tfrac 12 d_2^2\int G_2^\perp g_2
-\frac 14\left( d_1^2 -d_2^2\right)\int(G_1^\top g_1 - G_2^\perp g_2)
 =\frac\Gamma4| d|^2 =\frac\Gamma4| b|^4 .
\]

Estimate of $\bJ_{1,2}$ and $\bJ_{1,3}$.
From Lemma \ref{LE:Ty}, we have
\begin{align*}
q_1 & = b_1^2 G+b_2^2H +(3\Qo+10\omega\Qo^3)(2b_1V_1v_1 + v_1^2) 
+(\Qo+2\omega\Qo^3)(2b_2V_2v_2+v_2^2) + N_1,\\
q_2 &= b_1b_2 G_2+ 2(\Qo+2\omega\Qo^3)(b_1V_1v_2+b_2V_2v_1+v_1v_2) + N_2,
\end{align*}
where 
$|N_1|+|N_2|\lesssim |
u|^3 \lesssim|b|^3\rho^{24} +|v|^3$,
using \eqref{eq:eV}.
For $\bJ_{1,2}$, by definition of $q_2^\perp$, we have
\[
\int q_2^\perp g_2\chi_A =\int q_2 g_2\chi_A
 -\frac{\langle q_2,V_2\rangle}{\langle V_1,V_2\rangle}\int V_1 g_2\chi_A.
\]
The relation $\int V_1g_2=0$ (see~\eqref{eq:og}),
the fact that $1-\chi_A\equiv 0$ for $|y|<A$
and the decay properties of $V_1$ from \ref{it:bW} Lemma~\ref{LE:VW} show that
\[
\Bigl|\int V_1 g_2\chi_A\Bigr| =\Bigl|\int V_1 g_2(1-\chi_A)\Bigr|
\lesssim\int|V|(1-\chi_A)\lesssim\int_{|y|>A} \ee^{-\alpha|y|} dy
\lesssim\frac 1\oz \ee^{-\frac12{\oz A}}.
\]
Using $\langle V_1,V_2\rangle\gtrsim \oz^{-1}$ and
\begin{align*}
|\langle q_2,V_2\rangle| &\lesssim 
\int(|b|^2+|v|^2)\Qo +\int(|b|^3\rho^{24}+|v|^3)|V|\\
&\lesssim|b|^2+\|\nu v\|^2
+\frac 1\oz|b|^3+\|v\|_{L^\infty}\|\rho^4v\|^2
\lesssim|b|^2+\|\rho^4v\|^2,
\end{align*}
we obtain
\[
\Bigl|\frac{\langle q_2,V_2\rangle}{\langle V_1,V_2\rangle}\int V_1 g_2\chi_A\Bigr|
\lesssim \ee^{-\frac12\oz A}\left(|b|^2+\|\rho^4v\|^2\right).
\]
Moreover, by \eqref{eq:og},
$\int G_2^\perp g_2 =\int G_2 g_2$.
Using the expansion of $q_2$ above, we have
\begin{align*}
\int(q_2\chi_A -b_1 b_2 G_2)g_2
& = - b_1 b_2\int G_2 g_2(1-\chi_A)\\
&\quad + 2\int(\Qo+2\omega\Qo^3)( b_1 V_1 v_2 + b_2 V_2v_1 + v_1 v_2) g_2\chi_A
 +\int N_2 g_2\chi_A.
\end{align*}
Now, we estimate the three terms on the right hand side of the above identity.
First,
\[
\Bigl| b_1 b_2\int G_2 g_2(1-\chi_A)\Bigl|
\lesssim|b|^2\int\nu^{10}(1-\chi_A)\lesssim|b|^2 \ee^{-A}.
\]
Second,
\[
\Bigl|\int(\Qo+2\omega\Qo^3)(b_1V_1 v_2 + b_2V_2 v_1) g_2\chi_A\Bigr|
\lesssim|b|\int\nu^{10}|v|\lesssim|b|\|\nu v\|
\]
and
\[
\Bigl|\int(\Qo+2\omega\Qo^3) v_1 v_2 g_2\chi_A\Bigr|
\lesssim\|\nu v\|^2.
\]
Third, using~\ref{it:sv} of Lemma \ref{LE:sm}
and the definition of the cut-off function $\chi_A$,
\[
\Bigl|\int N_2 g_2 \chi_A\Bigr|
\lesssim|b|^3\int\rho^{24}
+\|v\|_{L^\infty}\int_{|y|<2A}|v|^2
\lesssim\frac\varepsilon\oz|b|^2 +\varepsilon\|\eta_A v\|^2.
\]
Therefore, for $A$ large(depending on $\oz$),
\[
|\bJ_{1,2}|\lesssim
 \left( \ee^{-\frac 12\oz A} +\varepsilon/\oz\right)
|b|^4+|b|^2\|\rho^4 v\|^2
+\varepsilon|b|^2\|\eta_A v\|^2
+|b|^3 \|\nu v\|.
\]
Using the expression of $q_1$, the estimate for $\bJ_{1,3}$ is the same
\[
|\bJ_{1,3}|\lesssim
 \left( \ee^{-\frac 12\oz A} +\varepsilon/\oz\right)
|b|^4+|b|^2\|\rho^4 v\|^2 
+\varepsilon|b|^2\|\eta_A v\|^2+|b|^3 \|\nu v\|.
\]

Estimate of $\bJ_2$.
By the definition of $\chi_A$ in~\eqref{eq:ch}, the bound
$|g_1|+|g_2|\lesssim 1$, and the definition of $\eta_A$, one has
\begin{align*}
|\bJ_2|
&\lesssim\frac{1}{A^2}|d|\int_{|y|<2A}|v|
+\frac{1}{A}|d|\int_{|y|<2A}|\py v|\\
&\lesssim\frac{1}{A\sqrt{A}}|b|^2\Bigl(\int_{|y|<2A}|v|^2\Bigr)^\frac12
+\frac{1}{\sqrt{A}}|b|^2\Bigl(\int_{|y|<2A}|\py v|^2\Bigr)^\frac12\\
&\lesssim\frac{|b|^2}{\sqrt{A}}\Bigl(\|\eta_A\py v\|^2
+\frac{1}{A^2}\|\eta_A v\|^2\Bigr)^\frac12.
\end{align*}

Estimate of $\bJ_3$.
For the two terms containing $\mu_1$ and $\mu_2$ in $\bJ_3$ we use orthogonality relations.
Indeed, by the definitions of $\mu_1$, $\mu_2$ and~\eqref{eq:og}, we have
\[
\int\mu_1 g_1\chi_A 
=-\int\mu_1 g_1(1-\chi_A),\quad
\int\mu_2 g_2\chi_A 
=-\int\mu_2 g_2(1-\chi_A).
\]
Thus, using also~\eqref{eq:Xv}, we obtain
\begin{align*}
\left|\int\mu_1 g_1\chi_A\right|
+\left|\int\mu_2 g_2\chi_A\right| 
&\lesssim\int(|\mu_1|+|\mu_2|)(1-\chi_A)\\
&\lesssim(|m_\gamma|+|m_\omega|) 
\int\nu^{10}(1-\chi_A)
\lesssim \ee^{-A}(|b|^2 +\|\nu v\|^2).
\end{align*}
Moreover, we have
\[
\left|\int p_2^\perp g_2\chi_A\right|
\lesssim\int_{|y|<2A}|p_2| 
+\frac{\int|p_2||V_2|}{|\langle V_1,V_2\rangle|}\int|V_1|
\lesssim\int_{|y|<2A}|p_2| +\int\rho^{8}|p_2|.
\]
We estimate
\begin{align*}
\int_{|y|<2A}|p_2|&\lesssim(|m_\gamma|+|m_\omega|) 
\int_{|y|<2A}(|u|+|y| |\py u|)\\
&\lesssim A^\frac32 \|u\|_{H^1} (|b|^2 +\|\nu v\|^2)
\lesssim \varepsilon A^\frac32(|b|^2 +\|\nu v\|^2)
\end{align*}
and, for $A$ large enough,
\begin{align*}
\int|p_2|\rho^{8}
&\lesssim(|m_\gamma|+|m_\omega|) 
\int(|u|+|y| |\py u|)\rho^8\\
&\lesssim\oz^{-\frac 32}\|u\|_{H^1}(|b|^2 +\|\nu v\|^2)
\lesssim \varepsilon A^\frac32(|b|^2 +\|\nu v\|^2).
\end{align*}
Estimating similarly $\int p_1^\perp g_1\chi_A$, we obtain
\[
\Bigl| \int p_1^\perp g_1\chi_A\Bigr|+\Bigl| \int p_2^\top g_2\chi_A\Bigr|\lesssim 
\varepsilon A^\frac32(|b|^2 +\|\nu v\|^2).
\]
Lastly,
since $|\omega\partial_\omega V|\lesssim\rho^8$, we have
$|r|\lesssim|m_\omega||b|\rho^8$ and so for $A$ large,
\begin{align*}
\left|\int r_2^\perp g_2\chi_A\right|
+\left|\int r_1^\perp g_1\chi_A\right|
 \lesssim \int |r| 
\lesssim \frac{|b|}\oz(|b|^2 +\|\nu v\|^2)
\lesssim \varepsilon A^\frac32(|b|^2 +\|\nu v\|^2)
\end{align*}
Thus,
\[
|\bJ_3|\lesssim \bigl(\ee^{-A} +\varepsilon A^\frac 32\bigr)|b|^2(|b|^2 +\|\nu v\|^2).
\]

Estimate of $\bJ_4$.
Using \eqref{eq:eB}, we have
\[
|\bJ_4| 
\lesssim|D|\int_{|y|\leq 2 A}|v| 
\lesssim\sqrt{A}|b||B|\|\eta_A v\|
\lesssim\oz\sqrt{A}|b|(|b|^2 +\|\rho^4 v\|^2)\|\eta_A v\|.
\]

Estimate of $\bJ_5$.
Using~\eqref{eq:eB}, we have
\[
|\bJ_5|\lesssim |b|^3|B|\lesssim 
\oz|b|^3(|b|^2 +\|\rho^4 v\|^2).
\]

Estimate of $\bJ_6$.
Using $|\dot g_1|+|\dot g_2| \lesssim |\dot \omega|(1+|y|)$ on $\RR$
(from Lemma \ref{le:gf}),
and then \eqref{eq:Xv},
\begin{align*}
\Bigl|d_1\int v_2\dot g_1\chi_A\Bigr|+ \Bigl|d_2\int v_1\dot g_2\chi_A\Bigr|
&\lesssim|b|^2 |\dot \omega|
\biggl(\int_{|y|<2A}(1+|y|)^2\biggr)^\frac12
\|\eta_A v\|\\
&\lesssim A^\frac32 |b|^2\left(\|\nu v\|^2 + |b|^2\right) \|\eta_A v\|.
\end{align*}
Lastly, using $|\dot\Gamma_1|+|\dot\Gamma_2|\lesssim |\dot \omega|$
and $|\dot \lambda|\lesssim \omega|\dot \omega|$,
\[
\Bigl|\dot\Gamma_1\frac{d_2}{2\lambda}| b|^2\Bigr| 
+\Bigl|\dot\Gamma_2\frac{d_1 d_2}{2\lambda}\Bigr| 
+\Bigl| \dot \lambda\Gamma_1\frac{d_2}{2\lambda^2}| b|^2 \Bigr|
+\Bigl|\dot \lambda\Gamma_2\frac{d_1 d_2}{2\lambda^2}\Bigr|
\lesssim |\dot \omega| |b|^4
\lesssim \left(\|\nu v\|^2 + |b|^2\right) |b|^4.
\]
Thus,
\[
|\bJ_6| \lesssim A^\frac32 |b|^2\left(\|\nu v\|^2 + |b|^2\right) \|\eta_A v\|
+\left(\|\nu v\|^2 + |b|^2\right) |b|^4.
\]

Gathering the above estimates, we have
\begin{align*}
\Big|\dot \bJ - \frac\Gamma4 |b|^4 \Big|
& \lesssim
 \bigl( \ee^{-\frac 12\oz A} +\varepsilon/\oz+ \varepsilon A^\frac 32\bigr)
|b|^4+(1 + \varepsilon A^\frac 32)|b|^2\|\rho^4 v\|^2 
+\varepsilon|b|^2\|\eta_A v\|^2\\
&\quad + \frac 1{\sqrt{A}} \ |b|^2 \Bigl(\|\eta_A\py v\|^2
+\frac{1}{A^2}\|\eta_A v\|^2\Bigr)^\frac12
+\oz\sqrt{A}|b|(|b|^2 +\|\rho^4 v\|^2)\|\eta_A v\|.
\end{align*}
From Lemma~\ref{LE:Fg}, $\Gamma=\omega\Gamma_0 + O(\omega^2)$ for a constant $\Gamma_0>0$.
Thus, for $\oz$ sufficiently small, for $A$ sufficiently large, 
and then for $\varepsilon$ sufficiently small, we have
\[
\oz|b|^4\leq C_1 \dot \bJ + \frac{C_2}{A\oz}\Bigl(\|\eta_A\py v\|^2
+\frac{1}{A^2}\|\eta_A v\|^2\Bigr)
\]
for two constants $C_1,C_2>0$.
Integrating the above estimate on $[0,s]$ for any $s\geq 0$,
using \eqref{eq:eJ} and then Lemma \ref{LE:v1}, we have proved
\begin{align*}
\int_0^s |b|^4
&\lesssim \frac 1{\oz}(|\bJ(s)|+|\bJ(0)|)
+\frac{1}{A\oz^2}\int_0^s\Bigl(\|\eta_A\py v\|^2+\frac{1}{A^2}\|\eta_A v\|^2\Bigr)\\
&\lesssim \frac {A^\frac 12 \varepsilon^3}{\oz}
+\frac{\varepsilon}{A\omega_0^2}
+ \frac{1}{A\oz^2 }\int_0^s\Bigl(\|\rho^4 v\|^2 +|b|^4\Bigr)
\end{align*}
which implies the result by taking $A$ large enough and then $\varepsilon$ small enough.
\end{proof}

\section{The transformed problem}\label{se:08}

For~$\theta>0$ small to be fixed, we set
$\Xt =(1 -\theta\py^2)^{-1}$.
We define $w=w_1 +\ii w_2$ by
\[
w_1 =\Xt^2\MM\cS^2 v_2, \quad
w_2 = -\Xt^2\cS^2\LP v_1.
\]
The above will be called the first transformed problem.
Some notation is needed.
Let
\[
\xi_Q = \frac{\Qo'}{\Qo}, \quad \xi_W = \frac{W_2'}{W_2}.
\]
Then, using
\[
\Qo'' -\Qo +\Qo^3 +\omega\Qo ^5=0,\quad
(\Qo')^2 -\Qo^2 +\tfrac 12\Qo^4 +\tfrac\omega3\Qo^6 = 0,
\]
 we compute
 $S^2 = \py^2 - 2 \xi_Q \py + 1 + \frac \omega 3 \Qo^4$
and 
\begin{align*}
\MM S^2 
& = - \py^4 + 2 \py^2 \cdot\xi_Q \cdot \py 
-\tfrac 43 \py \cdot \Qo^4 \cdot \py
+ \left(-2 \xi_Q + \tfrac {14}3 \omega\Qo^4\xi_Q\right)\cdot \py\\
&\quad + 1 - 6 \omega \Qo^4 + \tfrac {10}3 \omega\Qo^6 + \tfrac 73\omega^2 \Qo^8,
\end{align*}
and
\begin{align*}
S^2 \LP 
& = - \py^4 + 2 \py^2 \cdot\xi_Q \cdot \py 
+\py \cdot \left(-\Qo^2-\tfrac 83 \omega\Qo^4 \right)\cdot \py\\
&\quad + \left(-2 \xi_Q-2 \Qo\Qo'-14\omega\Qo^3 \Qo'\right)\cdot \py\\
&\quad + 1 -3\Qo^2+3\Qo^4 -\tfrac{134}3 \omega \Qo^4 - 33\omega \Qo^6 + 25 \omega^4\Qo^8.
\end{align*}
The operators 
\begin{align*}
\QM
& = 
2 \py^2 \cdot \PO\xi_Q \cdot \py 
-\tfrac 43 \py \cdot \PO(\Qo^4) \cdot \py
+ \PO\left(-2 \xi_Q + \tfrac {14}3 \omega\Qo^4\xi_Q\right)\cdot \py\\
&\quad +\PO\left(- 6 \omega \Qo^4 + \tfrac {10}3 \omega\Qo^6 
 + \tfrac 73 \omega^2 \Qo^8\right)
\end{align*}
and
\begin{align*}
\QP
& =
2 \py^2 \cdot\PO\xi_Q\cdot \py 
+\py \cdot \PO\left(-\Qo^2-\tfrac 83 \omega\Qo^4 \right)\cdot \py\\
&\quad + \PO\left(-2 \xi_Q-2 \Qo\Qo'-14\omega\Qo^3 \Qo'\right)\cdot \py\\
&\quad +\PO\left(-3\Qo^2+3\Qo^4 -\tfrac{134}3 \omega \Qo^4 - 33\omega \Qo^6 + 25 \omega^4\Qo^8\right)
\end{align*}
 are introduced to take into account
the time dependency of the potentials involved in the operators 
$\MM\cS^2$ and $\cS^2\LP$.
From the equation~\eqref{eq:vv} of~$v$ and the identity of Lemma~\ref{LE:LM}, 
using $\cS^2\mu_1=\cS^2\LP\mu_2=0$,
we check that~$w$ satisfies the system
\begin{equation}\label{eq:ww}
\begin{cases}
\dot w_1 =\MM w_2 -[\Xt^2,\Qo^4]\cS^2\LP v_1 +\Xt^2 n_2\\[2pt]
\dot w_2 = -\MP w_1 +\frac 13 [\Xt^2,\Qo^4]\MM\cS^2 v_2 -\Xt^2 n_1\end{cases}
\end{equation}
with the notation
$[\Xt^2,\Qo^4] =\Xt^2\Qo^4 -\Qo^4\Xt^2$
and where
\begin{align*}
n_1 &= -\cS^2\LP p_2^\perp +\cS^2\LP q_2^\perp+\cS^2\LP r_2^\perp
+\dot \omega\QP v_1 ,
\\
 n_2 &= -\MM\cS^2 p_1^\top +\MM\cS^2 q_1^\top +\MM\cS^2 r_1^\top
+\dot \omega\QM v_2 .
\end{align*}
Now, for $\vartheta >\theta$ small to be chosen
(in the proof of Lemma \ref{LE:vz}, estimating $\bK_2$, we will eventually
choose $\vartheta=\theta^{1/4}$),
we introduce the second transformed problem,
defining $z=z_1+\ii z_2$ by
\[ 
z_1 =\Xu\cU w_2,\quad 
z_2 = -\Xu\cU\MP w_1.
\]
Note that $U = \py - \xi_W$ and
\begin{align*}
UM_+ & = -\py^3 +\py \cdot \xi_W \cdot \py 
+\py - \xi_W' \py + \tfrac\omega3\Qo^4 \py
-\xi_W -\tfrac\omega3 \Qo^4 \xi_W +\tfrac \omega3(\Qo^4)'.
\end{align*}
We set
\begin{align*}
\PP & = - \PO \xi_W , \\
\PM & = \py \cdot \PO\xi_W \cdot \py 
 - (\PO \xi_W') \py + \tfrac 13 \PO( \omega \Qo^4) \py
+\PO\bigl(-\xi_W -\tfrac \omega 3 \Qo^4 \xi_W +\tfrac \omega3(\Qo^4)'\bigr).
\end{align*}
Using~\eqref{eq:ww} and the identity in Lemma~\ref{LE:sf}, we find
\begin{equation}\label{eq:zz}
\begin{cases}
\dot z_1 = z_2 +\frac 13\Xu\cU [\Xt^2,\Qo^4]\MM\cS^2 v_2 -\Xu\cU\Xt^2 n_1
+ \dot \omega\Xu\PP w_2\\[2pt]
\dot z_2 = -\cK z_1
- [\Xu,\cK]\cU w_2
 +\Xu\cU\MP [\Xt^2,\Qo^4]\cS^2\LP v_1
-\Xu\cU\MP\Xt^2 n_2 - \dot \omega \Xu\PM w_1
\end{cases}
\end{equation}
where $[\Xu,\cK] =\Xu\cK -\cK\Xu$.
We now give several technical results, most of them 
adapted from \cite{KMM2,KMMV,Ma22}.

\begin{lemma}[{\cite[Lemma 9]{Ma22}}]\label{LE:tc}
For~$\theta> 0$ sufficiently small and all~$h\in L^2(\RR)$,
\begin{align*}
&\|\Xt h\|\leq\|h\|, \
\|\py\Xt^\frac12 h\|\leq\theta^{-\frac 12}\|h\|,\
\|\rho\Xt h\|\lesssim\|\Xt(\rho h)\|, \ \|\eta_A^{-1}\Xt(\eta_A h)\|\lesssim\|\Xt h\|,\displaybreak[0]\\
&\|\eta_A\Xt h\|\lesssim\|\Xt(\eta_A h)\|, \quad 
\|\eta_A\Xt\py h\|\lesssim\theta^{-\frac 12}\|\eta_A h\|,\quad
\|\eta_A\Xt\py^2h\|\lesssim\theta^{-1}\|\eta_A h\|,
\\
&\|\rho^{-1}\Xt(\rho h)\|
\lesssim\|\Xt h\|,\quad
\|\rho^{-1}\Xt\py(\rho h)\|\lesssim\theta^{-\frac 12}\|h\|,\quad
\|\rho^{-1}\Xt\py^2(\rho h)\|\lesssim\theta^{-1}\| h\|.
\end{align*}
\end{lemma}

\begin{lemma}[{\cite[Lemma 10]{Ma22}}]\label{LE:ML}
For~$\theta> 0$ sufficiently small and all~$h\in H^1(\RR)$,
\begin{align*}
\|\eta_A\Xt^2\MM S^2 h\| +\|\eta_A\Xt^2 S^2\LP h\|
&\lesssim\theta^{-2} \|\eta_A h\|,\\
\|\eta_A\Xt^2\MM S^2 h\| +\|\eta_A\Xt^2 S^2\LP h\|
&\lesssim\theta^{-\frac32}\|\eta_A\py h\| +\|\eta_A h\|,\\
\|\eta_A\py\Xt^2\MM S^2 h\| +\|\eta_A\py\Xt^2 S^2\LP h\|
&\lesssim\theta^{-2}\|\eta_A\py h\| +\|\eta_A h\|.
\end{align*}
\end{lemma}

\begin{lemma}\label{LE:UM}
For~$\theta> 0$ sufficiently small and all~$h\in H^1(\RR)$,
\begin{align*}
&\|\eta_A\py^2\Xt Uh\|+\|\eta_A\py\Xt Uh\|+\|\eta_A\Xt U h\|
\lesssim \theta^{-1}\|\eta_A\py h\|
+ \|\eta_A h\|,\\
&\|\eta_A\Xt \MP h\| 
\lesssim\theta^{-1} \|\eta_A h\|,
\quad
\|\eta_A\Xt U\MP h\| 
\lesssim\theta^{-1}\|\eta_A\py h\|
+\|\eta_A h\|.
\end{align*}
\end{lemma}

\begin{proof}
The first estimate is deduced from 
$\|\eta_A\Xt U h\|\lesssim\|\eta_A U h\|$ (Lemma~\ref{LE:tc}), 
and then the expression of $U=\py -\xi_W$
with $|\xi_W|\lesssim\oz$ from \ref{it:bW}-\ref{it:aW} of Lemma~\ref{LE:VW}.
The second estimate is a consequence of
\begin{align*}
\|\eta_A\py\Xt Uh\| &\lesssim\|\eta_A\py \Xt \py h\|
+\left\|\eta_A\Xt\left(\xi_W\py h\right)\right\|
+\left\|\eta_A\Xt\left(h\py \xi_W\right)\right\|\\
&\lesssim\theta^{-\frac12}\|\eta_A\py h\|
+\oz\|\eta_A h\|.
\end{align*}
To prove the third estimate, we write
\begin{align*}
\|\eta_A\py^2\Xt Uh\| &\lesssim\|\eta_A\py^2\Xt \py h\|
+\|\eta_A\py \Xt (\xi_W\py h)\|
+\|\eta_A\Xt(h\py^2\xi_W)\|\\
&\lesssim \theta^{-1}\|\eta_A\py h\|
+\oz\|\eta_A h\|.
\end{align*}
The last two estimates follow from the definition of $\MP$
and the previous estimates.
\end{proof}

We apply the previous estimates to the definitions of $v$ and $w$.
\begin{lemma}\label{LE:zw}
For $0<\theta<\vartheta^2$ sufficiently small and for all $s\geq 0$,
\begin{align*}
\|\eta_A\py w\| + \|\eta_A w\| &\lesssim \theta^{-2}\|\eta_A\py v\|+\|\eta_A v\|,\\
\|\eta_A\py^2 z_1\| +\|\eta_A\py z_1\| + \|\eta_A z_1\| &\lesssim\vartheta^{-1}\|\eta_A\py w_2\|
+ \|\eta_A w_2\|,\\
\|\eta_A z_2\| &\lesssim\vartheta^{-1}\|\eta_A\py w_1\| + \|\eta_A w_1\|.
\end{align*}
\end{lemma}

\begin{lemma}[{\cite[Lemma 12]{Ma22}}]\label{LE:QQ}
For small $\theta>0$ and for any $g\in H^1(\RR)$,
\[
\|\eta_A \Xt^2 Q_- g \| + \|\eta_A \Xt^2 Q_+ g\| \lesssim 
\theta^{-1} \|\eta_A \py g\| + \|\eta_A g\|.
\]
\end{lemma}

\begin{lemma}\label{LE:PP}
For small $\theta>0$ and for any $g\in H^1(\RR)$,
\[
\|\eta_A \Xt P_- g \| \lesssim 
\theta^{-\frac 12} \|\eta_A \py g\| + \|\eta_A g\|, \quad
 \|\eta_A P_+ g\| \lesssim 
 \|\eta_A g\| 
\]
\end{lemma}
\begin{proof}
These estimates are consequences of 
the definitions of $P_-$, $P_+$ and
$|\py^k \PO \xi_W|\lesssim 1$ on $\RR$,
for any $k\geq 0$
(see the proof of Lemma \ref{LE:sf}
and in particular \eqref{eq:W2}).
\end{proof}

\begin{lemma}\label{LE:10} 
Let $\tilde z =\chi_A\zeta_B z$.
For all $s\geq 0$,
\[
\|\rho \py^2 z_1\| + \|\rho\py z_1\| +\|\rho z_1\| 
 \lesssim \|\py^2 \tilde z_1\| +\|\py\tilde z_1\| +\|\rho^\frac12\tilde z_1\|
+ A^{-2}\theta^{-\frac 52} (\|\eta_A\py v\| + \|\eta_A v\|).
\]
\end{lemma}
\begin{proof}
The proof is an adaptation of \cite[Proof of Lemma 18]{Ma22}.
We start by proving a preliminary estimate
\[
\int_{|y|\leq A}\rho^2 \left((\py^2 z_1)^2 +(\py z_1)^2 + z_1^2\right)
\lesssim \int\left((\py^2 \tilde z_1)^2 +(\py \tilde z_1)^2 +\rho (\tilde z_1)^2\right).
\]
Taking $B\geq 20/\oz$ so that $\rho\lesssim \zeta_B^2$,
and using the definition of $\tilde z_1$, which implies 
$\tilde z_1 =\zeta_B z_1$ for~$|y|\leq A$, one has
\[
\int_{|y|\leq A}\rho^2 z_1^2
\lesssim \int_{|y|\leq A} \rho\,\zeta_B^2 z_1^2 
\lesssim \int \rho \tilde z_1^2 .
\]
Using $\py \tilde z_1 =\zeta_B' z_1 +\zeta_B\py z_1$
and $|\zeta_B'|\lesssim B^{-1}\zeta_B$, we also have, for $|y|\leq A$, 
\[
\rho^2(\py z_1)^2 
\lesssim \rho\zeta_B^2 (\py z_1)^2
\lesssim \rho (\py \tilde z_1)^2 + B^{-2} \rho\zeta_B^2z_1^2 
\lesssim (\py \tilde z_1)^2 + \rho \tilde z_1^2
\]
and so
\[
\int_{|y|\leq A}\rho^2 (\py z_1)^2
\lesssim \int \left((\py \tilde z_1)^2+ \rho \tilde z_1^2\right)
\]
Similarly, using 
$\py^2 \tilde z_1 =\zeta_B'' z_1 + 2 \zeta_B' \py z_1 +\zeta_B\py^2 z_1$,
for $|y|\leq A$,
\begin{align*}
\rho^2 (\py^2 z_1)^2 
\lesssim \rho\zeta_B^2 (\py^2 z_1)^2
&\lesssim \rho (\py^2 \tilde z_1)^2 +B^{-1} \rho \zeta_B^2 (\py z_1)^2
+ B^{-1} \rho\zeta_B^2z_1^2 \\
&\lesssim (\py^2 \tilde z_1)^2 + (\py \tilde z_1)^2 + \rho \tilde z_1^2,
\end{align*}
and so
\[
\int_{|y|\leq A}\rho^2 (\py^2 z_1)^2
\lesssim\int \left((\py^2 \tilde z_1)^2+(\py \tilde z_1)^2+ \rho \tilde z_1^2\right).
\]
The preliminary estimate is proved.
Taking $A$ large so that $A\oz >\sqrt{A}> 40$ we have, 
for $|y| >A$,
\[
\rho^2 \lesssim \ee^{-\frac {\oz}{5} |y|}
\lesssim \ee^{-\frac{A \oz}{10}} \ee^{-\frac {\oz}{10} |y|}
\lesssim \ee^{-\frac {\sqrt{A}}{10}}\ee^{-\frac 4{A} |y|} \lesssim A^{-4} \eta_A^2.
\]
Thus, using also the estimates on $z_1$ in Lemma~\ref{LE:zw}
and $\theta<\vartheta^2$,
\begin{align*}
\int_{|y|\geq A}\rho^2\left((\py^2 z_1)^2 +(\py z_1)^2 + z_1^2\right)
&\lesssim A^{-4}
\left(\|\eta_A\py^2 z_1\|^2 +\|\eta_A\py z_1\|^2 +\|\eta_A z_1\|^2\right)\\
&\lesssim A^{-4}\theta^{-5} \left( \|\eta_A\py v\|^2 +\|\eta_A v\|^2\right),
\end{align*}
The proof follows by combining the above estimates.
\end{proof}

\section{Coercivity of the transformed problem}\label{se:09}
In the previous section, we have given direct estimates on
$z$ and $w$ in terms of $v$. In the present section, we 
prove reverse estimates, that is estimates on $w$ and then $v$ in terms of $z$.
Such estimates are based on the orthogonality relations satisfied by the function 
$v$ in \ref{it:ov} of Lemma \ref{LE:sm}
and on related almost orthogonality relations on $w$, see \eqref{eq:ow} below.

\begin{lemma}\label{LE:wz}
For all $s\geq 0$,
\begin{align*}
\|\rho^2 \py w_2\| +\|\rho^2 w_2\|
&\lesssim\vartheta\|\rho\py^2 z_1\| +\vartheta\|\rho\py z_1\| +\oz^{-1}\|\rho z_1\|,\\
\|\rho^2 \py w_1\| + \|\rho^2 w_1\| &\lesssim\oz^{-\frac 32}\|\rho z_2\| .
\end{align*}
\end{lemma}
\begin{proof}
We first check the
approximate orthogonality relations
on $w$
\begin{equation}\label{eq:ow}
|\langle w_1, W_2\rangle|\lesssim\theta\oz\|\rho^2 w_1\|,\quad
|\langle w_2, W_1\rangle|\lesssim\theta\oz\|\rho^2 w_2\|.
\end{equation}
Indeed, using $w_1 =\Xt^2\MM\cS^2 v_2$,
Lemma~\ref{LE:VW} and then~\ref{it:ov} of Lemma \ref{LE:sm},
we have
\begin{align*}
\langle w_1-\theta\py^2 w_1 , W_2\rangle
&=\langle\MM\cS^2 v_2,W_2\rangle
=\langle v_2,(S^*)^2\MM W_2\rangle\\
&=\lambda\langle v_2,(S^*)^2 W_1\rangle
=\lambda\langle v_2,V_1\rangle=0.
\end{align*}
By~\ref{it:bW} of Lemma~\ref{LE:VW}, we have
\[
|W_2''(y)|\lesssim\oz^2 \ee^{-\alpha|y|} +\oz \ee^{-|y|}
\lesssim\oz^2\rho^8 +\oz\nu,
\]
which implies $\|\rho^{-2} W_2''\|\lesssim\oz$.
Thus, by the Cauchy-Schwarz inequality, we have
\[
|\langle w_1 , W_2\rangle|= \theta |\langle w_1 , W_2''\rangle|
\lesssim \theta\oz\|\rho^2 w_1\| .
\]
A similar argument for $\langle w_2, W_1\rangle $ completes the 
proof of~\eqref{eq:ow}.

By definition of the functions $z_1$ and $z_2$, we have
\begin{align*}
z_1 -\vartheta\py^2 z_1 & = 
U w_2 = W_2\py\left(\frac {w_2}{W_2}\right),\\
z_2 -\vartheta\py^2 z_2 &
=-U\MP w_1 = - W_2\py\left(\frac{\MP w_1}{W_2}\right). 
\end{align*}
For the pair $(w_2,z_1)$, we write the above relation
\[
\py\left(\frac{w_2}{W_2}\right)
=
-\vartheta\py\left(\frac{\py z_1}{W_2}+\frac{W_2'z_1}{W_2^2}\right)
+\frac{m_2}{W_2} z_1
\]
where we have defined
\[
m_2 = 1 +\vartheta\left(\frac{W_2''W_2-2(W_2')^2}{W_2^2}\right).
\]
Integrating on $[0,y]$ and multiplying by $W_2$, we find
\begin{equation}\label{eq:2s}
w_2 = a W_2 -\vartheta\py z_1 -\vartheta\frac{W_2'}{W_2} z_1
+ W_2\int_0^y\frac{m_2}{W_2} z_1 .
\end{equation}
Here, $a$ is an integration constant, 
which we estimate now by projecting the above identity on $W_1$.
By~\ref{it:bW} and~\ref{it:aW} of Lemma~\ref{LE:VW}, and
the Cauchy-Schwarz inequality, we have
\[
|W_1'|+\Bigl|\frac{W_2'}{W_2} W_1\Bigr|\lesssim\oz \ee^{-\alpha|y|},\quad
\Bigl|\frac{m_2}{W_2}\Bigr|\lesssim \ee^{\alpha|y|},\quad
|\langle z_1,W_1'\rangle| +
\Bigl|\Bigl\langle z_1,\frac{W_2'W_1}{W_2}\Bigr\rangle\Bigr|
\lesssim\sqrt{\oz}\|\rho z_1\|,
\]
and
\[
\Bigl|\int_0^y\frac{m_2}{W_2} z_1\Bigr|
\lesssim\frac 1{\sqrt{\oz}}\rho^{-1}\ee^{\alpha|y|}\|\rho z_1\|,
\quad
\Big|\Big\langle W_2\int_0^y z_1\frac{m_2}{W_2} ,W_1\Big\rangle\Big|
\lesssim\frac 1{\oz\sqrt{\oz}}\|\rho z_1\|.
\]
Using
$\langle W_1, W_2\rangle =\alpha^{-1}(1+O(\omega))$
(see~\ref{it:aW} of Lemma~\ref{LE:VW}) and then~\eqref{eq:ow},
we obtain by projecting~\eqref{eq:2s} on $W_1$
\[
|a|\lesssim\oz \bigl(|\langle w_2,W_1\rangle|
+\oz^{-\frac 32}\|\rho z_1\|\big ) 
\lesssim\theta\oz^2\|\rho^2 w_2\|
+\oz^{-\frac12}\|\rho z_1\|.
\]
Then, multiplying~\eqref{eq:2s} by $\rho^2$, taking the $L^2$ norm and using the triangle inequality, we find
\[
\|\rho^2 w_2\|\lesssim\theta\oz^\frac32\|\rho^2 w_2\|
+\vartheta\|\rho^2\py z_1\|
+\oz^{-1}\|\rho z_1\|
\]
which implies, for $\theta$ small enough,
\[
\|\rho^2 w_2\|\lesssim\vartheta\|\rho^2\py z_1\|+\oz^{-1}\|\rho z_1\|.
\]
Now, differentiating \eqref{eq:2s},
\[
\py w_2 = a W_2' -\vartheta\py^2 z_1 -\vartheta\Bigl(\frac{W_2'}{W_2}\Bigr)' z_1-\vartheta \frac{W_2'}{W_2} \py z_1
+ W_2'\int_0^y\frac{m_2}{W_2} z_1 + m_2 z_1 ,
\]
and so, using similar estimates
\[
\|\rho^2 \py w_2\|\lesssim \vartheta\|\rho^2\py^2 z_1\|+
\vartheta\|\rho^2\py z_1\|+ \|\rho z_1\|.
\]
For the pair $(w_1,z_2)$, we proceed similarly.
We have
\begin{equation}\label{eq:Fr}
\MP w_1 = b W_2 +\vartheta\py z_2 +\vartheta\frac{W_2'}{W_2} z_2
- W_2\int_0^y\frac{m_2}{W_2}z_2 .
\end{equation}
We estimate the integration constant $b$ by projecting the above identity on $W_1$. By $\MP W_1 =\lambda W_2$ and~\eqref{eq:ow}, we have
\[
|\langle\MP w_1, W_1\rangle| 
=|\langle w_1,\MP W_1\rangle|
=\lambda|\langle w_1,W_2\rangle|
\lesssim\theta\oz\|\rho^2w_1\|.
\]
Thus, proceeding as for the estimate of $|a|$ before, we obtain 
\[
|b|\lesssim\theta\oz^2\|\rho^2 w_1\|+\oz^{-\frac12}\|\rho z_2\|.
\]
Now, we follow~\cite[proof of Lemma 21]{Ma22}.
Let $H_1$ and $H_2$ be solutions of the equation $\MP H=0$ satisfying
$H_1'H_2-H_1H_2'=1$ and, for all $k\geq 0$, on $\RR$,
\[
|H_1^{(k)}(y)|\lesssim \ee^{-y},\quad|H_2^{(k)}(y)|\lesssim \ee^{y}.
\]
(Such independent solutions exist since the equation $\MP h=0$
has no solution in $L^2$.)
The interest of introducing $H_1$ and $H_2$ lies on the formula
inverting~$\MP$.
\begin{equation}\label{eq:We}
w_1(y) = H_1(y)\int_{-\infty}^y H_2\MP w_1
+H_2(y)\int_y^{+\infty} H_1\MP w_1 .
\end{equation}
Now, to estimate $\|\rho^2 w_1\|$, 
we insert~\eqref{eq:Fr} into the above formula.
To avoid having derivatives of $z_2$ in the estimate for $w_1$, we
write by integration by parts
\[
 H_1(y)\int_{-\infty}^y H_2\py z_2
+H_2(y)\int_y^{+\infty} H_1\py z_2 
= - H_1(y)\int_{-\infty}^y H_2' z_2 - H_2\int_y^{+\infty} H_1' z_2.
\]
To handle the various terms in the expression of $w_1$
above, we note that for any $h$,
\[
\Big|\rho^2 H_1\int_{-\infty}^y H_2 h\Big|
\lesssim\rho^\frac12\|\rho^\frac32 h\| ,\quad 
\Big\|\rho^2 H_1\int_{-\infty}^y H_2 h\big\|
\lesssim\oz^{-\frac 12}\|\rho^\frac32 h\|.
\]
Using this observation and similar other estimates, we obtain
\[
\sqrt{\oz}\|\rho^2 w_1\|
\lesssim|b| \|\rho^\frac32 W_2\|
+\vartheta\|\rho^\frac 32 z_2\|
+ \left\|\rho^\frac32 W_2\int_0^y \frac{m_2}{W_2}z_2\right\|
\lesssim\theta \oz\|\rho^2 w_1\|+\oz^{-1}\|\rho z_2\|,
\]
which implies the estimate $\|\rho^2 w_1\|\lesssim \oz^{-3/2}\|\rho z_2\|$, for $\theta$ small enough.
Differentiating \eqref{eq:We}, we find
\[
\py w_1(y) = H_1'(y)\int_{-\infty}^y H_2\MP w_1
+H_2'(y)\int_y^{+\infty} H_1\MP w_1 
\]
and using similar estimates, we find
$\|\rho^2 \py w_1\|
\lesssim \oz^{-3/2}\|\rho z_2\|$.
\end{proof}

\begin{lemma}\label{LE:cy}
For all $s\geq 0$,
\begin{align*}
\|\rho^4 v_1\|&\lesssim\|\rho^2 w_2\|
\lesssim\vartheta\|\rho\py^2 z_1\| +\vartheta\|\rho\py z_1\| +\oz^{-1}\|\rho z_1\|,\\
\|\rho^4 v_2\|&\lesssim\|\rho^2 w_1\|
\lesssim\oz^{-\frac 32}\|\rho z_2\|.
\end{align*}
\end{lemma}
\begin{proof}
Recall that $w_1 =\Xt^2\MM\cS^2 v_2$ and $w_2 = -\Xt^2\cS^2\LP v_1$.
Thus, adapting the proof of Proposition 19 in \cite{Ma22}
(which is close to the one of Lemma \ref{LE:wz} of the present paper), the estimates 
$\|\rho^4 v_1\|\lesssim\|\rho^2 w_2\|$ and 
$\|\rho^4 v_2\|\lesssim\|\rho^2 w_1\|$
are consequences of the orthogonality relations~\ref{it:ov}
of Lemma \ref{LE:sm}.
In particular, we note that the sign of the quintic term has no impact on the result.
We complete the proof by using Lemma \ref{LE:wz}.
\end{proof}

\section{Estimate on the transformed problem}\label{se:10}
The last lemma provides the main estimate of this article,
based on a virial argument applied to the transformed problem \eqref{eq:zz}, and thus
relying on the repulsive nature of potential of the operator $\cK$
studied in Lemmas \ref{LE:v2}, \ref{LE:IK} and \ref{LE:nZ}.

\begin{lemma}\label{LE:vz}
For any $s>0$,
\[
 \int_0^s\left(\|\rho\py^2 z_1\|^2 +\|\rho\py z_1\|^2 + \|\rho z_1\|^2
+ \|\rho z_2\|^2\right)
\lesssim \sqrt{\varepsilon} +\frac {1}{\sqrt{A}}\int_0^s\|\rho^4 v\|^2 .
\]
\end{lemma}
\begin{remark}
As in the proof of Lemmas \ref{LE:v1} and \ref{LE:eb}, 
several parameters have to be adjusted in the proof of Lemma \ref{LE:vz}.
Recall that $\oz>0$ is a parameter to be taken sufficiently small, since
several key arguments are valid only for small solitons (starting by the construction of the internal mode in Lemma \ref{LE:VW}).
Then, the scale $B$ of the virial argument on the transformed problem is to be chosen sufficiently large, depending on $\oz$.
The parameter $\theta>0$ involved in the regularizing operator $\Xt$ is to be chosen sufficiently small, depending both on $B$
and on $\oz$ (the auxiliary parameter $\vartheta$ is to defined by $\vartheta=\theta^\frac14$ in the proof below, see the estimate of $\bK_2$). 
The parameter $A$, scale of the first virial argument,
is also to be taken large, depending on $\theta$, $B$ and $\oz$.
Finally, the parameter $\varepsilon>0$, controlling the size of the perturbation around the soliton is to be chosen small, depending on $A$, $\theta$, $B$ and $\oz$.
It would be possible to track explicitly how the required smallness of $\varepsilon$ depends
on $\oz$, but we do not pursue this issue here.
\end{remark}
\begin{proof}
We define
\[
\bK= -\int(\Xi_{A,B} z_1) z_2, \qquad 
\bL = \int \rho^2 z_1 z_2.
\]
Note that $\bK$ and $\bL$ are well-defined since for all $s\geq 0$, $z_1(s)\in H^2$ and $z_2(s)\in L^2$. Moreover,
by the properties of $\chi_A$,
$|\Phi_{A,B}|\lesssim B$ and $|\Phi_{A,B}'|\lesssim 1$,
\[
|\bK|\lesssim\|\Xi_{A,B} z_1\|\|\eta_A z_2\|\lesssim
B(\|\eta_A\py z_1\|+\|\eta_A z_1\|)\|\eta_A z_2\|
\]
and so, using Lemma~\ref{LE:zw},~\ref{it:sv} of Lemma \ref{LE:sm},
and taking $\varepsilon$ small enough, depending on $B$, $\theta$ and $\vartheta$,
\begin{equation}\label{eq:eK}
|\bK|\lesssim B\vartheta^{- 2}\theta^{-4}\|v\|_{H^1}^2
\lesssim B\vartheta^{- 2}\theta^{-4}\varepsilon^2\lesssim \varepsilon,
\end{equation}
We also check that
\begin{equation}\label{eq:eL}
|\bL|\lesssim \|\rho z_1\| \|\rho z_2\| \lesssim \vartheta^{-2} \theta^{-4} \varepsilon^2
\lesssim \varepsilon.
\end{equation}
By the equation of $z$ in~\eqref{eq:zz}, we compute
\[
\dot\bK =\bK_1 +\bK_2+\bK_3+\bK_4+\bK_5,
\]
where
\begin{align*}
\bK_1 &=\int(\Xi_{A,B} z_1)\cK z_1,\quad
\bK_2 =\int(\Xi_{A,B} z_1) [\Xu,\cK]\cU w_2,\\
\bK_3 &=\frac 13\int(\Xi_{A,B} z_2)\Xu\cU [\Xt^2,\Qo^4]\MM\cS^2 v_2
-\int(\Xi_{A,B} z_1)\Xu\cU\MP [\Xt^2,\Qo^4]\cS^2\LP v_1,\displaybreak[0]\\
\bK_4 &= -\int(\Xi_{A,B} z_2)\Xu\cU\Xt^2 n_1
 +\int(\Xi_{A,B} z_1)\Xu\cU\MP\Xt^2n_2,\\
\bK_5 &= \dot \omega\int(\Xi_{A,B} z_2)\Xu\PP w_2
 + \dot \omega\int(\Xi_{A,B} z_1)\Xu\PM w_1.
\end{align*}
Moreover, 
\[
\dot \bL=\bL_1+\bL_2+\bL_3+\bL_4+\bL_5,
\]
where
\begin{align*}
 \bL_1& = \int \rho^2 (z_2^2 - z_1 \cK z_1), \quad
 \bL_2= - \int \rho^2 z_1 [\Xu,\cK]\cU w_2 ,\\
\bL_3 &= \frac 13\int \rho^2 z_2 \Xu\cU [\Xt^2,\Qo^4]\MM\cS^2 v_2 
+ \int \rho^2 z_1 \Xu\cU\MP [\Xt^2,\Qo^4]\cS^2\LP v_1,\displaybreak[0]\\
\bL_4& =
- \int \rho^2 z_2 \Xu\cU\Xt^2 n_1 
- \int \rho^2 z_1 \Xu\cU\MP\Xt^2 n_2,\\
\bL_5 & = \dot\omega\int \rho^2 z_2\Xu\PP w_2 
 - \dot \omega\int \rho^2 z_1 \Xu\PM w_1.
\end{align*}
By the definition of the operator $\cK$ in Lemma~\ref{LE:sf} and integrations by parts, we expand
\[
\bK _1 =\bK_{1,1}+\bK_{1,2}+\bK_{1,3}+\bK_{1,4}
\]
where
\begin{align*}
\bK_{1,1}&= 4\int\Psi_{A,B}'(\py^2 z_1)^2 - 3\int\Psi_{A,B}'''(\py z_1)^2 
+\frac 12\int\Psi_{A,B}^{(5)} z_1^2\\
\bK_{1,2} &= 4\int\Psi_{A,B}'(\py z_1)^2 
-\int\left( 2\Psi_{A,B}' K_2 +\Psi_{A,B} K_2'\right)(\py z_1)^2\displaybreak[0]\\
\bK_{1,3}&= -\int\Psi_{A,B}''' z_1^2 +\frac 12\int\left(\Psi_{A,B}''' K_2 + 2\Psi_{A,B}'' K_2'+\Psi_{A,B}' K_2''\right) z_1^2\\
\bK_{1,4}&= 2\int\Psi_{A,B} K_1(\py z_1)^2
-\frac 12\int\left(\Psi_{A,B}'' K_1 +\Psi_{A,B}' K_1'\right) z_1^2 
-\int\Psi_{A,B} K_0' z_1^2,
\end{align*}
Recall that $\tilde z =\chi_A\zeta_B z$ and note that
$\py^2\tilde z_1 =\chi_A\zeta_B\py^2 z_1 + 2(\chi_A\zeta_B)'\py z_1 
+(\chi_A\zeta_B)'' z_1$, which
implies by integration by parts,
\begin{align*}
\int(\py^2\tilde z_1)^2 & =\int\chi_A^2\zeta_B^2(\py^2 z_1)^2 -2\int\Bigl(2(\chi_A\zeta_B)''\chi_A\zeta_B -\left((\chi_A\zeta_B)'\right)^2\Bigr)(\py z_1)^2\\&\quad +\int(\chi_A\zeta_B)''''\chi_A\zeta_B z_1^2.
\end{align*}
Thus, for the first term of $\bK_{1,1}$, we have
\[
4\int\Psi_{A,B}'(\py^2 z_1)^2 
= 4\int(\py^2\tilde z_1)^2 + 8\int\chi_A^2\left(2\zeta_B''\zeta_B -(\zeta_B')^2\right)(\py z_1)^2 - 4\int\chi_A^2\zeta_B''''\zeta_B z_1^2 +\bR_1,
\]
where $\bR_1$ contains all the terms
where the function $\chi_A$ has been differentiated
(considered as error terms in this computation)
\begin{align*}
\bR_1 & = 4\int(\chi_A^2)'\Phi_B(\py^2 z_1)^2
- 4\int((\chi_A\zeta_B)''''-\chi_A\zeta_B'''')\chi_A\zeta_B z_1^2\\
&\quad + 8\int\left(2\left((\chi_A\zeta_B)''-\chi_A\zeta_B''\right)\chi_A\zeta_B -\left(((\chi_A\zeta_B)')^2
 -\chi_A^2(\zeta_B')^2\right)\right)(\py z_1)^2.
\end{align*}
For the second term of $\bK_{1,1}$, we compute
\[
- 3\int\Psi_{A,B}'''(\py z_1)^2 = 
-6\int\chi_A^2(\zeta_B''\zeta_B +(\zeta_B')^2)(\py z_1)^2 +\bR_2
\]
where
\[
\bR_2 = -3\int\left(3(\chi_A^2)'(\zeta_B^2) ' +3(\chi_A^2)''\zeta_B^2 
+(\chi_A^2)'''\Phi_B\right)(\py z_1)^2.
\]
Setting
\[
\bR_3 = \frac 12\int\bigl(\Psi_{A,B}^{(5)}-\chi_A^2(\zeta_B^2)^{(4)}\bigr) z_1^2,
\]
we obtain
\begin{align*}
\bK_{1,1} &= 4\int(\py^2\tilde z_1)^2 
+\int\chi_A^2\left(10\zeta_B''\zeta_B - 14(\zeta_B')^2\right)(\py z_1)^2\\
&\quad +\int\chi_A^2\left(-3\zeta_B''''\zeta_B+ 4\zeta_B'''\zeta_B'+3(\zeta_B'')^2\right) z_1^2
+\bR_1+\bR_2+\bR_3.
\end{align*}
We continue with the next terms in the decomposition of $\bK_1$. We have
\[
\bK_{1,2} = 
4\int\chi_A^2\zeta_B^2(\py z_1)^2
-\int(2\chi_A^2\zeta_B^2 K_2 +\chi_A^2\Phi_B K_2')(\py z_1)^2
+\bR_4,
\]
where
\[
\bR_4 = 4\int(\chi_A^2)'\Phi_B(\py z_1)^2 
- 2\int(\chi_A^2)'\Phi_B K_2(\py z_1)^2,
\]
and
\begin{align*}
\bK_{1,3} & = 
-\int\chi_A^2(\zeta_B^2)'' z_1^2
+\frac 12\int\chi_A^2\left((\zeta_B^2)'' K_2 + 2(\zeta_B^2)' K_2'+\zeta_B^2 K_2''\right) z_1^2
+\bR_5,
\end{align*}
where
\begin{align*}
\bR_5 & = 
-\int(\Psi_{A,B}'''-\chi_A^2(\zeta_B^2)'') z_1^2
+\frac 12\int\left(\Psi_{A,B}'''-\chi_A^2(\zeta_B^2)''\right) K_2 z_1^2\\
&\quad +\int\left(2(\chi_A^2)'\zeta_B^2+(\chi_A^2)''\Phi_B\right) K_2'z_1^2
+\frac 12\int(\chi_A^2)'\Phi_B K_2'' z_1^2.
\end{align*}
Lastly,
\[
\bK_{1,4} = 2\int\chi_A^2\Phi_B K_1(\py z_1)^2
-\frac 12\int\chi_A^2\left((\zeta_B^2)' K_1+\zeta_B^2 K_1'\right) z_1^2
 -\int\chi_A^2\Phi_B K_0' z_1^2+\bR_6
\]
where
\[
\bR_6 = 
-\frac12\int\left(2(\chi_A^2)'\zeta_B^2+(\chi_A^2)''\Phi_B\right)K_1 z_1^2
-\frac12\int(\chi_A^2)'\Phi_B K_1'z_1^2.
\]
Summing up, we obtain
\begin{align*}
\bK_1 & = 4\int(\py^2\tilde z_1)^2
+ 4\int\chi_A^2\zeta_B^2(\py z_1)^2+\int \chi_A^2 \zeta_B^2 \xi_B(\py z_1)^2\\
&\quad +\int\chi_A^2\left(-3\zeta_B''''\zeta_B+ 4\zeta_B'''\zeta_B'+3(\zeta_B'')^2-(\zeta_B^2)''\right) z_1^2\\
&\quad +\frac 12\int\chi_A^2\left((\zeta_B^2)'' K_2 +(\zeta_B^2)'(2K_2'-K_1)+\zeta_B^2(K_2''
- K_1')-2\Phi_B K_0'\right) z_1^2+\sum_{j=1}^6 \bR_j
\end{align*}
where
\[
\xi_B = 10\frac{\zeta_B''}{\zeta_B} - 14\frac{(\zeta_B')^2}{\zeta_B^2}
 - 2 K_2 -\frac{\Phi_B}{\zeta_B^2} K_2'+ 2\frac{\Phi_B}{\zeta_B^2} K_1.
\]
Using $\py\tilde z_1 =\chi_A\zeta_B\py z_1 +(\chi_A\zeta_B)' z_1$,
we have by integration by parts,
\[
\int\chi_A^2\zeta_B^2(\py z_1)^2 = \int(\py\tilde z_1)^2 
+\int\chi_A\zeta_B(\chi_A\zeta_B)'' z_1^2.
\]
and
\[
\int \chi_A^2 \zeta_B^2 \xi_B(\py z_1)^2=
 \int \xi_B (\py \tilde z_1)^2 
 + \int \chi_A\zeta_B \bigl((\chi_A\zeta_B)'\xi_B\bigr)' z_1^2.
\]
Therefore, we rewrite the above expression for $\bK_1$ as
\[
\bK_1 =\bP +\sum_{j=1}^{9}\bR_j\quad 
\mbox{where} \quad
\bP = \int \left(4 (\py^2\tilde z_1)^2 + (4 + \xi_B) (\py\tilde z_1)^2
 +\pp\tilde z_1^2\right),
\]
the function $Y_0$ being defined in Lemma \ref{LE:v2}, and 
\begin{align*}
\bR_7& =4\int\chi_A\zeta_B(\chi_A''\zeta_B+2\chi_A'\zeta_B') z_1^2 + \int \chi_A\zeta_B (\chi_A''\zeta_B\xi_B+2\chi_A'\zeta_B'\xi_B+\chi_A'\zeta_B \xi_B') z_1^2,\displaybreak[0]\\
\bR_8 &=\int \chi_A^2\bigl( y {\zeta_B^2} - {\Phi_B}\bigr) K_0' z_1^2 +\frac 12\int\chi_A^2\left((\zeta_B^2)'' K_2 +(\zeta_B^2)'(2K_2'-K_1)\right)z_1^2,\displaybreak[0]\\
\bR_{9} &= 
 \int\chi_A^2\left(2 \zeta_B''\zeta_B-2(\zeta_B')^2
-3\zeta_B''''\zeta_B+ 4\zeta_B'''\zeta_B'+3(\zeta_B'')^2
+\zeta_B\zeta_B'' \xi_B+\zeta_B \zeta_B' \xi_B'\right) z_1^2.
\end{align*}

\emph{Lower bound on $\bP$.}
Taking $B$ sufficiently large,
$|\xi_B|\lesssim B^{-1} + \oz \ee^{-|y|/2}\lesssim \oz$,
and so
\[
\Bigl|\int \xi_B(\py \tilde z_1)^2 \Bigr|\lesssim \oz \|\py \tilde z_1\|^2.
\]
By Lemma~\ref{LE:v2}, we have
$|\pp|\leq C\oz \ee^{-|y|}$ for some $C>0$.
Moreover, by Lemma~\ref{LE:IK}, for $\oz$ small, we have
$\int \pp\gtrsim\oz$.
Applying Lemma~\ref{LE:BS} with $c=1$ and $Y=Y_0/C\oz$, for any $h\in H^1$, we have
\[
\oz \int \ee^{-|y|} h^2 \leq C_1 \int Y_0 h^2 + C_2 \oz \int (h')^2
\leq C_1 \Bigl( \int Y_0 h^2 + \int (h')^2\Bigr),
\]
for some constants $C_1,C_2>0$.
Using again Lemma \ref{LE:BS} with $c=\oz/10$ and $Y=e^{-|y|}$, we have
\[
\oz^2 \int \rho h^2 \leq \frac{C_3}{C_1} \oz \int \ee^{-|y|}h^2 + C_3\int (h')^2
\leq C_3 \Bigl( \int Y_0 h^2 + 2\int (h')^2\Bigr),
\]
for some constant $C_3>0$.
(Note that the above estimate holds for any function $h$ in $H^1$.
In the context of the present paper, one can also use the fact that 
the pair of functions $(z_1,z_2)$ is odd
and \cite[Claim 4.1]{KMM1}; see also the remark after Lemma \ref{LE:un}.)
Thus, 
\[
\bP 
\gtrsim\|\py^2\tilde z_1\|^2 + \|\py\tilde z_1\|^2 + \oz^2\|\rho^\frac 12\tilde z_1\|^2.
\]
Using now Lemma \ref{LE:10}, we have proved
\begin{equation}\label{eq:P1}
\oz^2 \big(\|\rho \py^2 z_1\|^2 + \|\rho\py z_1\|^2 +\|\rho z_1\|^2\big)
\lesssim 
\bP + A^{-4}\theta^{-5} (\|\eta_A\py v\|^2 + \|\eta_A v\|^2) .
\end{equation}

\emph{Estimates of $\bR_1,\ldots,\bR_7$.}
Note that all the terms in the expression of $\bR_1,\ldots,\bR_7$,
contain derivatives of the function $\chi_A$.
On the one hand, for all $k\geq 1$,
\[
\mbox{$|\chi_A^{(k)}(y)|\lesssim A^{-k}$ if $A<|y|<2A$ 
and $\chi_A^{(k)}=0 $ otherwise.}
\]
On the other hand, $|\Phi_B|\lesssim B$ and
for all $l\geq 1$, on $\RR$,
\[
|\zeta_B | 
+B |\zeta_B^{(l)} |\lesssim \ee^{-\frac{|y|}B}.
\]
As a consequence, for all $k\geq 1$ and all $l\geq 0$, on $\RR$,
\[
|\chi_A^{(k)}\Phi_B|\lesssim B A^{-k} \eta_A^2,\quad 
|\chi_A^{(k)} \zeta_B^{(l)}|\lesssim \ee^{-\frac AB} \eta_A^2
\lesssim \frac{B^3}{A^3} \eta_A^2,\quad
|\chi_A^{(k)} \ee^{-|y|}|\lesssim \ee^{-A} \eta_A^2.
\]
Therefore, examining all terms in $\bR_1,\ldots,\bR_7$, we check
that
\[
\sum_{j=1}^7 |\bR_j|
\lesssim \frac BA \Bigl( \|\eta_A\py^2 z_1\|^2
+ \|\eta_A\py z_1\|^2 + \frac {B^2}{A^2} \|\eta_A z_1\|^2\Bigr).
\]
Using Lemma \ref{LE:zw}, we obtain
\[
\sum_{j=1}^7 |\bR_j|
\lesssim
\frac B{A\theta^5} \left( \|\eta_A \py v\|^2+ \frac {B^2}{A^2}\|\eta_A v\|^2\right). 
\]

\emph{Estimate of $\bR_8$.}
For the first term in $\bR_8$, for $y\geq 0$, since
$0\leq \Phi_B \leq y$ and $\zeta_B\leq \ee^{-y/B}$, we check that
\[
0 \leq {\Phi_B} - y {\zeta_B^2}
\leq y \bigl( 1-\zeta_B^2 \bigr)
\leq y \bigl( 1-\ee^{\frac {-2 y}{B}}\bigr)
\leq \frac {2 }{B}y^2 
\]
Thus, using also $|K_0'|\lesssim \nu^{10}$ from \eqref{eq:Kj}, 
we obtain
\[
\Bigl|\int ( y {\zeta_B^2} - {\Phi_B} ) K_0' z_1^2\Bigr|
\lesssim \int | {\Phi_B}-y {\zeta_B^2} | \nu^{10} z_1^2
\lesssim \frac 1B \int y^2 \nu^{10} z_1^2
\lesssim \frac 1B \|\nu z_1\|^2.
\]
From \eqref{eq:Kj}, we also have
\[
\int\chi_A^2\left|(\zeta_B^2)'' K_2 +(\zeta_B^2)'(2K_2'-K_1)\right|z_1^2
\lesssim \frac 1B \|\nu z_1\|^2.
\]
In conclusion for this term,
\[
|\bR_8| \lesssim \frac 1B \|\nu z_1\|^2.
\]

\emph{Estimate of $\bR_9$.}
We write 
$\bR_9=\int (\iota_B+\iota_K) \tilde z_1^2$
where
\begin{align*}
\iota_B &=
2 \left(\ln \zeta_B\right)''
-3\frac{\zeta_B''''}{\zeta_B}+ 14\frac{\zeta_B'''\zeta_B'}{\zeta_B^2}
+13\frac{(\zeta_B'')^2}{\zeta_B^2}
-52\frac{(\zeta_B')^2\zeta_B''}{\zeta_B^3} 
+28 \frac{(\zeta_B')^4}{\zeta_B^4},\\
\iota_K &= \zeta_B^{-1}
\left(\zeta_B' \left( - 2 K_2 - ({\Phi_B}/{\zeta_B^2}) K_2'
+ 2 ({\Phi_B}/{\zeta_B^2}) K_1
\right)\right)'.
\end{align*}
We estimate $\iota_B$ and $\iota_K$.
On the one hand, using the cancellation $-3+14+13-52+28=0$
and the fact that the function $\chi$ is supported on $[-2,2]$,
we see that $\iota_B=0$ for $|y|>2$. Since 
$|\iota_B|\lesssim 1/B$ for $|y|<2$, we obtain
$|\iota_B| \lesssim \nu^2/B$.
On the other hand, by the estimates~\eqref{eq:Kj} of $K_2$ and $K_1$,
we have $|\iota_K| \lesssim \oz \nu^2/B$.
Thus,
\[
|\bR_9|\lesssim \frac 1{B} \|\nu \tilde z_1\|^2\lesssim \frac 1{B} \|\nu z_1\|^2.
\]
Taking $B$ large enough (depending $\oz$), using \eqref{eq:P1} and the 
above estimates for $\bR_j$,
\[
C_1 \oz^2 \big(\|\rho \py^2 z_1\|^2 + \|\rho\py z_1\|^2 +\|\rho z_1\|^2\big)
\leq 
\bK_1
+ \frac B{A\theta^5} \Bigl( \|\eta_A \py v\|^2+ \frac {B^2}{A^2}\|\eta_A v\|^2\Bigr),
\]
for a constant $C_1>0$.

\emph{Estimate of $\bL_1$.}
By the definition of the operator $\cK$ and the properties of the 
functions $K_2$, $K_1$ and $K_0$ in Lemma \ref{LE:sf}, it holds
for a constant $C_2>0$,
\[
 \bL_1 \geq \|\rho z_2\|^2 
-C_2 \big( \|\rho \py^2 z_1\|^2 + \|\rho \py z_1\|^2 +\|\rho z_1\|^2\big).
\]
Setting $C=C_1/2C_2$, it follows that 
\begin{equation}\label{eq:P2}
\oz^2 \bZ \lesssim 
\bK_1+C \oz^2 \bL_1
+ \frac B{A\theta^5} \Bigl( \|\eta_A \py v\|^2+ \frac {B^2}{A^2}\|\eta_A v\|^2\Bigr),
\end{equation}
where we have set
$\bZ = \|\rho \py^2 z_1\|^2 + \|\rho\py z_1\|^2 +\|\rho z_1\|^2 +\|\rho z_2\|^2$.

\emph{Estimates of $\bK_2$ and $\bL_2$.}
By the Cauchy-Schwarz inequality, we have
\[
|\bK_2|\lesssim\|\rho\Xi_{A,B} z_1\|
\|\rho^{-1} [\Xu,\cK]\cU w_2\|.
\]
By the estimates
$|\Psi_{A,B}|\lesssim B$ and $|\Psi'_{A,B}|\lesssim 1$, we have
$\|\rho\Xi_{A,B} z_1\|
\lesssim B\|\rho\py z_1\|+\|\rho z_1\|$.
Observe that
\[
[\Xu,\cK]\cU w_2
=\Xu [\cK,\Xu^{-1}]\Xu\cU w_2 =\Xu [\cK,\Xu^{-1}] z_1.
\]
Moreover, by the expression of the operator $\cK$ in Lemma \ref{LE:sf}
\begin{align*}
[\cK,\Xu^{-1}] z_1 
& = [K_2,\Xu^{-1}]\py^2z_1 + [K_1,\Xu^{-1}]\py z_1 + [K_0,\Xu^{-1}]z_1\\
& =\vartheta\left(2\py(K_2'\py^2 z_1) +(-K_2''+2K_1')\py^2z_1
+(K_1'' + 2 K_0')\py z_1 + K_0'' z_1\right).
\end{align*}
Thus, using Lemma~\ref{LE:tc} to estimate the first term on the right hand side
and then \eqref{eq:Kj}, one has
\begin{equation}\label{eq:k2}
\|\rho^{-1} [\Xu,\cK]\cU w_2\|
\lesssim\oz \vartheta^{\frac 12}
\left(\|\rho\py^2 z_1\| +\|\rho\py z_1\| +\|\rho z_1\|\right).
\end{equation}
Choosing $\vartheta=\theta^\frac 14$ and using $\oz \lesssim 1$, one obtains
\[
|\bK_2|\lesssim B \theta^{\frac 18}\bZ.
\]
Similarly, using the Cauchy-Schwarz inequality and \eqref{eq:k2}, we have
\[
|\bL_2|\lesssim\|\rho z_1\|
\|\rho [\Xu,\cK]\cU w_2\|
\lesssim\theta^{\frac 18}\bZ.
\]

\emph{Estimates of $\bK_3$ and $\bL_3$.}
Using Lemma~\ref{LE:tc}, the relation 
\[\rho\Xi_{A,B} z_2 =\py( 2\rho\Psi_{A,B} z_2)
-2\rho '\Psi_{A,B} z_2 -\rho\Psi_{A,B}' z_2,\]
then again Lemma~\ref{LE:tc} and the estimates
$|\Psi_{A,B}|\lesssim B$ and $|\Psi'_{A,B}|\lesssim 1$, we get
\begin{align*}
\|\rho\Xu\Xi_{A,B} z_2\|
&\lesssim 
\|\Xu(\rho\Xi_{A,B} z_2)\|\\
&\lesssim\|\Xu\py(\rho\Psi_{A,B} z_2)\|
+\|\Xu(\rho'\Psi_{A,B} z_2)\|+\|\Xu(\rho\Psi_{A,B}' z_2)\|\\
&\lesssim 
\vartheta^{-\frac 12}\|\rho\Psi_{A,B} z_2\|
+\|\rho'\Psi_{A,B} z_2\|+\|\rho\Psi_{A,B}' z_2\|
\lesssim B\vartheta^{-\frac 12}\|\rho z_2\|.
\end{align*}
Then, using $[\Xt^2,\Qo^4]\MM\cS^2 v_2=\Xt^2 [\Qo^4,\Xt^{-2}] w_1 $, we note that
\begin{align*}
&[\Xt^2,\Qo^4]\MM\cS^2 v_2
 =2\theta\Xt^2\left( 2(\Qo^4)'\py w_1 +(\Qo^4)''w_1\right)\\
&\quad-\theta^2\Xt^2\left( 4\py^2((\Qo^4)'\py w_1)
- 2\py^2((\Qo^4)''w_1) + 4\py((\Qo^4)'''w_1)
-(\Qo^4)''''w_1\right).
\end{align*}
Thus, using $U=\py -\xi_W$, the estimate $|\xi_W|\lesssim\oz$
and Lemma~\ref{LE:tc},
\begin{align}
\|\rho^{-1}\cU [\Xt^2,\Qo^4]\MM\cS^2 v_2\|
&\lesssim\|\rho^{-1}\py\Xt^2 [\Qo^4,\Xt^{-2}] w_1\|
+\oz\|\rho^{-1}\Xt^2 [\Qo^4,\Xt^{-2}] w_1\|\nonumber \\
&\lesssim\theta^{\frac 12}(\|\rho^2\py w_1\|+\|\rho^2 w_1\|).\label{eq:k3}
\end{align}
In view of the above estimates, we estimate the first term in $\bK_3$ 
by using the Cauchy-Schwarz inequality
\begin{align*}
\Bigl|\int(\Xi_{A,B} z_2)(\Xu\cU [\Xt^2,\Qo^4]\MM\cS^2 v_2)\Bigr|
&\lesssim\|\rho\Xu\Xi_{A,B} z_2\|
\|\rho^{-1}\cU [\Xt^2,\Qo^4]\MM\cS^2 v_2\|\\
&\lesssim B\vartheta^{-\frac 12}\theta^{\frac 12}\|\rho z_2\|(\|\rho^2\py w_1\|+\|\rho^2 w_1\|).
\end{align*}
For the second term in $\bK_3$, we see that
$\|\rho\Xi_{A,B} z_1\|\lesssim B\|\rho\py z_1\|+\|\rho z_1\|$.
Moreover,
\begin{align*}
& [\Xt^2,\Qo^4]\cS^2\LP v_1
= -2\theta\Xt^2\left( 2(\Qo^4)'\py w_2 +(\Qo^4)''w_2\right)\\
&\quad +\theta^2\Xt^2\left( 4\py^2((\Qo^4)'\py w_2)
- 2\py^2((\Qo^4)''w_2) + 4\py((\Qo^4)'''w_2)
-(\Qo^4)''''w_2\right)
\end{align*}
so that
\begin{equation}\label{eq:3k}
\|\rho^{-1}\Xu\cU\MP[\Xt^2,\Qo^4]\cS^2\LP v_1\|
 \lesssim \vartheta^{-1}\theta^\frac12\left(\|\rho^2\py w_2\| +\|\rho^2 w_2\|\right).
 \end{equation}
By the Cauchy-Schwarz inequality, we have
\begin{multline*}
 \Bigl|\int(\Xi_{A,B} z_1)\Xu\cU\MP [\Xt^2,\Qo^4]\cS^2\LP v_1\Bigr|
 \lesssim\|\rho\Xi_{A,B} z_1\|
\|\rho^{-1}\Xu\cU\MP [\Xt^2,\Qo^4]\cS^2\LP v_1\|\\
 \lesssim B\vartheta^{-1}\theta^\frac12 (\|\rho\py z_1\|+\|\rho z_1\|)
\left(\|\rho^2\py w_2\| +\|\rho^2 w_2\|\right).
\end{multline*}
Therefore, summing up and
 recalling that $\vartheta=\theta^\frac14$,
\[
|\bK_3|\lesssim 
B\theta^\frac14 (\|\rho\py z_1\|+\|\rho z_1\|+\|\rho z_2\|)
\left(\|\rho^2\py w\| +\|\rho^2 w\|\right).
\]
Now, we use Lemma~\ref{LE:wz}
and we take $\theta$ small depending on $\oz$ and $B$,
\[
|\bK_3|\lesssim 
B \theta^\frac14\oz^{-\frac32} \bZ\lesssim 
 \theta^\frac18 \bZ.
\]
Similarly, using \eqref{eq:k3}, \eqref{eq:3k} and then Lemma~\ref{LE:wz},
one obtains for $\theta$ small
\[
|\bL_3| \lesssim
\theta^\frac14 (\|\rho z_1\|+\|\rho z_2\|)
\left(\|\rho^2\py w\| +\|\rho^2 w\|\right)\\
\lesssim 
\theta^\frac18 \bZ.
\]
Therefore, taking $\theta>0$ small enough
(depending on $\oz$ and $B$), using \eqref{eq:P2} and the 
above estimates on $\bK_2$, $\bL_2$, $\bK_3$ and $\bL_3$, it holds
\begin{equation}\label{eq:P3}
\oz^2 \bZ\lesssim 
\bK_1+\bK_2+\bK_3 + C\oz^2 (\bL_1+\bL_2+\bL_3)
+ \frac B{A\theta^5} \Bigl( \|\eta_A \py v\|^2+ \frac {B^2}{A^2}\|\eta_A v\|^2\Bigr) .
\end{equation}

\emph{Estimates of $\bK_4$ and $\bL_4$.}
Recall the decomposition (from Lemma \ref{LE:Ty} and the proof of Lemma \ref{LE:eb})
\begin{align*}
q_1 & = q_{1,1} + q_{1,2}, \quad q_2 = q_{2,1} + q_{2,2},\quad
q_{1,1} = b_1^2 G + b_2^2 H,\quad q_{2,1} = b_1b_2 G_2,\\
q_{1,2} & = (3\Qo+10\omega\Qo^3)(2b_1V_1v_1 + v_1^2) 
+(\Qo+2\omega\Qo^3)(2b_2V_2v_2+v_2^2) + N_1,\\
q_{2,2} & = 2(\Qo+2\omega\Qo^3)(2b_1V_1v_2+2b_2V_2v_1+v_1v_2) + N_2
\end{align*}
where 
$|N_1|+|N_2|\lesssim |u|^3 \lesssim|b|^3\rho^{24} +|v|^3$.
We set
\begin{align*}
n_{1,1} &= \cS^2\LP q_{2,1}^\perp,\quad
n_{1,2} = -\cS^2\LP p_2^\perp +\cS^2\LP q_{2,2}^\perp+\cS^2\LP r_2^\perp
+ \dot\omega\QP v_1 ,
\\
n_{2,1} &= \MM\cS^2 q_{1,1}^\top ,\quad
 n_{2,2} = -\MM\cS^2 p_1^\top +\MM\cS^2 q_{1,2}^\top +\MM\cS^2 r_1^\top
+\dot\omega\QM v_2 .
\end{align*}
By the expressions of $G$, $H$, $G_2$ and Lemma \ref{LE:tp}, it holds
$|q_{2,1}^\perp|+|q_{1,1}^\top|\lesssim |b|^2(\nu + \sqrt{\oz} \rho^8)$ for all $k\geq 0$, 
and so $|n_{1,1}^{(k)}|+|n_{2,1}^{(k)}|\lesssim \nu + \sqrt{\oz} \rho^8$
for all $k\geq 0$.
Using $|\Phi_B|\leq B$ and $|\Phi_B'|\leq 1$, the definition of $U$,
and Lemma \ref{LE:tc}, we get
\[
\|\rho^{-1}\Xi_{A,B}(\Xu\cU\Xt^2 n_{1,1})\|
\lesssim B \left( \|\rho^{-1} \py^2 n_{1,1}\|+\|\rho^{-1} \py n_{1,1}\|+\|\rho^{-1} n_{1,1}\|\right) \lesssim B |b|^2.
\]
Thus,
by the Cauchy-Schwarz inequality, we obtain
\[
\Bigl|\int(\Xi_{A,B} z_2)\Xu\cU\Xt^2 n_{1,1}\Bigr|
 \lesssim 
\|\rho z_2 \|\|\rho^{-1}\Xi_{A,B}(\Xu\cU\Xt^2 n_{1,1})\|\lesssim 
B |b|^2 \|\rho z_2 \|.
\]
Similarly,
\[
\Bigl|\int(\Xi_{A,B} z_1)\Xu\cU\MP\Xt^2 n_{2,1}\Bigr|
\lesssim \|\rho z_1\| \|\rho^{-1} \Xi_{A,B}\Xu \cU \MP\Xt^2 n_{2,1}\|
\lesssim B |b|^2 \|\rho z_1\|.
\]
We turn to the estimates concerning $n_{1,2}$ and $n_{2,2}$.
By the Cauchy-Schwarz inequality,
\[
\Bigl|\int(\Xi_{A,B} z_2)\Xu\cU\Xt^2 n_{1,2}\Bigr|
 \lesssim 
\|\eta_A^{-1} \cU^* \Xu(\Xi_{A,B} z_2)\|\|\eta_A\Xt^2 n_{1,2}\|.
\]
Using the expression of $U^*$, Lemma \ref{LE:tc}
and the definition of $\Xi_{A,B}$
(involving the function $\chi_A$, supported on $[-2A,2A]$)
\begin{align*}
\|\eta_A^{-1} \cU^* \Xu(\Xi_{A,B} z_2)\|
&\lesssim
\|\eta_A^{-1}\Xu\py (\Xi_{A,B} z_2)\|
+\|\eta_A^{-1}\Xu(\Xi_{A,B} z_2)\| \\
&\lesssim
\|\Xu(\eta_A^{-1}\py (\Xi_{A,B} z_2))\|
+\|\Xu(\eta_A^{-1} \Xi_{A,B} z_2)\|
 \lesssim B\vartheta^{-1} \|\eta_A z_2 \|. 
\end{align*}
Using Lemma~\ref{LE:ML} and Lemma \ref{LE:QQ}, we also have
\[
\|\eta_A \Xt^2 n_{1,2}\| 
\lesssim\theta^{-2}\|\eta_A p_2^\perp\|
+\theta^{-2}\|\eta_A q_{2,2}^\perp\| +\theta^{-2} \|\eta_A r_2^\perp\| 
+ |\dot\omega|\theta^{-1} (\|\eta_A \py v_1\|+\|\eta_A v_1\|).
\]
By Lemma \ref{LE:tp}
and $|y|\rho \lesssim 1/\oz\lesssim A$, $|y|\eta_A\lesssim A$, we get
the pointwise estimate
\begin{align*}
\eta_A |p_2^\perp|
&\lesssim \eta_A (|m_\gamma|+|m_\omega|) \big(|y\py u|+|u|
+\sqrt{\oz}\rho^8(\|\rho^4 y \py u\|+\|\rho^4 u\|) \big)\\
&\lesssim A (|m_\gamma|+|m_\omega|) (|\py u|+|u|)
\lesssim A (|m_\gamma|+|m_\omega|) (|b| \rho^8+|\py v|+|v|)
\end{align*}
Thus, using \eqref{eq:Xv}, $A\geq 1/\sqrt{\oz}$ and \ref{it:sv} of Lemma \ref{LE:sm},
\[
\|\eta_A p_2^\perp\|
\lesssim A (\|\nu v\|^2 + |b|^2)( |b|/\sqrt{\oz}+\| \py v\|+\| v\|)
\lesssim A^2 \varepsilon (\|\nu v\|^2 + |b|^2).
\]
Using $|q_{2,2}|\lesssim |v|^2 + |b| |v|\nu + |b|^3\rho^{24}$
we have by Lemma \ref{LE:tp},
\[
|q_{2,2}^\perp| \lesssim
|v|^2 + |b||v| \nu + |b|^3\rho^{24} + \rho^8 (\varepsilon \|\rho v\|+|b|^3 )
\lesssim \varepsilon (|v| + \|\rho v\| \rho^8 + |b|^2\rho^8).
\]
Thus,
\[
\|\eta_A q_{2,2}^\perp\| \lesssim (\varepsilon/ \sqrt{\oz}) ( \|\eta_A v\|+|b|^2)
\lesssim A \varepsilon ( \|\eta_A v\|+|b|^2).
\]
Moreover, using
$|r_2|\lesssim |m_\omega| |b| \rho^8$,
we have by Lemma \ref{LE:tp},
$|r_2^\perp|\lesssim |m_\omega| |b| \rho^8$,
and by \eqref{eq:Xv},
\[
\|\eta_A r_2^\perp\| \lesssim (1/\sqrt{\oz}) |m_\omega| |b|
\lesssim A\varepsilon(\|\nu v\|^2 + |b|^2).
\]
Gathering these estimates, we have proved
$\|\eta_A^2\Xt^2 n_{1,2}\| 
\lesssim A^2\theta^{-2}\varepsilon( \|\eta_A v\| + |b|^2)$.
Thus,
\[
\Bigl|\int(\Xi_{A,B} z_2)\Xu\cU\Xt^2 n_{1,2}\Bigr|
 \lesssim 
 A^2 B\theta^{-\frac 94}\varepsilon \|\eta_A z_2 \| 
 ( \|\eta_A v\| + |b|^2 ).
\]
Second, using $U= \py - \xi_W$, integration by parts
and the Cauchy-Schwarz inequality,
\begin{align*}
&\Bigl|\int(\Xi_{A,B} z_1)\Xu\cU\MP\Xt^2 n_{2,2}\Bigr|\\
&\quad \lesssim \Bigl|\int(\Xi_{A,B} z_1)\Xu\py \MP\Xt^2 n_{2,2}\Bigr|
+\Bigl|\int(\Xi_{A,B} z_1)\Xu \xi_W\MP\Xt^2 n_{2,2}\Bigr|\\
& \quad \lesssim \|\eta_A^{-1} \py(\Xi_{A,B} z_1)\| \|\eta_A\Xu\MP\Xt^2 n_{2,2}\|
+ \|\eta_A^{-1} ( \Xi_{A,B} z_1)\| \|\eta_A\Xu(\xi_W\MP\Xt^2 n_{2,2})\|.
\end{align*}
Arguing as for the previous term, using Lemma \ref{LE:UM}, we find
\begin{align*}
\|\eta_A^{-1} \py(\Xi_{A,B} z_1)\| + \|\eta_A^{-1} ( \Xi_{A,B} z_1)\|
\lesssim B (\|\eta_A \py^2z_1\|+ \|\eta_A \py z_1\|+ \|\eta_A z_1\|),\\
\|\eta_A\Xu\MP\Xt^2 n_{2,2}\|+
\|\eta_A\Xu(\xi_W\MP\Xt^2 n_{2,2})\|
\lesssim A^2\theta^{-2}\vartheta^{-1}\varepsilon( \|\eta_A v\| + |b|^2).
\end{align*}
Thus,
\[
\Bigl|\int(\Xi_{A,B} z_1)\Xu\cU\MP\Xt^2 n_{2,2}\Bigr|\\
\lesssim A^2 B\theta^{-\frac 94} ( \|\eta_A \py^2z_1\|
+ \|\eta_A \py z_1\|+ \|\eta_A z_1\|) 
\|\eta_A \Xt^2 n_{2,2}\|.
\]
In conclusion for the term $\bK_4$, we have obtained
\[
|\bK_4| \lesssim B|b|^2\|\rho z\|+
A^2 B\theta^{-\frac 94}\varepsilon ( \|\eta_A \py^2z_1\|+ \|\eta_A \py z_1\|+ \|\eta_A z_1\|
+\|\eta_A z_2\|) (\|\eta_A v\|+ |b|^2 ).
\]
Similarly, we check that
\[
|\bL_4|
 \lesssim |b|^2\|\rho z\|
+A^2 \theta^{-\frac 94}\varepsilon ( \|\eta_A \py^2z_1\|+ \|\eta_A \py z_1\|+ \|\eta_A z_1\|
+\|\eta_A z_2\|) (\|\eta_A v\|+ |b|^2 ).
\]

\emph{Estimates of $\bK_5$ and $\bL_5$.}
Using Lemma \ref{LE:PP}, we have
\begin{align*}
&\Bigl| \dot\omega \int(\Xi_{A,B} z_2)\Xu\PP w_2\Bigr|
 +\Bigl| \dot\omega \int(\Xi_{A,B} z_1)\Xu\PM w_1\Bigr|\\
& \ 
\lesssim |m_\omega| \|\eta_A^{-1}\Xu \Xi_{A,B} z_2\|\|\eta_A \PP w_2\|
+|m_\omega| \|\eta_A^{-1}\Xi_{A,B} z_1\|\|\eta_A \Xu\PM w_1 \|\\
& \ 
\lesssim B \vartheta^{-1} \big(\|\nu v\|^2 + |b|^2\big)
( \|\eta_A \py^2z_1\|+ \|\eta_A \py z_1\|+ \|\eta_A z_1\|+ \|\eta_A z_2\|) 
(\|\eta_A \py w\|+\|\eta_A w\|). 
\end{align*}
Using also Lemma \ref{LE:zw}, we obtain
\[
|\bK_5|\lesssim 
B\theta^{-\frac 94}\varepsilon (\|\nu v\|^2 + |b|^2 )
( \|\eta_A \py^2z_1\|+ \|\eta_A \py z_1\|+ \|\eta_A z_1\|+\|\eta_A z_2\|) .
\]
Similarly,
\[
|\bL_5|
\lesssim 
\theta^{-\frac 94}\varepsilon (\|\nu v\|^2 + |b|^2 )( \|\eta_A \py^2z_1\|+ \|\eta_A \py z_1\|+ \|\eta_A z_1\|
+\|\eta_A z_2\|) .
\]
Using Lemma \ref{LE:zw}, the estimates on $\bK_4$, $\bL_4$, $\bK_5$, $\bL_5$ 
imply
\[
|\bK_4|+|\bL_4|+|\bK_5|+|\bL_5|\lesssim 
B|b|^2\bZ^\frac 12+
A^2 B\theta^{-9}\varepsilon \big(\|\eta_A \py v\|+ \|\eta_A v\|\big)
(\|\eta_A v\|+ |b|^2).
\]
Inserting this in \eqref{eq:P3} and
taking $\varepsilon$ sufficiently small depending on $\theta$ and $A$, we get
\[
\oz^2 \bZ\lesssim \dot \bK +C \oz^2 \dot \bL + \frac B{A\theta^5} \Bigl( \|\eta_A \py v\|^2+ \frac {B^2}{A^2}\|\eta_A v\|^2\Bigr) 
 + B^2 \oz^{-2} |b|^4.
\]
For any $s\geq 0$, integrating this estimate on $[0,s]$, using \eqref{eq:eK} and \eqref{eq:eL}, we get
\[
\oz^2 \int_0^s \bZ
\lesssim\varepsilon +\frac {B }{A\theta^5}\int_0^s \Bigl(\|\eta_A\py v\|^2 +\frac {B^2}{A^2}\|\eta_A v\|^2\Bigr)
+B^2\omega_0^{-2} \int_0^s |b|^4.
\]
Using Lemma \ref{LE:v1} and then Lemma \ref{LE:eb}, we finally obtain
\[  \int_0^s \bZ
\lesssim 
 \frac{B^2}{\theta^5\omega_0^4}\varepsilon+
\frac {B^3}{A\theta^5\oz^6} \int_0^s \|\rho^4 v\|^2.
\]
We complete the proof by recalling the definition of $\bZ$ and
choosing constants as described in the remark following the statement 
of Lemma \ref{LE:vz}, in particular we take $A$ sufficiently
large (depending on all the other parameters except $\varepsilon$)
and then $\varepsilon$ sufficiently small.
\end{proof}

\section{Final estimates}
We complete the proof of Theorem \ref{TH:as}.
Using first Lemma~\ref{LE:cy} and then Lemma~\ref{LE:vz}, 
we obtain for all $s>0$,
\[
\oz^{ 3}\int_0^s\|\rho^4 v\|^2 
 \lesssim \int_0^s\left(\|\rho\py^2 z_1\|^2 + \|\rho\py z_1\|^2
+ \|\rho z_1\|^2
+\|\rho z_2\|^2\right) 
 \lesssim \sqrt{\varepsilon} +\frac {1}{\sqrt{A} }\int_0^s\|\rho^4 v\|^2 .
\]
Therefore, taking $A$ sufficiently large (depending on $\oz$),
then passing to the limit as $s\to +\infty$,
and taking $\varepsilon$ sufficiently small,
we have proved the key estimate
\begin{equation}\label{eq:ky}
\int_0^{+\infty}\|\rho^4 v \|^2 \lesssim 1.
\end{equation}
By Lemma~\ref{LE:eb} and then Lemma~\ref{LE:v1},
passing to the limit $s\to\infty$ it follows that
\begin{equation}\label{eq:im}
\int_0^{+\infty}\big(|b|^4+\|\rho\py v\|^2 +\|\rho v\|^2\big)\lesssim
\int_0^{+\infty}\big(|b|^4+\|\eta_A\py v\|^2 +\|\eta_A v\|^2\big)\lesssim A^2.
\end{equation}
In particular, there exists a sequence $s_n\to +\infty$ such that 
\[
\mathop{\lim}_{n\to+\infty} 
 |b(s_n)|^4+\|\rho\py v(s_n)\|^2 +\|\rho v(s_n)\|^2 =0
\]
Recall that setting $\cM = |b|^4+\|\rho v \|^2$,
 Lemma \ref{LE:31} states that
$|\dot\cM|\lesssim |b|^4 +\|\rho\py v\|^2 +\|\rho v\|^2$.
Let $s>0$. Integrating on $(s,s_n)$ for $n$ such that $s_n>s$, we obtain
\[
\cM(s)\leq\cM(s_n) +\int_s^{s_n} |\dot\cM|
\lesssim \cM(s_n) +\int_s^{s_n}\big(|b|^4+\|\rho\py v\|^2 +\|\rho v\|^2\big),
\]
and so $\cM(s)\lesssim \int_s^{+\infty} (|b|^4 +\|\rho\py v\|^2 +\|\rho v\|^2)$ by passing to the limit $n\to +\infty$. Thus, using \eqref{eq:im},
$\lim_{s\to +\infty}\cM(s)=0$.

Finally, by Lemma \ref{LE:Om} and \eqref{eq:im}, the function
$\ln \omega + \Omega $ has a finite limit 
as $s\to +\infty$. Since $\lim_{+\infty}|b|=0$,
we have $\lim_{+\infty} \Omega =0$,
and so $\ln \omega(s)$ has a finite limit as $s\to+\infty$.
Thus, there exists $\omega_+>0$, close to $\oz$ 
by \ref{it:pa} of Lemma \ref{LE:tm}, such that
$\lim_{+\infty} \omega = \omega_+$.
One obtains $\lim_{+\infty} \dot \gamma=1$ by \eqref{eq:Xv}, 
which implies $\lim_{t\to+\infty} d\gamma/dt = \omega_+$ by change of variable.

\subsection*{Acknowledgements}
The author is grateful to the anonymous referee for insightful comments.
He also thanks Guillaume Rialland (UVSQ, France) 
for a thorough check of the manuscript.

\end{document}